\documentclass{article}
\usepackage[utf8]{inputenc}
\pdfoutput=1
\usepackage{amsmath,amssymb,amsthm}

\usepackage{subfigure}
\usepackage{dsfont}
\usepackage{color}
\usepackage{xcolor}
\usepackage{fullpage}
\usepackage[
natbib,
maxcitenames=3, 
maxbibnames=10,  
sorting=nyt,
sortcites,
defernumbers,
backend=biber,
backref,
doi=false,
url=false,
giveninits
]{biblatex}
\usepackage{todonotes}
\setuptodonotes{inline}

\usepackage{tikz}
\usetikzlibrary{calc}

\usepackage{hyperref}
\hypersetup{colorlinks,
            breaklinks,
            linkcolor=blue,
            urlcolor=blue,
            anchorcolor=blue,
            citecolor=blue}
\usepackage[capitalize,nameinlink]{cleveref}
\crefname{assumption}{Assumption}{Assumptions}

\usepackage{hhline}

\newcommand{\R}{\mathbb{R}}
\newcommand{\N}{\mathbb{N}}
\newcommand{\eps}{\varepsilon}

\graphicspath{{./figures/}}

\DeclareMathOperator*{\argmin}{arg\,min}

\DeclareMathOperator{\dist}{dist}
\newcommand{\norm}[1]{\left\|#1\right\|}
\newcommand{\abs}[1]{\left|#1\right|}
\newcommand{\st}{\,:\,}
\newcommand{\de}{\,\mathrm{d}}
\newcommand{\one}{\mathds{1}}
\DeclareMathOperator{\supp}{supp}

\newtheorem{theorem}{Theorem}[section]
\newtheorem*{theorem*}{Theorem}
\newtheorem{proposition}[theorem]{Proposition}
\newtheorem{lemma}[theorem]{Lemma}
\newtheorem{corollary}[theorem]{Corollary}
\newtheorem*{corollary*}{Corollary}
\theoremstyle{definition}
\newtheorem{definition}[theorem]{Definition}
\newtheorem{assumption}{Assumption}
\newtheorem{example}{Example}
\theoremstyle{remark}
\newtheorem{remark}[theorem]{Remark}
\definecolor{pistachio}{rgb}{0.58, 0.77, 0.45}

\newcommand{\comment}[1]{}

\numberwithin{equation}{section}

\renewcommand{\vec}[1]{\mathbf{#1}}
\newcommand{\closure}[1]{\overline{#1}}
\newcommand{\relclosure}[1]{\overline{#1}^\mathrm{rel}}
\newcommand{\relpartial}{\partial^\mathrm{rel}}
\newcommand{\domain}{\Omega}
\newcommand{\constr}{\mathcal{O}}

\newcommand{\gscale}{h}
\newcommand{\nlscale}{\eps}
\newcommand{\res}{\delta}

\newcommand{\grad}{\nabla}

\renewcommand{\L}{\mathcal{L}}
\DeclareMathOperator{\Lip}{Lip}

\newcommand{\revision}[1]{{\color{magenta}#1}}
\renewcommand{\revision}[1]{{#1}}

\title{Uniform Convergence Rates for Lipschitz Learning on Graphs}
\author{
Leon Bungert
\thanks{Hausdorff Center for Mathematics, University of Bonn, Endenicher Allee 62, Villa Maria, 53115 Bonn, Germany. \href{mailto:leon.bungert@hcm.uni-bonn.de}{leon.bungert@hcm.uni-bonn.de}} 
\and 
Jeff Calder 
\thanks{School of Mathematics, University of Minnesota, 127 Vincent Hall, 206 Church St. S.E., Minneapolis, MN 55455, USA.
\href{mailto:jwcalder@umn.edu}{jwcalder@umn.edu}
}
\and 
Tim Roith
\thanks{Department of Mathematics, University of Erlangen–Nürnberg,
Cauerstraße 11, 91058 Erlangen, Germany. \href{mailto:tim.roith@fau.de}{tim.roith@fau.de}}
}
\date{\today}

\makeatletter
\let\blx@rerun@biber\relax
\makeatother

\begin{document}

\maketitle

\begin{abstract}
    Lipschitz learning is a graph-based semi-supervised learning method where one extends labels from a labeled to an unlabeled data set by solving the infinity Laplace equation on a weighted graph.
    In this work we prove uniform convergence rates for solutions of the graph infinity Laplace equation as the number of vertices grows to infinity. 
    Their continuum limits are absolutely minimizing Lipschitz extensions with respect to the geodesic metric of the domain where the graph vertices are sampled from.
    We work under very general assumptions on the graph weights, the set of labeled vertices, and the continuum domain.
    Our main contribution is that we obtain quantitative convergence rates even for very sparsely connected graphs, as they typically appear in applications like semi-supervised learning. In particular, our framework allows for graph bandwidths down to the connectivity radius. 
    For proving this we first show a quantitative convergence statement for graph distance functions to geodesic distance functions in the continuum. 
    Using the ``comparison with distance functions'' principle, we can pass these convergence statements to infinity harmonic functions and absolutely minimizing Lipschitz extensions.
\end{abstract}
{
  \hypersetup{linkcolor=black}
  \tableofcontents
}
\section{Introduction}

The last few years have seen a deluge of discrete to continuum convergence results for various problems in graph-based learning. 
This theory makes connections between discrete machine learning and continuum partial differential equations or variational problems, leading to new insights and better algorithms. 
Current works use either a Gamma-convergence framework \cite{roith2021continuum,Slep19,GarcSlep15,van2012gamma}, sometimes leading to convergence rates in energy norms (i.e., $L^2$), or PDE techniques like the maximum principle \cite{calder2020rates,yuan2020continuum,calder2019consistency,calder2018game,garcia2020maximum}, which give uniform convergence rates, albeit with more restrictive conditions on the graph bandwidth that often rule out the sparse graphs used in practice.
The different problems considered can be structured into consistency of spectral clustering---e.g., the convergence of eigenvalues and eigenvectors of graph Laplacian operators~\cite{Lux08,calder2020improved,trillos2020error,calder2020Lip}---and of semi-supervised learning problems, where one is interested in convergence of solutions to boundary value problems of graph Laplacian operators, e.g., \cite{roith2021continuum,Slep19,calder2019consistency,calder2018game,calder2020poisson,calder2020properly}.

In this work we focus on graph-based semi-supervised learning.
Given a large data set $\domain_n$ with $n\in\N$ elements, the general task consists in extending labels $g:\constr_n\to\R$ from a (typically much smaller) set of labeled data points $\constr_n\subset\domain_n$ to the rest of the data set.\footnote{Generally the labels are vector-valued, in $\R^k$ if there are $k$ classes, but this can be reduced to the scalar case with the \emph{one-versus-rest} approach to multi-class classification \cite{murphy2012machine}. } 
To include information about the unlabeled data into the problem, one builds a graph $G_n:=(\domain_n,w_n)$ where $w_n$ denotes a weight function.
The \emph{semi-supervised smoothness assumption} (see, e.g., \cite{van2020survey}) then demands that data points with high similarity, i.e., connected with an edge of large weight, carry similar labels.
To enforce this, the method of Laplacian learning \cite{zhu03} has been developed which requires solving the graph Laplace equation on $G_n$ with prescribed boundary values on $\constr_n$.
Since this method behaves poorly if the label set $\constr_n$ is small \cite{nadler2009semi}, the more general framework of $p$-Laplacian learning \cite{Ala16} was suggested, which mitigates this drawback for sufficiently large values of $p\in(1,\infty)$ \cite{flores2019algorithms}.
It consists in solving the following nonlinear graph $p$-Laplacian equation (see~\cite{elmoataz2015p} for details on the graph $p$-Laplacian)
\begin{align}\label{eq:p-LL}
    (\forall x\in\domain_n\revision{\setminus\constr_n})\;
    \sum_{y\in\domain_n}w_n(x,y)\abs{u(x)-u(y)}^{p-2}(u(x)-u(y)) = 0,\; \text{subject to $u=g$ on $\constr_n$}.
\end{align}
The underlying reason why $p$-Laplacian learning outperforms its special case $p=2$ is the Sobolev embedding $W^{1,p} \hookrightarrow C^{0,1-\frac{d}{p}}$, which implies that only for $p>d$, where $d$ is the ambient space dimension, it is meaningful to prescribe boundary values on the label set $\constr_n$ in the limit $n\to\infty$.
Since the dimension $d$ of the data can be quite large it seems canonical to consider the limit of $p$-Laplacian learning as $p\to\infty$.
The resulting method is called \emph{Lipschitz learning} \cite{kyng2015algorithms} and takes the form
\begin{align}\label{eq:LL}
    (\forall x\in\domain_n\revision{\setminus\constr_n})\;
    \min_{y\in\domain_n}w_n(x,y)(u(y)-u(x)) + 
    \max_{y\in\domain_n}w_n(x,y)(u(y)-u(x))
    = 0,\; \text{subject to $u=g$ on $\constr_n$}.
\end{align}
Note that solutions to this problem are \emph{absolutely minimizing Lipschitz extensions} of $g$ \cite{aronsson2004tour,juutinen2002absolutely}.  
This means that they extend the labels $g$ from $\constr_n$ to $\domain_n$ without increasing the (graph) Lipschitz constant and, in addition, are locally minimal in the sense that their Lipschitz constant cannot be reduced on any subset of $\domain_n$. The operator on the left hand side in \labelcref{eq:LL} is referred to as graph infinity Laplacian since it arises as \revision{the} limit of the graph $p$-Laplacian operators in \labelcref{eq:p-LL}. It has been established experimentally that Lipschitz learning performs very similarly to $p$-Laplacian learning for $d < p < \infty$ \cite{flores2019algorithms}, and so it has the added benefit over $p$-Laplace learning of having one less parameter, $p$, to tune. In addition, while Lipschitz learning \emph{forgets} the distribution of the labeled data \cite{Slep19,Ala16,calder2019consistency}, which is generally undesirable for semi-supervised learning, it is possible to reweight the graph to introduce a strong dependence on the data density \cite{calder2019consistency}.

The importance of Lipschitz-based algorithms in semi-supervised learning has already been observed in \cite{Lux04} and in \cite{kyng2015algorithms} algorithms for solving the Lipschitz learning problem have been proposed (see also \cite{Ober04,flores2019algorithms}).
In~\cite{calder2019consistency} a continuum limit for \labelcref{eq:LL} in the setting of geometric graphs was proved by one author of the present paper.
For $w_n(x,y):=\eta(\abs{x-y}/\gscale_n)$ with $\gscale_n>0$ being a graph length scale and $\eta:[0,\infty)\to[0,\infty)$ a sufficiently well-behaved kernel profile, graph vertices contained in the flat torus, i.e., $\domain_n\subset\mathbb{T}^d$, and a fixed label set $\constr_n=\constr$ it was shown in \cite{calder2019consistency} that as $n\to\infty$ solutions of \labelcref{eq:LL} uniformly converge to the viscosity solution of the infinity Laplacian equation
\begin{align}
    \Delta_\infty u = 0 \text{ on $\mathbb{T}^d\revision{\setminus\constr}$}\quad \text{subject to $u=g$ on $\constr$}.
\end{align}
Here, for a smooth function $u$ the infinity Laplacian is defined as
\begin{align}
    \Delta_\infty u := \langle \nabla u, D^2 u\nabla u\rangle = \sum_{i,j=1}^d \partial_i u\,\partial_j u\, \partial^2_{ij}u.
\end{align}
The main hypothesis for the proof in \cite{calder2019consistency} is that the graph length scale $\gscale_n$ is sufficiently large compared to the resolution $\res_n$ (see \cref{subsec:discr2cont} for the definition) of the graph; in particular, it is assumed that 
\begin{align}\label{eq:jeff_scaling}
    \lim_{n\to\infty} \frac{\res_n^2}{\gscale_n^3} = 0
\end{align}
holds. Furthermore, the viscosity solutions techniques require that the kernel function $\eta$ is sufficiently smooth.
In the related work \cite{roith2021continuum} by the two other authors of the present paper, continuum limits for general Lipschitz extensions on geometric graphs are proved using Gamma-convergence \revision{of the graph Lipschitz constant functional
\(
    u\mapsto
    \max_{x,y\in\domain_n}w_n(x,y)\abs{u(x)-u(y)}
\)
to the $L^\infty$-norm of the gradient
\(
    u\mapsto \norm{\nabla u}_{L^\infty(\domain)}
\)
in a suitable $L^\infty$-type topology.
Note that minimizers of these functionals, satisfying label constraints on $\domain_n$ and $\domain$, respectively, are not unique.
In particular, the Gamma-convergence result does not imply the convergence to \emph{absolute minimizers} of the continuum problem in a straightforward way.
Furthermore, establishing non-asymptotic convergence rates using Gamma-convergence, which by definition is of asymptotic flavor, is a difficult endeavour.}
Still, \revision{the approach in \cite{roith2021continuum}} opened the door to a much more general analysis of the Lipschitz learning problem.
In particular, in \cite{roith2021continuum} the weakest possible scaling assumption
\begin{align}\label{eq:weakest_scaling}
    \lim_{n\to\infty} \frac{\res_n}{\gscale_n} = 0,
\end{align}
which ensures that the graph $G_n$ is connected, was already sufficient to prove Gamma-convergence of the Lipschitz-constant functional.
Furthermore, the theory in \cite{roith2021continuum} applies to more general kernels $\eta$ and a large class of continuum domains.
The paper \cite{roith2021continuum} was also the first to work with general (non-fixed) label sets $\constr_n$ in a Lipschitz learning context (see \cite{calder2020rates} for a corresponding approach for $p$-Laplace learning \labelcref{eq:p-LL} with $p=2$).
Note that the weakest scaling assumption \labelcref{eq:weakest_scaling} has already appeared in \cite{gruyer2007absolutely} where discrete to continuum convergence of absolutely minimizing Lipschitz extensions on unweighted graphs was proved. 

When it comes to convergence rates of solutions to \labelcref{eq:LL} to a suitable continuum problem much less is known. 
To the best of our knowledge the only contribution in this direction is \cite{smart2010infinity}, where convergence rates are proved under very special conditions: the graph is unweighted and assumed to be a regular grid.
Furthermore, the author works with Dirichlet boundary conditions and requires the even stricter scaling assumption
\begin{align}\label{eq:charlie_scaling}
    \lim_{n\to\infty}\frac{\res_n}{\gscale_n^2} = 0.
\end{align}
The main motivational factors for the present paper are therefore the following:
\begin{itemize}
    \item We would like to prove convergence rates for \labelcref{eq:LL} which are valid down to the smallest scaling~\labelcref{eq:weakest_scaling}.
    \item The labeled set $\constr_n$ should be completely free, in particular, it should \emph{not} have to approximate the boundary of $\domain$, which is not realistic for semi-supervised learning.
    \item We would like to work under minimal assumptions on the weight function $\eta$ and the continuum domain $\domain\subset\R^d$ from which the graph vertices in $\domain_n$ are drawn.
\end{itemize}
Let us reiterate that, while the previous scaling conditions \labelcref{eq:jeff_scaling}, \labelcref{eq:weakest_scaling}, and \labelcref{eq:charlie_scaling} might look very similar, they in fact have strong impacts on the numerical complexity of the problem \labelcref{eq:LL}.
To give a concrete example, let $\domain_n$ coincide with a square grid
of the hypercube $[0,1]^d$.
In this case it holds that $\res_n\sim 1/n^{1/d}$ and suitable graph length scales
have the form $\gscale_n=1/n^\alpha$ where $\alpha\in(0,1/d)$ for the weakest scaling
\labelcref{eq:weakest_scaling}, $\alpha\in(0,2/(3d))$ for the scaling 
\labelcref{eq:jeff_scaling}, and $\alpha\in(0,1/(2d))$ for the largest 
scaling \labelcref{eq:charlie_scaling}.
Correspondly, the degree of a graph vertex, i.e., the number of \revision{neighbors}
which have to be considered when solving \labelcref{eq:LL}, scale like
$n\gscale_n^d= n^{1-\alpha d}$.
For $n=10,000$ vertices the size of these computational stencils would scale like 
$5$, $100$, and $500$ for the different scalings. These are dramatic differences in numerical complexity, and the differences grow larger as $n$ increases (e.g., at $n=10^6$ points there is a $1,000$-fold difference in sparsity between the weakest and strongest conditions). Furthermore, in applications of graph-based learning in practice, it is common to use very sparse graphs (typically $k$-nearest neighbor graphs) that operate slightly above the graph connectivity threshold \cite{von2007tutorial}.  
These observations justify the need for analysis and convergence rates that hold in the sparsest connectivity regime.

Let us briefly discuss where the restrictive length scale conditions \labelcref{eq:jeff_scaling} and \labelcref{eq:charlie_scaling} arise in discrete to continuum convergence, and how the techniques used in this paper overcome these issues. Both of these conditions stem from the use of pointwise consistency of the graph infinity Laplacian given in \labelcref{eq:LL} with the continuous infinity Laplacian $\Delta_\infty$, which is one of the main steps required for uniform convergence rates. Pointwise consistency results normally have two steps, one in which we pass to a nonlocal operator by controlling the randomness (or variance) in the graph construction, and a second step that uses Taylor expansion to pass from a nonlocal operator to a local PDE operator. For Lipschitz learning \labelcref{eq:LL} with geometric weights $w_n(x,y) = \eta(|x-y|/\gscale_n)$, the corresponding nonlocal equation (see \cite{chambolle2012holder} for analytical results on such a nonlocal infinity Laplacian equation) is obtained by replacing the discrete set $\Omega_n$ with the continuum domain $\Omega$, yielding
\begin{align}\label{eq:LL_nonlocal}
    \min_{y \in \domain} \eta\left( \frac{|x-y|}{\gscale_n}\right)(u(y)-u(x)) + 
    \max_{y\in \domain}\eta\left( \frac{|x-y|}{\gscale_n}\right)(u(y)-u(x)) = 0.
\end{align}
To control the difference between the discrete and nonlocal operators, we need to prove estimates of the form
\begin{equation}\label{eq:discretetononlocal}
\max_{y\in\domain_n}w_n(x,y)(u(y)-u(x)) = \max_{y\in \domain}\eta\left( \frac{|x-y|}{\gscale_n}\right)(u(y)-u(x)) + O(r_n),
\end{equation}
where $r_n$ depends on the regularity of $\eta$ and $u$. If $\eta$ and $u$ are Lipschitz continuous, then the function
\begin{equation}\label{eq:term}
y \mapsto \eta\left( \frac{|x-y|}{\gscale_n}\right)(u(y)-u(x))
\end{equation}
is Lipschitz continuous with Lipschitz constant independent of $\gscale_n>0$, provided $\eta$ has compact support so that we can use the bound $|u(y)-u(x)|\leq C\gscale_n$. It then follows that $r_n = \res_n$ is exactly the resolution of the point cloud $\domain_n$ (see \cref{subsec:discr2cont} for the definition). We can of course repeat the same argument for the minimum term in \labelcref{eq:LL_nonlocal}. When Taylor expanding $u(y)-u(x)$ to obtain consistency to $\Delta_\infty$, the second order terms are of size $O(\gscale_n^2)$, so we have to divide both side of the nonlocal equation \labelcref{eq:LL_nonlocal} by $\gscale_n^2$ in order to obtain consistency with the infinity Laplace equation $\Delta_\infty u= 0$, leading to pointwise consistency truncation errors of the form $O\left( \delta_n\gscale_n^{-2}\right)$. 
The condition that this term converges to zero is exactly the strongest length scale condition \labelcref{eq:charlie_scaling}.  When $\eta$ and $u$ are $C^2$, we can improve the error in \labelcref{eq:discretetononlocal} to be $r_n = \res_n^2\gscale_n^{-1}$, which follows from examining the Hessian of the mapping in \labelcref{eq:term} (see \cite[Lemma 4.1]{calder2019consistency}). This leads to the pointwise consistency error term $O(\res_n^2\gscale_n^{-3})$, which corresponds to the intermediate condition \labelcref{eq:jeff_scaling}. This condition was suitable in \cite{calder2019consistency} since the viscosity solution framework only requires pointwise consistency for smooth functions $u$, though no convergence rate can be established in this way. In addition to the restrictive length scale conditions on $h_n$ required for pointwise consistency, the kernel $\eta$ is also required to be Lipschitz for \labelcref{eq:charlie_scaling} and $C^2$ for \labelcref{eq:jeff_scaling}, and in either case, the restrictive condition that $t \mapsto t\eta(t)$ has a unique maximum near which it must be strongly concave, is needed to pass from nonlocal to local (see \cite[Section 2.1]{flores2019algorithms}). 

In the present approach we overcome these limitations by introducing a nonlocal infinity Laplacian on a homogenized length scale $\nlscale>\gscale_n$.
The homogenized operator has the form
\begin{align}
    \Delta_\infty^\nlscale u(x) = \frac{1}{\nlscale^2}\left(\inf_{y\in B_\nlscale(x)}(u(y)-u(x))
    +
    \sup_{y\in B_\nlscale(x)}(u(y)-u(x))\right)
\end{align}
and notably the effect of the kernel function $\eta$ has been ``homogenized out''.
Although sometimes hidden, this operator frequently appears in the analysis of equations involving the infinity Laplacian (cf.~\cite{armstrong2010easy,aronsson2004tour,manfredi2012definition,peres2009tug,peres2008tug,lewicka2017obstacle,mazon2012best}).
Most prominently, in \cite{armstrong2010easy} it was used as intermediate step in a very simple and elementary proof of the maximum principle for the infinity Laplace equation.
There, the key ingredient is the astonishing property that whenever $-\Delta_\infty u\leq 0$, the function $u^\nlscale(x):=\sup_{B_\eps(x)}u$ satisfies $-\Delta_\infty^\nlscale u^\nlscale \leq 0$ for any $\nlscale>0$.
For proving this one only utilizes that $u$ satisfies the celebrated ``comparison with cones'' property \cite{aronsson2004tour} which in our more general setting is rather a comparison with distance functions \cite{champion2007principles}.

A key step in our approach is an approximate version of this statement. We show that a graph infinity harmonic function can be extended to a continuum function $u_n^\nlscale$ that satisfies
\begin{align*}
    -\Delta_\infty^\nlscale u_n^\nlscale(x) \leq O\left(\frac{\res_n}{\gscale_n\nlscale} + \frac{\gscale_n}{\nlscale^2}\right).
\end{align*}
Compared to the truncation error of $O(\res_n\gscale_n^{-2})$ from \cite{smart2010infinity} we have improved by replacing one $\gscale_n$ by the (asymptotically larger) nonlocal length scale $\nlscale$.
For this we pay the price of having an additional $\gscale_n\nlscale^{-2}$ term which, however, does not contribute to the error for the small length scales $\gscale_n$ that we are interested in.
The convergence rates then follow by an application of the maximum principle for the operator $\Delta_\infty^\nlscale$ in combination with some (approximate) Lipschitz estimates to go back from $u_n^\nlscale$ to the graph solution.

The main caveat in proving our results is that we are dealing with three different metrics that have to be put into relation with each other.
First, there is the Euclidean metric of the ambient space which enters through the graph weights and quantifies the approximation of the continuum domain by the graph.
Second, there is the geodesic metric which is inherited from the Euclidean one and constitutes the natural metric in the (possibly non-convex) domain $\domain$.
Finally, there is also the graph metric, which describes shortest paths in the graph and is inherently connected to the Lipschitz learning problem \labelcref{eq:LL}.
For proving our continuum limit we need that all three metrics behave similarly on locally scales.
While it is straightforward to show that the graph metric approximates the geodesic metric, we have to assume that the latter also approximates the Euclidean metric locally. 
This we ensure by posing a very mild regularity condition on the continuum domain $\domain$ which prevents internal corners, where the geodesic metric has a constant discrepancy to the Euclidean one.
In particular, this allows us to work with arbitrary (non-smooth) convex domains and also with non-convex domains whose boundaries are smooth in non-convex regions.

The rest of this paper is organized as follows:
\cref{sec:main_results} summarizes the paper by stating all assumptions, the precise settings for the discrete and continuum equations and the passage between these two worlds, and our main results.
Subsequently, we investigate both the discrete and continuum problems in much more detail.
\cref{sec:discrete_problem} introduces weighted graphs, proves that infinity harmonic functions on a graph satisfy comparison with graph distance functions, and are absolutely minimizing Lipschitz extensions.
In \cref{sec:continuum_problem} we study the continuum limit, namely absolutely minimizing Lipschitz extensions (AMLEs).
Most importantly, in \cref{sec:max_ball_perturb} we show that AMLEs give rise to sub- and supersolutions of a nonlocal infinity Laplace equation and collect several perturbation results for this equation which were proved in \cite{smart2010infinity}.
The main part of the paper is \cref{sec:convergence} where we first prove convergence rates of graph distance function to geodesic ones (\cref{sec:cones}) and then use these results to establish convergence rates for graph infinity harmonic functions (\cref{sec:cvgc_inf_harm}).
Finally, in \cref{sec:numerics} we perform some numerical experiments where we compute empirical rates of convergence of AMLEs on a non-convex domain.

\section{Setting, Assumptions, and Main Results}
\label{sec:main_results}

To simplify the distinction between discrete and continuum, we will from now on denote general points in the continuum domain $\closure\domain$ in regular font $x,y,z$, etc., while we will denote points specifically belonging to the point cloud $\domain_n$ with bold font $\vec x,\vec y,\vec z$, etc. 
We will use the same notation for functions on $\closure\domain$, denoted, for example, as $u,v,w:\closure\domain \to \R$, while we will again use bold notation and an index $n$ for functions defined solely on the point cloud $\domain_n$, i.e., $\vec u_n,\vec v_n,\vec w_n:\domain_n\to \R$.

\subsection{The Discrete Problem}

We let $\domain_n$ be a point cloud and $\constr_n\subset\domain_n$ be a set of labelled vertices.
We construct a graph with vertices $\domain_n$ in the following way:
Let $\gscale>0$ be a graph length scale, let  $\eta:[0,\infty)\to [0,\infty)$ be a function satisfying \cref{ass:kernel} below, and define the rescaled kernel as $\eta_\gscale(t) = \frac{1}{\sigma_\eta\gscale}\eta\left( \frac{t}{\gscale}\right)$, where $\sigma_\eta=\sup_{t\geq 0}t\eta(t)$.
\begin{assumption}\label{ass:kernel}
The function \revision{$\eta:(0,\infty)\to [0,\infty)$}
is nonincreasing with $\supp\eta\subset[0,1]$.
Furthermore, $\sigma_\eta:=\revision{\sup_{t> 0}}t\eta(t)\in(0,\infty)$ and we assume that there exists $t_0\in(0,1]$ \revision{(chosen as the largest number with this property)} with $\sigma_\eta=t_0\eta(t_0)$.
\end{assumption}
\begin{example}
A couple of popular kernel functions $\eta$ which satisfy \cref{ass:kernel} are given below:
\begin{itemize}
    \item (constant weights) $\eta(t)=1_{[0,1]}(t)$,
    \item (exponential weights) $\eta(t)=\exp(-t^2/(2\sigma^2))1_{[0,1]}(t)$,
    \item (non-integrable weights) $\eta(t)=\frac{1}{t^p}1_{[0,1]}(t)$ with $p\in(0,1]$.
\end{itemize}
Note that the assumption that $t_0>0$ achieves the \revision{supremum} is not redundant.
Indeed, for the kernel
\begin{align*}
    \eta(t) = \frac{1}{t^{\frac{1}{t-1}+2}}1_{[0,1]}(t),\quad\revision{t>0,}
\end{align*}
the \revision{supremum} of $t\mapsto t\eta(t)$ is achieved at $t_0=0$ although $\sigma_\eta=1$.
\end{example}
Using the kernel function, the edge connecting two vertices $\vec x,\vec y\in\domain_n$ carries the weight $\eta_\gscale(\abs{\vec x-\vec y})$ \revision{whenever $\vec x\neq\vec y$ and zero otherwise}.
\revision{Consequently, in the whole paper we use the convention that $\eta(0)\cdot 0 := 0$ even though $\eta$ is not defined in $0$.
}
\revision{Of course, the weight between two vertices $\vec x,\vec y\in\domain_n$} is \revision{also} zero whenever $\abs{\vec x-\vec y}>\gscale$, in which case we regard the vertices as not connected.
The number $\gscale>0$ can be interpreted as characteristic connectivity length scale of the graph.

The main objective of this paper is to prove convergence rates for solutions to the graph infinity Laplacian equation to solutions of the corresponding continuum problem.
To state the discrete problem we let $\vec g_n:\constr_n\to\R$ be a labelling function and $\vec u_n:\domain_n\to\R$ a solution of the graph infinity Laplacian equation
\begin{align}\label{eq:discrete_problem}
\begin{cases}
\min_{\vec y\in\domain_n} \eta_h(\abs{\vec x-\vec y})\left(\vec u_n(\vec y) - \vec u_n(\vec x)\right) +  \max_{\vec y\in\domain_n} \eta_h(\abs{\vec x-\vec y})\left(\vec u_n(\vec y) - \vec u_n(\vec x)\right) = 0,\; &\vec x\in \domain_n\setminus\constr_n,\\
\vec u_n(\vec x) = \vec g_n(\vec x), \; &\vec x\in \constr_n.
\end{cases}
\end{align}
The solution of this problem (see \cite{calder2019consistency} for existence and uniqueness) is characterized by the fact that it is an absolutely minimizing Lipschitz extension of the labelling function $\vec g_n$ in the sense that
\begin{align}
    \Lip_n(\vec u_n;A) = \Lip_n(\vec u_n;\partial A),\quad \forall A\subset\domain_n.
\end{align}
Here $\Lip_n(\vec u_n;A)$ is a discrete Lipschitz constant of the graph function $\vec u_n$ on a subset of vertices $A\subset\domain_n$ with graph boundary $\partial A$.
We refer to \cref{sec:discrete_problem} for precise definitions and a proof of this statement in an abstract graph setting.
In particular, not only do we have $\Lip_n(\vec u_n;\domain_n)=\Lip_n(\vec g_n;\constr_n)$, which would be satisfied by any Lipschitz extension, but the Lipschitz constant of $\vec u_n$ is also locally minimal.

\subsection{Discrete to Continuum}
\label{subsec:discr2cont}

To set the scene for a discrete to continuum study, we \revision{let $\domain\subset\R^d$ be an open and bounded domain and} assume that the point cloud $\domain_n$ is sampled from $\closure{\domain}$ and that the constraint set $\constr_n$ approximates a closed set $\constr\subset\closure{\domain}$.
Of course, we need to quantify how well $\domain_n$ and $\constr_n$ approximate their continuum counterparts $\domain$ and $\constr$. 
For this we introduce the Hausdorff distance of two sets $A,B\subset\R^d$:
\begin{align}
    d_H(A,B):=\sup_{x\in A}\inf_{y\in B}\abs{x-y}\vee \sup_{x\in B}\inf_{y\in A}\abs{x-y}.
\end{align}
Note that since $\domain_n\subset\closure\domain$, its Hausdorff distance simplifies to
\begin{align}
    d_H(\domain_n,\closure\domain) = \sup_{x\in\closure\domain}\inf_{\vec y\in\domain_n}\abs{x-\vec y}.
\end{align}
To control both approximations at the same time we define the resolution of the graph as
\begin{equation}\label{eq:graph_res}
    \res_n := d_H(\domain_n,\closure\domain) \vee d_H(\constr_n,\constr),
\end{equation}
\revision{where} we recall that $a\vee b = \max\{a,b\}$ and $a\wedge b = \min\{a,b\}$. For convenience we introduce a closest point projection $\pi_n:\closure\domain\to\domain_n$, meaning that
\begin{align}\label{eq:CCP}
    \pi_n(x) \in \argmin_{\vec x \in \domain_n}\abs{x-\vec x},\quad x\in\closure\domain.
\end{align}
Note that the closest point projection has the important property that
\begin{align}
    \abs{\pi_n(x) - x} \leq \res_n,\quad\forall x\in\closure\domain,
\end{align}
where $\res_n$ is the graph resolution \labelcref{eq:graph_res}.
For a function $\vec{u}:\domain_n\rightarrow\R$ we define its piecewise constant extension to $\closure\domain$ as
\begin{align}\label{eq:extension}
    u_n : \closure\domain \to \R,\quad u_n := \vec u \circ \pi_n.
\end{align}
Note that $u_n$ is a simple function, constant on each Voron\"i cell \cite{voronoi1908nouvellesa,voronoi1908nouvellesb} of $\domain_n$.

For treating the labelling sets $\constr_n$ and $\constr$ we use a projection $\pi_{\constr_n}:\constr\to\constr_n$, satisfying
\begin{align}
    \pi_{\constr_n}(z) \in \argmin_{\vec z\in\constr_n}\abs{z-\vec z},\quad z\in\constr.
\end{align}
By the definition of the graph resolution \labelcref{eq:graph_res} it holds
\begin{align}\label{eq:graph_res_closest_point}
    \abs{\pi_{\constr_n}(z) - z} \leq \res_n,\quad\forall z\in\constr.
\end{align}

\subsection{The Continuum Problem}
\label{subsec:continuum_problem}

Before we formulate the continuum problem, we need to discuss some properties of the continuum domain $\closure\domain$ and the label set $\constr$.
Unlike in classical PDE theory we do not prescribe boundary values on the topological boundary $\partial\domain$ but instead on the closed subset $\constr\subset\closure\domain$ which is equipped with a continuous labeling function $g:\constr\to\R$.

Our motivation for this comes from applications like image classification where, e.g., $\domain=(0,1)^d$ could model the space of grayscale images with $d$ pixels and $\constr$ is a subset of labeled images. 
In this case it is not meaningful to assume that $\constr=\partial\domain$ but instead $\constr$ could even be a discrete subset of $\closure\domain$. For notational convenience we will for the rest of the paper use the notation
\begin{align}
    \domain_\constr &:= \closure\domain\setminus\constr.
\end{align}
We define the geodesic distance $d_\domain(x,y)$ by
\begin{align}
    d_\domain(x,y) = \inf\left\{ \int_0^1 |\dot{\xi}(t)|\, dt \, : \, \xi\in C^1([0,1];\closure \domain) \text{ with } \xi(0)=x \text{ and } \xi(1)=y\right\},\quad x,y\in\closure\domain.
\end{align}
\revision{By definition it holds $d_\domain(x,y)\geq\abs{x-y}$ and} if the line segment from $x$ to $y$ is contained in $\closure \domain$, which, in particular, is the case for convex $\closure\domain$, then $d_\domain(x,y)=|x-y|$.
We pose the following assumption on the domain which relates the intrinsic geodesic distance with the extrinsic Euclidean one.
The assumption is satisfied, e.g., for convex sets or sets with mildly smooth boundaries (see \cref{sec:convergence}), \revision{and essentially demands that $\domain$ has no inward-pointing cusps.}
\begin{assumption}\label{ass:geodesic_euclidean_2}
There exists a function $\phi:[0,\infty)\to[0,\infty)$ with $\lim_{h\downarrow 0}\frac{\phi(h)}{h} = 0$ and $r_\domain>0$ such that 
\begin{equation}\label{eq:geodesic_euclidean}
d_\domain(x,y) \leq |x-y| + \phi(\abs{x-y}), \quad \text{for all }x,y\in\domain\st |x-y|\leq r_\domain.
\end{equation}
\revision{Furthermore, for $\gscale>0$ we define $\sigma_\phi(\gscale) = \sup_{0 < s \leq \gscale}\frac{\phi(s)}{s}$.}
\end{assumption}
Using the geodesic distance function we can also define the distance between a point and a set and, more generally, \revision{between} two sets as
\begin{align}
    \dist_\domain(A,B) := \inf_{\substack{x\in A\\y\in B}}d_\domain(x,y)
\end{align}
and we abbreviate $\dist_\domain(x,A):=\dist_\domain(\{x\},A)$ for sets $A,B\subset\closure\domain$.
The corresponding distance functions with respect to the Euclidean metric are denoted by $\dist(\cdot,\cdot)$.

In what follows we forget that $\closure\domain$ is a subset of $\R^d$ and simply regard $(\closure\domain,d_\domain)$ as a metric space, or more specifically, a \emph{length space}.
In particular, all topological notions like interior, closure, boundary, etc., will be understood with respect to the relative (intrinsic) topology of this length space.
For example, $\domain_\constr$ is relatively open, its relative boundary $\relpartial\domain_\constr$ coincides with $\constr$, and its relative closure is $\relclosure\domain_\constr=\closure\domain$.
Note that for a subset $A\subset\closure\domain$ it holds $\relclosure A=\closure A$, which is why we do not distinguish those two different closures anymore.
When it is clear from the context we also omit the word ``relative''.
All closures and \revision{boundaries} with respect to $\domain$ as an open subset of $\R^d$ will be denoted by the usual symbols.

The continuum problem is to find an absolutely minimizing Lipschitz extension (AMLE) of the function $g:\constr\to\R$ to the whole domain $\closure\domain$.
For this, we define the Lipschitz constant of a function $u:\closure\domain\to\R$ on a subset $A\subset\closure\domain$ as
\begin{align}
    \Lip_\domain(u;A) := \sup_{x,y\in A}\frac{\abs{u(x)-u(y)}}{d_\domain(x,y)},
\end{align}
and we abbreviate $\Lip_\domain(u):=\Lip_\domain(u;\closure\domain)$.
A function $u\in C(\closure\domain)$ is an absolutely minimizing Lipschitz extension of $g$ if it holds that
\begin{align}\label{eq:AMLE}
    \begin{cases}
    \Lip_\domain(u;\closure V)
    = \Lip_\domain(u;\relpartial V)\quad&\text{for all relatively open and connected subsets $V\subset\domain_\constr$},\\
    u = g \quad&\text{on }\constr.
    \end{cases}
\end{align}

\subsection{Main Results}

We now state our general uniform convergence result for solutions $\vec u_n:\domain_n\to\R$ of the graph infinity Laplacian equation \labelcref{eq:discrete_problem} to an AMLE $u:\closure\domain\to\R$ satisfying \labelcref{eq:AMLE}.
Obviously, we also have to quantify in which sense the labelling functions $\vec g_n:\constr_n\to\R$ converge to $g:\constr\to\R$. 
In the simplest case $\constr_n=\constr$ and $\vec g_n=g$ for all $n\in\N$.
However, in general $\constr_n\neq\constr$ and we can work under the following much more general assumption.
\begin{assumption}\label{ass:data}
The labelling functions $\vec g_n:\constr_n\to\R$ and $g:\constr\to\R$ satisfy:
\begin{itemize}
    \item $\sup_{n\in\N}\Lip_n(\vec g_n;\constr_n)<\infty$ and $\Lip_\domain(g;\constr)<\infty$.
    \item There exists $C>0$ such that for all $z\in\constr$ it holds that $\abs{\vec g_n(\pi_{\constr_n}(z))-g(z)}\leq C\res_n$.
\end{itemize}
\end{assumption}
\begin{example}
In this example we illustrate some important special cases where this assumption is satisfied.
\begin{itemize}
    \item If $\vec g_n=g$ and $\constr_n=\constr$ for all $n\in\N$ and $g:\constr\to\R$ is Lipschitz with respect to the Euclidean distance, meaning $\Lip(g;\constr):=\sup_{x,y\in\constr}\frac{\abs{g(x)-g(y)}}{\abs{x-y}}<\infty$, then \cref{ass:data} is fulfilled.
    This is the setting of~\cite{calder2019consistency}.
    \item If $\constr_n=\constr$ for all $n\in\N$ then $\pi_{\constr_n}(z)=z$ for all $z\in\N$ and hence \cref{ass:data} reduces to sufficiently fast uniform convergence of $\vec g_n$ to $g$ on $\constr$ if $g:\constr\to\R$ is Lipschitz.
    \item If $\vec g_n=g$ where $g:\closure\domain\to\R$ is defined on the whole of $\closure\domain$ and is Lipschitz, then \cref{ass:data} is satisfied by \revision{properties}
    of the graph resolution $\res_n$\revision{, see \labelcref{eq:graph_res_closest_point}}.
    This is the setting of \cite{roith2021continuum}.
\end{itemize}

\end{example}

Our main result is non-asymptotic and depends on a free parameter $\nlscale>0$.
Since this parameter arises from a nonlocal homogenized problem (see \cref{sec:max_ball_perturb} for details), we refer to it as the nonlocal length scale.
The relative scaling between the graph resolution $\res_n$, the graph length scale $\gscale$, and the nonlocal length scale $\nlscale$ is of utmost importance for our convergence statement.
\revision{Later we shall optimize over $\nlscale$ to obtain convergence rates depending only on $\res_n$ and $\gscale$.}
\begin{assumption}\label{ass:scaling}
The parameters $\res_n>0$, $\gscale>0$, and $\nlscale>0$ satisfy
\begin{subequations}
\begin{align}
    \revision{\gscale} &\leq \revision{r_\domain,} \\
    \frac{\gscale}{\nlscale} &< \frac{1}{2}, \\
    \sigma_\phi(\gscale) + \frac{\res_n}{\gscale} 
    &\leq
    \revision{\frac{t_0}{4+2\sigma_\phi(\gscale)}}\left(1-2\frac{\gscale}{\nlscale}\right).
\end{align}
\end{subequations}
\revision{Here $t_0>0$, $r_\domain>0$, and $\sigma_\phi(\gscale)$ are defined in \cref{ass:kernel,ass:geodesic_euclidean_2}, respectively.}
\end{assumption}
\revision{%
\begin{remark}
By \cref{ass:geodesic_euclidean_2} $\sigma_\phi(\gscale)\to 0$ as $\gscale\to 0$.
Hence, for very small graph bandwidths the scaling assumption on general domains approaches the one on convex domains where $\phi=0$.
\end{remark}}
For a concise presentation we will in the following use the notation $a\lesssim b$ which is very common in the PDE community and means $a\leq C\,b$ for a universal constant $C>0$, possibly depending on the kernel $\eta$ or the label function $g$.
Furthermore, the notation $a_n\ll b_n$ means $a_n/b_n\to 0$ as $n\to\infty$.
\begin{theorem}[General Convergence Result]\label{thm:general_convergence_result}
Let $\nlscale>0$, let \cref{ass:kernel,ass:geodesic_euclidean_2,ass:data,ass:scaling} hold, let $\vec u_n:\domain_n\to\R$ solve \labelcref{eq:discrete_problem}, and $u:\closure\domain\to\R$ solve \labelcref{eq:AMLE}.
\begin{enumerate}
    \item It holds 
    \begin{align*}
        \sup_{\domain_n}\abs{u-\vec u_n} = \sup_{\closure\domain}\abs{u- u_n} \lesssim \nlscale +  \sqrt[3]{\frac{\res_n}{\gscale\nlscale} + \frac{\gscale}{\nlscale^2}+\frac{\phi(\gscale)}{\gscale\nlscale}}.
    \end{align*}
    \item If $\inf_{\closure\domain}\abs{\grad u}>0$, then it even holds 
    \begin{align*}
        \sup_{\domain_n}\abs{u-\vec u_n} = \sup_{\closure\domain}\abs{u- u_n} \lesssim \nlscale +  {\frac{\res_n}{\gscale\nlscale} + \frac{\gscale}{\nlscale^2}+\frac{\phi(\gscale)}{\gscale\nlscale}}.
    \end{align*}
Here, $u_n:\closure\domain\to\R$ denotes the piecewise constant extension of $\vec u_n$, defined in \labelcref{eq:extension}.
\end{enumerate}

\end{theorem}
The main advantage of our results is that they allow one to obtain uniform convergence rates for the weakest possible scaling assumption on the graph length scale $\gscale$ with respect to the graph resolution $\res_n$.
Indeed, $\gscale=\gscale_n$ satisfying $\frac{\res_n}{\gscale_n}\to 0$ and $\gscale_n\to 0$ is the smallest possible graph length scale such that the graph is asymptotically connected.
For such scaling we recover known convergence results from \cite{gruyer2007absolutely,roith2021continuum} and improve the result from \cite{calder2019consistency}.
\begin{corollary}[Convergence under Weakest Scaling Assumption]
For $\gscale=\gscale_n$ and $\res_n\ll\gscale_n\ll\nlscale$ it holds that $u_n\to u$ uniformly on $\closure\domain$ as $n\to\infty$.
\end{corollary}
\begin{proof}
Using the scaling assumption $\res_n\ll\gscale_n\ll\nlscale$ \revision{and keeping $\nlscale>0$ fixed, the entire root in \cref{thm:general_convergence_result} converges to zero and we obtain} that
\begin{align*}
    \lim_{n\to\infty}\sup_{\closure\domain}\abs{u-u_n} \lesssim \nlscale.
\end{align*}
Since $\nlscale>0$ was arbitrary, this shows the assertion.
\end{proof}
For every graph scaling $\gscale_n$ which is asymptotically larger than the weakest one, our general convergence theorem allows us to define a nonlocal scaling $\nlscale=\nlscale_n$ for which a uniform convergence rate holds true.
\revision{By optimizing over $\nlscale_n$ we obtain convergence rates in different regimes, which we discuss in the following.}
\revision{
\begin{corollary}[Small Length Scale Regime]\label{cor:small_scale}
We have the following convergence rates:
\begin{enumerate}
    \item For \revision{$\res_n \lesssim \gscale \lesssim \res_n^{\frac{5}{9}}$} it holds
    \begin{align*}
        \sup_{\domain_n}\abs{u-\vec u_n}
        =
        \sup_{\closure\domain}\abs{u-u_n} 
        \lesssim\left(\frac{\res_n + \phi(\gscale)}{\gscale}\right)^\frac{1}{4}.
    \end{align*}
    \item If $\inf_{\closure\domain}\abs{\grad u}>0$, then for \revision{$\res_n \lesssim \gscale \lesssim \res_n^{\frac{3}{5}}$} it holds
    \begin{align*}
        \sup_{\domain_n}\abs{u-\vec u_n} 
        =
        \sup_{\closure\domain}\abs{u-u_n} 
        \lesssim \left(\frac{\res_n + \phi(\gscale)}{\gscale}\right)^\frac{1}{2}.
    \end{align*}
\end{enumerate}
\end{corollary}
\begin{proof}
For 1, we choose 
\[\nlscale = \left( \frac{\res_n + \phi(\gscale)}{\gscale}\right)^{\frac{1}{4}}.\]
The condition $\gscale \lesssim \res_n^{\frac{5}{9}}$ implies that
\[\frac{\gscale}{\nlscale^2}\lesssim \frac{\res_n + \phi(\gscale)}{\gscale \nlscale},\]
and so 1 follows from Theorem \ref{thm:general_convergence_result}.

The proof for 2 is similar, but instead we choose
\[\nlscale = \left( \frac{\res_n + \phi(\gscale)}{\gscale}\right)^{\frac{1}{2}}.\qedhere\]
\end{proof}}

\revision{For larger length scales we make a case distinction based on the regularity of the boundary, in other words, the decay of the function $\phi$ in \cref{ass:geodesic_euclidean_2}.}

\begin{corollary}[Large Length Scale Regime for Smooth Boundaries]\label{cor:large_scale_smooth}
Let $\phi(\gscale)\lesssim \gscale^{\revision{\frac{9}{5}}}$ (e.g., if $\domain$ is convex or $\partial\domain$ is at least $C^{1,\revision{\frac{4}{5}}}$).
Then we have the following convergence rates:
\begin{enumerate}
    \item For $\gscale\gtrsim\res_n^\frac{5}{9}$ it holds
    \begin{align*}
        \sup_{\domain_n}\abs{u-\vec u_n}
        =
        \sup_{\closure\domain}\abs{u-u_n} 
        \lesssim\gscale^\frac{1}{5}.
    \end{align*}
    \item If $\inf_{\closure\domain}\abs{\grad u}>0$, then for $\gscale \gtrsim\res_n^\frac{3}{5}$ it holds
    \begin{align*}
        \sup_{\domain_n}\abs{u-\vec u_n} 
        =
        \sup_{\closure\domain}\abs{u-u_n} 
        \lesssim \gscale^\frac{1}{3}.
    \end{align*}
\end{enumerate}
\end{corollary}
\begin{proof}
\revision{%
Since $\phi(\gscale)\lesssim\gscale^{\revision{\frac{9}{5}}}$, we get
\begin{align*}
    \frac{\phi(\gscale)}{\gscale\nlscale}\lesssim\frac{\gscale^\frac{4}{5}}{\nlscale} = \frac{\nlscale}{\gscale^\frac{1}{5}}\frac{\gscale}{\nlscale^2}.
\end{align*}
For both choices $\nlscale:=\gscale^\frac{1}{5}$ (respectively, $\nlscale:=\gscale^\frac{1}{3}$) it holds $\frac{\phi(\gscale)}{\gscale\nlscale}\lesssim\frac{\gscale}{\nlscale^2}$ and we can}
can absorb the term $\frac{\phi(\gscale)}{\gscale\nlscale}$ into $\frac{\gscale}{\nlscale^2}$.
\revision{Furthermore, for these two choices} it holds $\frac{\gscale}{\nlscale^2}\gtrsim\frac{\res_n}{\gscale\nlscale}$ and moreover $\sqrt[3]{\frac{\gscale}{\nlscale^2}}=\nlscale$ (respectively, ${\frac{\gscale}{\nlscale^2}}=\nlscale$).
\end{proof}

\revision{In the case of less regular boundaries, where $\gscale^\frac{9}{5} \ll \phi(\gscale)\ll \gscale$ as $\gscale\to 0$, the third error term dominates.
Note that this regime includes non-smooth domains, see \cref{prop:neumann_star} below.
}

\begin{corollary}[Large Length Scale Regime for Less Smooth Boundaries]\label{cor:less_smooth}
We have the following convergence rates:
\begin{enumerate}
    \item For \revision{$\phi(\gscale)^\frac{5}{9}\gtrsim \gscale\gtrsim \res_n^\frac{5}{9}$} it holds
    \begin{align*}
        \sup_{\domain_n}\abs{u-\vec u_n}
        =
        \sup_{\closure\domain}\abs{u-u_n} 
        \lesssim\left(\frac{\phi(\gscale)}{\gscale}\right)^\frac{1}{4}.
    \end{align*}
    \item If $\inf_{\closure\domain}\abs{\grad u}>0$, then for \revision{$\phi(\gscale)^\frac{3}{5}\gtrsim \gscale\gtrsim \res_n^\frac{3}{5}$} it holds
    \begin{align*}
        \sup_{\domain_n}\abs{u-\vec u_n} 
        =
        \sup_{\closure\domain}\abs{u-u_n} 
        \lesssim \left(\frac{\phi(\gscale)}{\gscale}\right)^\frac{1}{2}.
    \end{align*}
\end{enumerate}
\end{corollary}
\begin{proof}
\revision{Assuming for a moment that $\phi(\gscale)\gtrsim \frac{\gscale^2}{\nlscale}$ and $\phi(\gscale)\gtrsim\res_n$ the term} $\frac{\phi(\gscale)}{\gscale\nlscale}$ dominates $\frac{\res_n}{\gscale\nlscale}$ and $\frac{\gscale}{\nlscale^2}$. 
Furthermore, for $\nlscale:=\left({\phi(\gscale)}/{\gscale}\right)^\frac{1}{4}$ (respectively, $\nlscale:=\left({\phi(\gscale)}/{\gscale}\right)^\frac{1}{2}$) it holds $\sqrt[3]{\frac{\phi(\gscale)}{\gscale\nlscale}}=\nlscale$ (respectively, ${\frac{\phi(\gscale)}{\gscale\nlscale}}=\nlscale$).
\revision{Note that for these choices of $\nlscale$ the condition $\phi(\gscale)\gtrsim\frac{\gscale^2}{\nlscale}$ is equivalent to $\phi(\gscale)^\frac{5}{9}\gtrsim\gscale$ or $\phi(\gscale)^\frac{3}{5}\gtrsim\gscale$, respectively.} 
\end{proof}
\begin{remark}[Relation to previous results]
In \cite[Theorem 3.1.3]{smart2010infinity} convergence rates are established for the special case where $\domain_n=[0,1]_\gscale^2$ is a grid, $\constr_n=\domain_n\cap\partial[0,1]^2$ consists of boundary vertices, and the weight function is chosen as $\eta = 1_{[0,1]}$. The condition that $\Omega$ is a grid in two dimensions is not directly used in \cite[Theorem 3.1.3]{smart2010infinity} and can be relaxed without any difficulty. Using the arguments in our paper, the condition that $\constr_n=\domain_n\cap\partial[0,1]^2$ can be relaxed as well, provided the boundary $\partial \Omega$ is $C^{2,\alpha}$, or $\Omega$ is convex.  However, the assumption $\eta=1_{[0,1]}$ seems essential to this result, since pointwise consistency for the discrete graph $\infty$-Laplacian is established using comparison with cones. Translating these results to our setting we would obtain 
\begin{align*}
    \sup_{\domain_n}\abs{u-\vec u_n} = \sup_{\closure\domain}\abs{u-u_n} \lesssim h + \left(\frac{\res_n + \phi(\gscale)}{\gscale^2}\right)^\frac{1}{3}.
\end{align*}
In the case that $\inf_{\closure\domain}\abs{\grad u}>0$, the result can be improved to read
\begin{align*}
    \sup_{\domain_n}\abs{u-\vec u_n} = \sup_{\closure\domain}\abs{u-u_n} \lesssim h + \frac{\res_n + \phi(\gscale)}{\gscale^2}.
\end{align*}
Note that these rates only apply to very large stencil sizes and become meaningless as $\gscale$ approaches $\res_n^\frac{1}{2}$. In contrast, our rates work for arbitrary small graph length scales above the connectivity threshold. In addition, for the rates to be non-vacuous, it is necesary that $\phi(\gscale) \ll \gscale^2$, which requires the boundary to be uniformly $C^{2}$ (or, say, $C^{2,\alpha}$ for $\alpha>0$) when $\Omega$ is non-convex. In constrast, our results only require $\phi(h) \ll h$, which holds for non-convex domains with uniformly $C^1$ boundary. Finally, it does not seem clear how to relax the restrictive condition that $\eta=1_{[0,1]}$ in \cite[Theorem 3.1.3]{smart2010infinity}, which is important for applications to semi-supervised learning, where a very common choice is the Gaussian weight $\eta(t) = \exp(-t^2/(2\sigma^2))1_{[0,1]}$.
\end{remark}

\subsection{Extensions of our Main Results}
\label{sec:ext}

Lastly, we mention two extensions of these results that would be natural to pursue.
\begin{enumerate}
\item In \cite{calder2019consistency} a reweighted Lipschitz learning problem was studied of the form \labelcref{eq:LL} with weights
\begin{equation}\label{eq:reweighted}
w_n(\vec x,\vec y) = d_n(\vec x)^\alpha d_n(\vec y)^\alpha\eta_h(|\vec x-\vec y|),
\end{equation}
where $d_n(\vec x)$ is the degree of node $\vec x$, given by
\[d_n(\vec x) = \sum_{\vec y\in \domain_n}\eta_h(|\vec x-\vec y|).\]
This reweighting modifes the continuum infinity Laplace equation (see \cite[Theorem 2.4]{calder2019consistency}) by adding a drift term along the gradient of the data density; in particular, the continuum equation becomes
\[\Delta_\infty u + 2\alpha \nabla \log \rho \cdot \nabla u = 0\]
in the case that the point cloud $\domain_n$ has density $\rho(x)$ locally as $n\to \infty$. The motivation for the reweighting was to introduce the density $\rho$ into the continuum model in order to improve the semi-supervised classification results. Indeed, the whole point of semi-supervised learning is to use properties of the unlabeled data, e.g., its distribution, to improve classification results. Numerical experiments in \cite{calder2019consistency} established improved performance of the reweighted Lipschitz learning classifier for both synthetic and real data. 

It would be natural to extend the results in this paper to the reweighted model from \cite{calder2019consistency} with weights given in \labelcref{eq:reweighted}. We expect the main modifications to be to the notion of geodesic distance on the domain $\domain$, which will now be weighted by $\rho^{-2\alpha}$. The geodesic distance functions, which are the continuum limit of the graph versions, would be solutions of the eikonal equation $\rho^{2\alpha}|\nabla u|=1$, and the continuum model would be modified accordingly. 
\item In machine learning, it is common to take the \emph{manifold assumption} \cite{goodfellow2016machine}, whereby the data $\Omega_n\subset \R^d$ is sampled from an underlying data manifold $\mathcal{M}\subset \R^d$ of much smaller intrinsic dimension $\text{dim}(\mathcal{M}) = D \ll d$. It would also be natural to extend our results to this setting.
Since an embedded manifold can be equipped with the geodesic metric inherited from the metric on the ambient space, and furthermore comparison with distance functions is a purely metric notion, we expect that our arguments translate to this setting in a straightforward way.
\end{enumerate}


\section{Discrete: Graph Infinity Laplacian Equation}
\label{sec:discrete_problem}

In this section we investigate the discrete problem of the graph infinity Laplacian equation \labelcref{eq:discrete_problem}.
In fact, we will work in an abstract weighted graph setup which is more general than the geometric graphs which we use for our final discrete to continuum results.
The main result in this section is that graph infinity harmonic functions satisfy a comparison principle with graph distance functions.

\subsection{Weighted Graphs}

Let $G=(X,W)$ be a weighted graph with nodes $X$ and edge weights $W=(w_{xy})_{x,y\in X}$. The edge weights are assumed to be nonnegative ($w_{xy}\geq 0$) and symmetric ($w_{xy}=w_{yx}$). 
For $\constr\subset X$ and $g:\constr \to \R$, we consider the graph $\infty$-Laplace equation 
\begin{equation}\label{eq:graph_inflap}
\left\lbrace
\begin{aligned}
\L^G_\infty u &= 0,&& \text{in } X\setminus \constr,\\
u &= g,&& \text{on } \constr,
\end{aligned}
\right.
\end{equation}
where the graph $\infty$-Laplacian is defined by
\begin{equation}\label{eq:GL}
\L^G_\infty u(x) = \min_{y\in X} w_{xy}(u(y) - u(x)) +  \max_{y\in X} w_{xy}(u(y) - u(x)).
\end{equation}
The well-posedness of \labelcref{eq:graph_inflap} was established in \cite{calder2019consistency} and is discussed in more detail in \cref{sec:comp_graph_cones} below. In graph-based semi-supervised learning, \labelcref{eq:graph_inflap} is referred to as Lipschitz learning \cite{calder2019consistency,roith2021continuum,kyng2015algorithms}, since it propagates the labels $u=g$ on $\constr$ to the entire graph in a Lipschitz continuous manner. 
The discrete $\infty$-Laplace equation has also been used for image inpainting, which is concerned with filling in missing parts of an image (e.g., in art restoration), see \cite{elmoataz2015p}.

In this section we study the discrete problem \labelcref{eq:graph_inflap} and show that solutions satisfy a comparison principle with graph distance functions which is the discrete analogue of \cref{def:comparison_distance_functions}.

We define the graph distance function $d_G:X\times X \to \R\cup \{\infty\}$ by
\begin{equation}\label{eq:graph_dist}
d_G(x,y) = \min \left\{\sum_{i=1}^{N}w_{x_{i-1}x_{i}}^{-1} \st N\in\N,\,\{x_i\}_{i=0}^N\subset X,\,x_0=x,\,x_N=y \right\}
\end{equation}
for $x\neq y$, and $d_G(x,x)=0$. 
Here we use the convention that if $w_{x_{i-1}x_i}=0$ for any $i$ then $w_{x_{i-1}x_i}^{-1}=\infty$. 
We note that since the graph $G$ is symmetric, we have $d_G(x,y)=d_G(y,x)$.  

The notion of graph connectivity is now simple to state in terms of the graph distance function.
\begin{definition}\label{def:connectivity}
We say that $G$ is \emph{connected}  if $d_G(x,y) < \infty$ for all $x,y\in X$. 
We say that $G$ is \emph{connected to $\constr$} if for every $x\in X$ there exists $y\in \constr$ with $d_G(x,y) < \infty$.
\end{definition}
We remark that a connected graph is also connected to any constraint set $\constr \subset X$. However, a disconnected graph can still be connected to $\constr$, provided $\constr$ has a non-empty intersection with every connected component of $G$. 

In order to study comparison with graph distance functions, we need a notion of graph boundary.  For $X'\subset X$ we define the boundary of $X'$, denoted $\partial X'$, as the set
\begin{equation}\label{eq:graph_boundary}
\partial X' = \{x\in X\setminus X' \, : \, w_{xy}>0 \text{ for some }y\in X'\}.
\end{equation}
In other words, the boundary $\partial X'$ consists of all nodes in $X\setminus X'$ that are connected to some node in $X'$. This is often called the \emph{exterior} boundary of $X'$. The \emph{interior} boundary can be defined similarly as all nodes in $X'$ that are connected to a node in $X\setminus X'$. While we only use the notion of exterior boundary defined in \labelcref{eq:graph_boundary} in this paper, the results in this section hold, with minor modifications, for the interior boundary. We also define the \emph{closure}  of $X'$, denoted $\closure{X'}$, to be the set
\begin{equation}\label{eq:closure}
\closure{X'}  = X' \cup \partial X'.
\end{equation}

\subsection{Comparison with Graph Distance Functions}
\label{sec:comp_graph_cones}

Our main result in this section shows that graph infinity harmonic functions satisfy comparison with graph distance functions.
\begin{theorem}\label{thm:comp_graph_cones}
Assume $G$ is connected. 
Let $X'\subsetneq X$ and suppose that $u:\closure{X'}\to \R$ satisfies $-\L^G_\infty u(x) \leq 0$ for all $x\in X'$.
Then for every $a \geq 0$ and $z\in X\setminus X'$ we have
\begin{equation}\label{eq:comp_graph_cones}
\max_{\closure{X'}}(u - a\,d_G(\cdot,z)) = \max_{\partial X'}(u -  a\,d_G(\cdot,z)).
\end{equation}
If $u:\closure{X'}\to \R$ satisfies $-\L^G_\infty u(x) \geq 0$ for all $x\in  X'$, then for every $a \geq 0$ and $z\in X\setminus X'$ we have
\begin{equation}\label{eq:comp_graph_cones_below}
\min_{\closure{X'}}(u + a\,d_G(\cdot,z)) = \min_{\partial X'}(u +  a\,d_G(\cdot,z))
\end{equation}
\end{theorem}
\begin{remark}
One could also investigate the converse statement, namely that comparison with graph distance functions implies being a solution to the graph infinity Laplacian equation \labelcref{eq:graph_inflap}.
However, since this is not needed for our result we do not consider this question here.
Note also that we do not have to assume that $X'$ is a connected subset of $X$, as it is typically done for defining AMLEs (see \labelcref{eq:AMLE}).
\end{remark}

The proof of \cref{thm:comp_graph_cones} relies on a discrete comparison principle for \labelcref{eq:graph_inflap}, which was established in \cite{calder2019consistency}. We include the statement of the result below for reference. 
\begin{theorem}[{\cite[Theorem 3.1]{calder2019consistency}}]\label{thm:graph_comp_princ}
Assume that $G$ is connected to $\constr$, and let $u,v:X\to \R$ satisfy 
\[-\L^G_\infty u(x) \leq 0 \leq -\L^G_\infty v(x) \ \ \text{for all } x\in X\setminus \constr.\] 
Then
\[\max_{X}(u-v) = \max_{\constr}(u-v).\]
\end{theorem}
We note that it follows from \cref{thm:graph_comp_princ} and the Perron method that \labelcref{eq:graph_inflap} has a unique solution $u$ whenever $G$ is connected to $\constr$. We refer to \cite[Theorem 3.4]{calder2019consistency} and \cite[Theorem 4]{calder2018game} for the proof of this result.

We now establish that graph distance functions satisfy the graph eikonal equation.
\begin{lemma}\label{lem:eikonal_graph}
Assume $G$ is connected and let $z\in X$. Then the graph distance function $u:=d_G(\cdot,z)$ is the unique solution of the graph eikonal equation
\begin{equation}\label{eq:eikonal_graph}
\max_{y\in X} w_{xy}(u(x) - u(y)) = 1 \ \ \text{for all }x\in X\setminus \{z\}
\end{equation}
satisfying $u(z)=0$.
\end{lemma}
\begin{proof}
It follows directly from the definition of the graph distance $d_G$, \labelcref{eq:graph_dist}, that \revision{$u=d_G(\cdot,z)$} satisfies the dynamic programming principle
\begin{equation}\label{eq:dpp}
u(x)=\min_{y\in X}\left\{u(y)+w_{xy}^{-1}\right\},\quad x\in X\setminus\{z\}.
\end{equation}
As above, we take $w_{xy}^{-1}=\infty$ when $w_{xy}=0$. Since $G$ is connected, for every $x\in X$ there exists $y\in X$ with $w_{xy}>0$, so the minimum in \labelcref{eq:dpp} can be restricted to $y\in X$ with $w_{xy}>0$ (i.e., graph neighbors of $x$). We note that \labelcref{eq:dpp} can be rearranged to show that
\[\max_{y\in X}\{u(x) - u(y) - w_{xy}^{-1}\} = 0,\quad x\in X\setminus\{z\}.\]
It follows that
\begin{equation}\label{eq:eikonal_pos}
\max_{y\in X}\{w_{xy}(u(x) - u(y)) - 1\}  =0, \quad x\in X\setminus\{z\},
\end{equation}
which is equivalent to \labelcref{eq:eikonal_graph}.

To show uniqueness, let $u:X\to \R$ be a solution of \labelcref{eq:eikonal_graph} satisfying $u(z)=0$. Let $\lambda>1$ and let $x_0\in X$ be a point at which $u-\lambda\,d_G(\cdot,z)$ attains its maximum over $X$. Then we have
\[u(x_0) - u(x) \geq \lambda (d_G(x_0,z) - d_G(x,z)), \ \ x\in X.\]
If $x_0\neq z$, then we have
\[1 = \max_{y\in X} w_{x_0y}(u(x_0) - u(y)) \geq  \lambda\max_{y\in X} w_{x_0y}(d_G(x_0,z) - d_G(y,z))  = \lambda > 1,\]
which is a contradiction. Therefore $x_0=z$ and so
\[\max_{X}(u - \lambda\,d_G(\cdot,z)) = u(z) - \lambda\,d_G(z,z) = 0.\]
Since $\lambda>1$ is arbitrary, it follows that $u\leq d_G(\cdot,z)$. A similar argument, examining the minimum of $u - \lambda\,d_G(\cdot,z)$ for $0 < \lambda <1$ shows that $u \geq d_G(\cdot,z)$, which completes the proof.
\end{proof}

In order to apply \cref{thm:graph_comp_princ} to prove \cref{thm:comp_graph_cones}, we need to show that \revision{the subgraph of $G$ with vertices $\closure X'$} is connected to its boundary $\partial X'$ in the sense of \cref{def:connectivity}.  
\begin{proposition}\label{prop:connectivity}
Assume $G$ is connected. Let  $X' \subsetneq X$ and let $G'=(\closure{X'},W')$ be the subgraph of $G$ with weights $W'=(w_{xy})_{x,y\in \closure{X'}}$. Then $G'$ is connected to $\partial X'$. 
\end{proposition}
\begin{proof}
\revision{If $x\in\closure X'\setminus X'$ then $x\in\partial X'$ and $x$ is trivially connected to $\partial X'$.}
Let therefore $x\in X'$. Since $G$ is connected and $X'\subsetneq X$, there exists $y\in X\setminus X'$ and a path $x_0=x,x_1,x_2,\dots,x_n=y$ with $w_{x_ix_{i+1}}>0$ for $i=0,\dots,n-1$. Since $x_0\in X'$ and $x_n\in X\setminus X'$, there exists $0 \leq j\leq n-1$ such that $x_i \in X'$ for $0 \leq i \leq j$ and $x_{j+1}\in X\setminus X'$. It follows that $x_{j+1}\in \partial X'$ and 
\[d_{G'}(x,x_{j+1}) \leq \sum_{i=0}^{j}w_{x_ix_{i+1}}^{-1} < \infty.\]
This completes the proof.
\end{proof}
We now give the proof of \cref{thm:comp_graph_cones}. 
\begin{proof}[Proof of \cref{thm:comp_graph_cones}]
For notational simplicity, let us set $v = d_G(\cdot,z)$. We first show that
\begin{equation}\label{eq:cone_super}
\L^G_\infty v(x) \leq 0 \ \ \text{for all }  x\in X\setminus \{z\}.
\end{equation}
Let $x\in X\setminus \{z\}$. By \cref{lem:eikonal_graph} we have
\[\min_{y\in X}w_{xy}(v(y) - v(x)) = - \max_{y\in X}w_{xy}(v(x)-v(y))= -1.\]
Now let $y_*\in X$ \revision{be} such that
\[w_{xy_*}(v(y_*) - v(x)) =\max_{y\in X}w_{xy}(v(y) - v(x)).\]
Invoking \cref{lem:eikonal_graph} again yields
\[\max_{y\in X}w_{xy}(v(y) - v(x))=w_{xy_*}(v(y_*) - v(x)) \leq \max_{z\in X} w_{zy_*}(v(y_*)-v(z)) = 1,\]
since the weights are symmetric, so that $w_{zy_*} = w_{y_*z}$. This establishes the claim \labelcref{eq:cone_super}.

We now consider the subgraph $G'=(\closure{X'},W')$, as in \cref{prop:connectivity}. Since $\L^{G'}_\infty v(x)=\L^{G}_\infty v(x)$ for all $x\in X'$ we have
\[\L^{G'}_\revision{\infty} (a v)(x) = a \L^{G'}_\revision{\infty} v(x) = a \L^G_\revision{\infty} v(x) \leq 0\]
for all $x\in X'$ and $a\geq 0$. By \cref{prop:connectivity}, $G'$ is connected to $\partial X'$, and so we can invoke the comparison principle, \cref{thm:graph_comp_princ}, to obtain \revision{that for $u$ satisfying $-\L^{G}_\infty(u)\leq 0$ it holds}
\[\max_{\closure{X'}} (u - a\,v) = \max_{\partial X'} (u - a\,v).\]
\revision{
If, conversely, it holds $-\L^{G}_\infty u \geq 0$ we get $-\L^G_\infty(-u)\leq 0$ and hence
\begin{align*}
   \max_{\closure{X'}} (-u - a\,v) = \max_{\partial X'} (-u - a\,v) 
\end{align*}
which is equivalent to
\begin{align*}
   \min_{\closure{X'}} (u + a\,v) = \min_{\partial X'} (u + a\,v).
\end{align*}
This concludes the proof.
}
\end{proof}

\subsection{Relations to Absolute Minimizers}

Since according to \cref{thm:comp_graph_cones}, the solution of \labelcref{eq:graph_inflap} satisfies comparison with graph distance functions, it seems natural that it is also an absolutely minimal Lipschitz extension of $g$ from the constraint set $\constr$ to the entire graph $X$. 
Since, for the sake of simplifying later proofs, we are not working with connected test sets $X'$ in \cref{thm:comp_graph_cones}, we cannot use standard results like \cite[Proposition 4.1]{juutinen2006equivalence} and will prove the statement ourselves.

For a subset $X'\subset X$, and a function $u:X' \to \R$, we define the graph Lipschitz constant
\begin{equation}\label{eq:lip}
\Lip_G(u;X') = \max_{x,y\in X'} \frac{|u(x)-u(y)|}{d_G(x,y)}
\end{equation}
and we abbreviate $\Lip_G(u):=\Lip_G(u;X)$.
Before we can \revision{prove} the absolute minimality we need the following \revision{Lemma}.
\begin{lemma}\label{lem:subsets}
Let $X'\subset X$, $x\in\closure{X'}$, and $X'':=X'\setminus\{x\}$.
Then it holds
\begin{align*}
    \partial X''\subset\partial X'\cup\{x\},
    \qquad
    \closure{X'}\setminus\closure{X''}\subset\partial X'\cup\{x\}.
\end{align*}
\end{lemma}
\begin{proof}
If $X''=\emptyset$ the statements are trivial so we assume that $X''\neq\emptyset$.

For the first statement we distinguish two cases: 
If $z\in \partial X'' \cap X'$, then by definition of the graph boundary \labelcref{eq:graph_boundary} it holds $z\notin X''$ which together with $z\in X'$ implies $z=x$.
If $z\in\partial X'' \setminus X'$, then again by \labelcref{eq:graph_boundary} and the fact that $X''\neq\emptyset$ it follows that there exists $y\in X''\subset X'$ with $w_{yz}>0$ and hence $z\in\partial X'$.
Combining both cases shows the desired statement.

For the second statement we let $z\in\closure{X'}\setminus\closure{X''}$ and argue as follows:
If $z=x$ nothing needs to be shown and therefore we assume $z\neq x$.
We know that in particular it holds $z\in\closure{X'}=X'\cup\partial X'$ so if $z\in\partial X'$ the proof is complete.
The only interesting case is $z\in X'$ in which case we get that
\begin{align*}
    z \in X'\setminus\closure{X''}\subset X'\setminus X'' = X'\setminus(X'\setminus\{x\})=\{x\}.
\end{align*}
This contradicts our assumption and we can conclude.
\end{proof}

\begin{proposition}\label{prop:lipschitz_discr}
Assume $G$ is connected, let $g:\constr\to\R$ be Lipschitz continuous, and let $u:X\to\R$ be a solution of \labelcref{eq:graph_inflap}.
Then for all subsets $X'\subset X\revision{\setminus\constr}$ it holds
\begin{align}
    \Lip_G(u;\closure{X'}) = \Lip_G(u;\partial X').
\end{align}
\end{proposition}
\begin{proof}
By definition of $\Lip_G(u;\partial X')$ we have that
\begin{align}
u(x) - u(y) \leq \Lip_G(u;\partial X')\, d_G(x,y),\quad\forall x,y\in \partial X'.
\end{align}
Using \cref{thm:comp_graph_cones} we get
\begin{align*}
\max_{\closure{X'}} \left(u -\Lip_G(u;\partial X')\, d_G(\cdot,y)\right) &= \max_{\partial X'} \left(u -\Lip_G(u;\partial X')\, d_G(\cdot,y)\right) \leq u(y),\quad\forall y\in \partial X',
\end{align*}
and therefore
\begin{align}\label{eq:lip_to_constr}
u(x) - u(y)&\leq \Lip_G(u;\partial X')\, d_G(x,y), \quad \forall x\in \closure{X'},\;\forall y\in \partial X'.
\end{align}
Note that \labelcref{eq:lip_to_constr} implies that
\begin{align*}
    u(x) = \min_{y\in \partial X'\cup \{x\}}\{u(y) + \Lip_G(u;\partial X')\, d_G(x,y)\},\quad\forall x\in\closure{X'}.
\end{align*}
Fixing $x\in\closure{X'}$, setting $X'' = X'\setminus\{x\}$ and using comparison with cones from below from \cref{thm:comp_graph_cones} we have
\begin{align*}
\min_{\closure{X''}}\left(u +\Lip_G(u;\partial X')\, d_G(x,\cdot)\right) &= \min_{\partial X''}\left(u +\Lip_G(u;\partial X')\, d_G(x,\cdot)\right)\\ &\geq 
\min_{\partial X' \cup\{x\}}\left(u +\Lip_G(u;\partial X')\, d_G(x,\cdot)\right)=
u(x),
\end{align*}
since $\partial X''\subset\partial X'\cup\{x\}$ according to \cref{lem:subsets}.
It follows that
\begin{align}\label{eq:lip_constr_inner}
u(x) - u(y) \leq  \Lip_G(u;\partial X')\, d_G(x,y),\quad\forall y\in\closure{X''}.
\end{align}
By \cref{lem:subsets} we know that $\closure{X'}\setminus\closure{X''}
\subset\partial X'\cup\{x\}$. Using \labelcref{eq:lip_to_constr} we then have that \labelcref{eq:lip_constr_inner} also holds for every $y\in\closure{X'}\setminus\closure{X''}$
and therefore
\begin{align*}
u(x) - u(y) \leq  \Lip_G(u;\partial X')\, d_G(x,y) \quad \forall y\in\closure{X'}.
\end{align*}
Since $x\in\closure{X'}$ was arbitrary, this concludes the proof.
\end{proof}


\section{Continuum: Absolutely Minimizing Lipschitz Extensions}
\label{sec:continuum_problem}

\subsection{Equivalence with Comparison with Distance Functions}

In \cref{sec:main_results} we have already introduced the continuum problem \labelcref{eq:AMLE} of finding an Absolutely Minimizing Lipschitz Extension (AMLE) of the label function $g:\constr\to\R$.
In fact, we will not work with AMLEs directly but rely on an intriguing \revision{property} of theirs, called Comparison with Distance Functions (CDF), see \cite{champion2007principles,juutinen2006equivalence}.
\begin{definition}[CDF]\label{def:comparison_distance_functions}
We shall say that a upper semicontinuous function $u\in USC(\closure \domain)$ satisfies CDF from above in $\domain_\constr$, if  for each relatively open and connected subset $V \subset \domain_\constr$, any \revision{$x_0\in \closure\domain\setminus V$} and $a\geq 0$ we have
\[\max_{\closure V}(u - a\,d_\domain(x_0,\cdot)) = \max_{\revision{\relpartial V}}(u - a\,d_\domain(x_0,\cdot)).\]
Similarly, we say  $u\in LSC(\closure\domain)$ satisfies CDF from below in $\domain_\constr$, if for each relatively open and connected subset \revision{$V\subset \domain_\constr$}, any \revision{$x_0\in \closure\domain\setminus V$} and $a\geq 0$ we have
\[\min_{\closure V}(u + a\,d_\domain(x_0,\cdot)) = \min_{\revision{\relpartial V}}(u + a\,d_\domain(x_0,\cdot)).\]
We say $u\in C(\closure\domain)$ satisfies CDF if it satisfies CDF from above and below.
\end{definition}
\revision{
\begin{remark}
Note that this definition is a special case of \cite[Definition 2.3]{juutinen2006equivalence}.
For this, one regards $X:=\closure\domain$ equipped with $d_\domain$ as a length space and $\domain_\constr$ as an open subset of $X$ with respect to the topology of $d_\domain$, as explained in \cref{subsec:continuum_problem}. 
Consequently, open subsets of $\domain_\constr$ with respect to this topology are relatively open subsets.
\end{remark}}

In fact, being an AMLE is equivalent to satisfying CDF, which is why we will only work with this property for the rest of the paper.
\begin{proposition}[{\cite[Proposition 4.1]{juutinen2006equivalence}}]
Let $g:\constr\to\R$ be Lipschitz continuous.
A function $u:C(\closure\domain)\to\R$ with $u=g$ on $\constr$ is an AMLE of $g$ if and only if it satisfies CDF on $\domain_\constr$.
\end{proposition}
Hence, we can equivalently reformulate the continuum problem as to find a function $u:\closure\domain\to\R$ which attains the label values on $\constr$ and satisfies CDF in $\domain_\constr$:
\begin{align}\label{eq:continuum_problem}
    \begin{cases}
    u\text{ satisfies CDF} &\text{on }\domain_\constr,\\
    u = g\quad&\text{on }\constr.
    \end{cases}
\end{align}
For convenience we also introduce the notion of sub- and super solutions to \labelcref{eq:continuum_problem}.
\begin{definition}[Sub- and supersolutions]
We call $u\in USC(\closure\domain)$ a subsolution of \labelcref{eq:continuum_problem} if
\begin{align*}
    \begin{cases}
    u\text{ satisfies CDF from above} &\text{on }\domain_\constr,\\
    u \leq g\quad&\text{on }\constr.
    \end{cases}
\end{align*}
Furthermore, $u\in LSC(\closure\domain)$ is called a supersolution if $-u$ is a subsolution.
Obviously, being both a sub- and a supersolution is equivalent to being a solution of \labelcref{eq:continuum_problem} and hence also to being an AMLE.
\end{definition}

\subsection{Relations to Infinity Laplacian Equations}
\label{sec:cont_rel_inf_lapl}

Under some regularity conditions on the boundary of the domain $\domain$ (now regarded as \revision{a} subset of $\R^d$), \revision{and if $\constr$ is a subset of $\partial\domain$}, one can relate being an AMLE, or equivalently satisfying CDF, with solving an infinity Laplacian equation, where the infinity Laplacian operator is defined as
\begin{align*}
\Delta_\infty u = \langle \nabla u, D^2 u \nabla u\rangle.
\end{align*}

\begin{proposition}
Let $\partial\domain$ denote the boundary of $\domain$, regarded as subset of $\R^d$.
Then the following is true:
\begin{enumerate}
    \item If $\constr=\partial\domain$ then $u\in C(\closure\domain)$ is an AMLE of $g:\partial\domain\to\R$ if and only if it is a viscosity solution of
    \begin{align}
        \begin{cases}
        -\Delta_\infty u &= 0\quad\text{in }\domain,\\
        u &= g\quad\text{on }\partial\domain.
        \end{cases}
    \end{align}
    \item If $\domain$ is \revision{smooth and} convex \revision{and $\constr\subset\partial\domain$} then $u\in C(\closure\domain)$ is an AMLE of $g:\constr\to\R$ if and only if it is a viscosity solution of
    \begin{align}
        \begin{cases}
        -\Delta_\infty u &= 0\quad\text{in }\domain\setminus\constr,\\
        \frac{\partial u}{\partial\nu} &= 0\quad\text{on }\partial\domain\setminus\constr,\\
        u &= g\quad\text{on }\constr,
        \end{cases}
    \end{align}
    where the Neumann boundary conditions are to be understood in the weak viscosity sense, cf. \cite{crandall1992user,esposito2015neumann}.
\end{enumerate}
\end{proposition}
\begin{proof}
The first statement is well-known and contained in the seminal paper by Aronsson \cite{aronsson2004tour}\revision{, see Theorem 4.1 therein.}
The second one is a special case of the results in \cite{armstrong2011infinity}.
\revision{In Lemma 3.1 therein} the authors show that, \revision{under the hypothesis that $\domain$ is smooth and convex}, the solution of the PDE admits CDF and hence is an AMLE. 
Furthermore, by the uniqueness of the CDF problem (see \cref{prop:uniqueness} below), \revision{which does, however, not require any convexity}, this is in fact an equivalence.
\end{proof}

\subsection{Max-Ball and Perturbation Statements}
\label{sec:max_ball_perturb}

The following ``max-ball'' and perturbation statements will be essential when we prove discrete-to-continuum convergence rates.
In a simpler setting, using Dirichlet boundary conditions $\constr=\partial\domain$, they were all proved in \cite{smart2010infinity}.
However, for our more general problem \labelcref{eq:continuum_problem} we need to reprove them.

They main point of the max-ball statement is that sub- and supersolutions of \labelcref{eq:continuum_problem} can be turned into sub- and supersolutions of a nonlocal difference equation which is much easier to study. 
In particular, as shown in \cite{armstrong2010easy} the nonlocal equation allows us to derive a maximum principle for \labelcref{eq:continuum_problem} which will let us prove uniqueness and discrete-to-continuum convergence.

The perturbation statement will \revision{allow} us to turn sub- or supersolutions of the nonlocal equation into strict sub- or supersolutions which will turn out useful for deriving convergence rates.

Because of the nonlocal nature of the above statements, one has to work on \revision{an} inner parallel set of $\closure\domain$.
However, 
it suffices to maintain positive distance to the constraint set $\constr$ only.
In what follows we denote by
\begin{align}
    B_\domain(x,\eps) := \{y\in\closure\domain\st d_\domain(x,y)<\eps\}
\end{align}
the open geodesic ball around $x$ with radius $\eps>0$ and its closure is denoted by $\closure B_\domain(x,\eps)$.
Using this, we can define the set
\begin{align}\label{eq:inner_parallel_set}
    \domain^{\nlscale}_\constr:=\{x\in\closure{\domain}\st \dist_\domain(x,\constr)>\nlscale\}.
\end{align}
Note that this ``inner parallel set'' deliberately includes those parts of the topological boundary $\partial\domain$ which are sufficiently far from the constraint set $\constr$.

For a function $u:\closure\domain\to\R$ and $x\in\closure\domain$ we also define
\begin{alignat}{2}
    \label{eq:T-operators}
    T^\nlscale u(x) &:= \sup_{\closure B_\domain(x,\nlscale)} u, 
    \qquad
    T_\nlscale u(x) &&:= \inf_{\closure B_\domain(x,\nlscale)} u,
    \\
    \label{eq:slopes}
    S_\nlscale^+ u &:= \frac{1}{\nlscale}(T^\nlscale u - u),
    \qquad
    S_\nlscale^- u &&:= \frac{1}{\nlscale}(u-T_\nlscale u),
    \\
    \label{eq:nonlocal_lapl}
    \Delta_\infty^\nlscale u &:= \frac{1}{\nlscale}\left(S^+_\nlscale u - S^-_\nlscale u\right). &&
\end{alignat}
We refer to the last object as nonlocal infinity Laplacian operator.
\revision{Interestingly, harmonic functions with respect to this operator coincide with AMLEs with respect to a discrete ``step distance'' and we refer the interested reader to \cite{mazon2012best} for more details.}
The following lemma, the proof of which is adapted from \cite{armstrong2010easy}, is the desired max-ball statement.
It states that functions satisfying CDF can be turned into a sub- and supersolution of the nonlocal infinity Laplacian equation.
\begin{lemma}[Max-Ball]\label{lem:nonloc2loc}
Let $u\in USC(\closure\domain)$ and $v\in LSC(\closure\domain)$ satisfy CDF on $\domain_\constr$ from above and from below, respectively.
Then for all $\nlscale > 0$ it holds
\begin{align}
    -\Delta_\infty^\nlscale T^\nlscale u(x_0) 
    \leq 
    0
    \leq
    -\Delta_\infty^\nlscale T_\nlscale v(x_0),\quad \forall x_0\in \domain^{2\nlscale}_\constr.
\end{align}
\end{lemma}
\begin{proof}
We just show the inequality for the subsolution $u$, the other case works analogously.
Since $u$ is upper semicontinous we can select points $y_0\in\closure B_\domain(x_0,\nlscale)$ and $z_0\in \closure B_\domain(x_0,2\nlscale)$ such that
\begin{align*}
    u(y_0) = T^\nlscale u(x_0), \quad
    u(z_0) = T^{2\nlscale} u(x_0).
\end{align*}
Using the definition of the nonlocal infinity Laplacian \labelcref{eq:nonlocal_lapl} one computes
\begin{align}
    \label{ineq:max_ball_pf_1}
    -\nlscale^2\Delta_\infty^\nlscale T^\nlscale u(x_0) 
    =2T^\nlscale u(x_0)
    - (T^\nlscale T^\nlscale u)(x_0)
    - (T_\nlscale T^\nlscale u)(x_0) 
    \leq 2 u(y_0) - u(z_0) - u(x_0).
\end{align}
This is true since
$x_0\in \closure B_\domain(y,\nlscale)$ for all 
$y\in \closure B_\domain(x_0,\nlscale)$ and therefore,
\begin{align*}
T^\nlscale u(y)=\sup_{B_\domain(y,\nlscale)} u \geq u(x_0)
\implies
(T_\nlscale T^\nlscale u)(x_0)=
\inf_{y\in B_\domain(x_0,\nlscale)} T^\nlscale(y) 
\geq u(x_0).
\end{align*}
Now one observes that the following inequality is true
\begin{align*}
    u(w) \leq u(x_0) + \frac{u(z_0)-u(x_0)}{2\nlscale}d_\domain(w,x_0),\quad\forall w\in
    \relpartial(B_\domain(x_0,2\nlscale)\setminus\{x_0\}).
\end{align*}
For $w=x_0$ the inequality is obviously correct.
If $d_\domain(w,x_0)=2\nlscale$, then the definition of $z_0$ shows it.

Since $u$ satisfies CDF from above on $\domain_\constr$ and using that $\dist_\domain(x_0,\constr)>2\nlscale$, we obtain
\begin{align*}
    u(w) \leq u(x_0) + \frac{u(z_0)-u(x_0)}{2\nlscale}d_\domain(w,x_0) ,\quad\forall w\in\closure B_\domain(x_0,2\nlscale).
\end{align*}
Choosing $w=y_0$ we get
\begin{align*}
    u(y_0) \leq u(x_0) + \frac{u(z_0)-u(x_0)}{2\nlscale}d_\domain(y_0,x_0)
    \leq 
    u(x_0) + \frac{u(z_0)-u(x_0)}{2}.
\end{align*}
Using this inequality together with \labelcref{ineq:max_ball_pf_1} yields the assertion.
\end{proof}
The next results are contained in \cite{smart2010infinity} which treats a more general graph Laplacian for a graph $G=(X,E,Y)$ consisting of a vertex set $X$, an edge set $E$, and a ``boundary set'' $Y$.
This graph Laplacian is defined as
\begin{align}
    \Delta_\infty^G u(x) = \sup_{\{x,y\}\in E}(u(y)-u(x)) - 
    \sup_{\{x,y\}\in E}(u(x)-u(y)) 
\end{align}
and for the choice 
\begin{align*}
    X &:= \domain_\constr^\nlscale,
    \\
    Y &:= \{x\in X\st \dist_\domain(x_0,\constr)\leq 2\nlscale\},
    \\
    E &:= \{\{x,y\}\subset X\st x\in X\setminus Y,\;0<d_\domain(x,y)\leq\nlscale\},
\end{align*}
it holds
\begin{align}
    \Delta_\infty^G u(x) = \nlscale^2\Delta_\infty^\nlscale u(x),
\end{align}
where $\Delta_\infty^\nlscale$ is defined in \labelcref{eq:nonlocal_lapl}.
Hence, we can use the results from \cite{smart2010infinity} to obtain statements about the nonlocal Laplacian $\Delta_\infty^\nlscale$.

The first result is a maximum principle for the nonlocal infinity Laplacian.
\begin{lemma}\label{lem:nonlocal_max_princ}
Assume that for a constant $C\geq 0$ the functions $u,v:\domain^{\nlscale}_\constr\to\R$ satisfy
\begin{align}
-\Delta_\infty^\nlscale u \leq \revision{C} \leq -\Delta_\infty^\nlscale v
\end{align}
in $\domain^{2\nlscale}_\constr$.
Then it holds
\begin{align}
    \sup_{\domain^\nlscale_\constr}(u-v) = \sup_{\domain^\nlscale_\constr\setminus\domain^{2\nlscale}_\constr}(u-v).
\end{align}
\end{lemma}
\begin{proof}
\revision{See} \cite[Theorem 2.6.5]{smart2010infinity}.
\end{proof}
The following two lemmas are perturbation results.
The first one allows us to turn a supersolution of the nonlocal equation into one with strictly positive gradient.
The second one shows how to turn a supersolution into a strict supersolution.

\begin{lemma}\label{lem:perturb_pos_grad}
If $u:\domain_\constr^{\nlscale}\to\R$ is bounded from below and satisfies $-\Delta_\infty^\nlscale u \geq 0$ in $\domain^{2\nlscale}_{\constr}$, then for any $\delta>0$ there is a function $v:\domain_\constr^\nlscale\to\R$ that satisfies
\begin{align*}
    -\Delta_\infty^\nlscale v\geq 0,\quad S_\nlscale^- v \geq {\delta},\quad u \leq v \leq u + 2 \delta \dist_\domain(\cdot,\domain_\constr^\nlscale\setminus\domain_\constr^{2\nlscale})\quad\text{on }\domain_\constr^{2\nlscale}.
\end{align*}
\end{lemma}
\begin{proof}
\revision{See} \cite[Lemma 2.6.3]{smart2010infinity}.
\end{proof}
\begin{lemma}\label{lem:perturb_strict}
Suppose $v:\domain_\constr^\nlscale\to\R$ is bounded and $-\Delta_\infty^\nlscale v \geq 0$ on $\domain_\constr^{2\nlscale}$.
Then there exists $\delta_0>0$ such that for any $0\leq\delta\leq\delta_0$ the function $w:=v-\delta v^2$ satisfies
\begin{align*}
    -\Delta_\infty^\nlscale w \geq -\Delta_\infty^\nlscale v + \delta(S^-_\nlscale v)^2\quad\text{on }\domain_\constr^{2\nlscale}.
\end{align*}
\end{lemma}
\begin{proof}
If $v\geq 0$ the map $t\mapsto t+\delta t^2$ is monotone on the range of $v$ for all $\delta\geq 0$.
If $v\geq -c$ for some constant $c>0$ then it is monotone for all $0\leq \delta\leq \frac{1}{2c}$.
From there on the proof works verbatim as in \cite[Lemma 2.6.4]{smart2010infinity}.
\end{proof}
\begin{remark}
Obviously, these two lemmata have analogous versions for subharmonic functions, which are obtained by replacing $u$ with $-u$, $S_\nlscale^-$ with $S_\nlscale^+$, etc.
\end{remark}

\subsection{Existence and Uniqueness}

Existence for AMLEs \labelcref{eq:AMLE} or equivalently of solutions for the CDF problem \labelcref{eq:continuum_problem} is well understood. Indeed, since $(\closure\domain,d_\domain)$ is a length space, existence of solutions can be obtained using Perron's method, see \cite{juutinen2006equivalence,juutinen2002absolutely,milman1999absolutely} for details.

Using \cref{lem:nonloc2loc,lem:nonlocal_max_princ} and the same strategy as in \cite{armstrong2010easy} we obtain uniqueness of solutions to the continuum problem \labelcref{eq:continuum_problem} which we state in the following proposition.

\begin{proposition}[Uniqueness]\label{prop:uniqueness}
There exists at most one solution of \labelcref{eq:continuum_problem}.
\end{proposition}
\begin{proof}
The proof works verbatim as for the main result in \cite{armstrong2010easy}.
Letting $u,v\in C(\closure\domain)$ denote two solutions, \cref{lem:nonloc2loc} implies that
\begin{align*}
    -\Delta_\infty^\nlscale T^\nlscale u(x_0) 
    \leq 
    0
    \leq -\Delta_\infty^\nlscale T^\nlscale v(x_0) ,\quad\forall x_0\in\domain^{2\nlscale}_\constr.
\end{align*}
Now \cref{lem:nonlocal_max_princ} implies
\begin{align*}
    \sup_{\domain^\nlscale_\constr}(T^\nlscale u - T_\nlscale v) = 
    \sup_{\domain^\nlscale_\constr\setminus\domain^{2\nlscale}_\constr}(T^\nlscale u - T_\nlscale v).
\end{align*}
Sending $\nlscale\searrow 0$ implies
\begin{align*}
    \sup_{\closure\domain}(u-v) = 
    \sup_{\constr}(u-v) = 0
\end{align*}
and hence $u\leq v$ in $\closure\domain$.
Swapping the roles of $u$ and $v$ finally implies that $u=v$ on $\closure\domain$.
\end{proof}


\section{Discrete-to-Continuum Convergence}\label{sec:convergence}

In this section we first prove convergence rates of the graph distance functions defined in \cref{sec:discrete_problem} to geodesic distance functions, in the continuum limit on geometric graphs.
Then we utilize this to prove convergence rates of solutions to the graph infinity Laplacian equation \labelcref{eq:graph_inflap} to the continuum problem \labelcref{eq:continuum_problem}.

We write $B(x,r):=\{y\in\R^d\st\abs{x-y}<r\}$ for the Euclidean open ball of radius $r>0$ centered at $x\in \R^d$ and denote its closure by $\revision{\closure B(x,r)}$.
Recall that we define the geodesic distance $d_\domain(x,y)$ on $\closure\domain$ by
\[d_\domain(x,y) = \inf\left\{ \int_0^1 |\dot{\xi}(t)|\, \de t \st \xi\in C^1([0,1];\closure \domain) \text{ with } \xi(0)=x \text{ and } \xi(1)=y\right\}.\]
If the line segment from $x$ to $y$ is contained in $\closure \domain$, then $d_\domain(x,y)=|x-y|$. 
Remember that we pose \cref{ass:geodesic_euclidean_2} on the relation between the geodesic and Euclidean distance.
In the following we list important cases where the assumption is satisfied.
\begin{proposition}\label{prop:boundary_reg}
Let $\domain\subset\R^d$ be an open and bounded domain.
\begin{enumerate}
    \item If $\domain$ is convex, \cref{ass:geodesic_euclidean_2} is satisfied with $\phi=0$.
    \item If $\domain$ has a $C^{1,\revision{\alpha}}$ boundary \revision{for $\alpha\in(0,1]$}, \cref{ass:geodesic_euclidean_2} is satisfied with $\phi(h)=Ch^{\revision{1+\alpha}}$ for some constant $C>0$.
\end{enumerate}
\end{proposition}
\begin{proof}
\revision{
The proof of 1 is immediate. The proof of 2 is split into 3 steps.
\paragraph{Step 1} We first show that for every $x,y\in \overline{\domain}$ we have
\begin{equation}\label{eq:geod_eucl1}
d_\domain(x,y) \leq |x-y| + \sup_{z\in \partial \domain \cap B(x,r)} \left\{ d_\domain(x,z) - |x-z|\right\},
\end{equation}
where $r=|x-y|$. To see this, we use a maximum principle argument. Let $u(z) = d_\domain(x,z)$, $v(z)= |x-z|$, $\lambda>1$ and set $V= \domain \cap B(x,r)$. Let $z_*\in \overline{V}$ be a point where $u-\lambda v$ attains its maximum value over $\overline{V}$. Since $u$ is a viscosity solution of $|\nabla u|=1$ on $V\setminus \{x\}$ and $\lambda>1$, we must have $z_*\in \partial V\cap \{x\}$.  Now, we claim that $z_*\not\in \domain\cap \partial B(x,r)$, since if this were the case, then we could find a small $\varepsilon>0$ such that $B(z_*,\varepsilon)\subset \domain$, and setting $\nu = \frac{x - z_*}{r}$ we have
\[u(z_* + \varepsilon \nu) - \lambda v(z_* + \varepsilon\nu) \geq u(z_*) - \varepsilon - \lambda v(z_*) + \lambda\varepsilon = u(z_*) - \lambda v(z_*) + (\lambda-1)\varepsilon,  \]
which contradicts the optimality of $z_*$, as $z_* + \varepsilon \nu\in V$. Thus $z_*\in \partial\domain \cap B(x,r)$ or $z_*=x$. If $z_*=x$ then $u - \lambda v \leq 0$ on $\overline{V}$. If $z_*\in \partial \domain \cap B(x,r)$ then 
\[u - \lambda v \leq \sup_{\partial \domain \cap B(x,r)} \left\{ u(z) - \lambda v\right\}.\]
Sending $\lambda\searrow 1$ establishes \labelcref{eq:geod_eucl1}. 

\paragraph{Step 2} We now show that
\begin{equation}\label{eq:boundary_est}
d_\domain(x,z) \leq |x-z| + C|x-z|^{1+\alpha},
\end{equation}
whenever  $x\in \overline{\domain}$, $z\in \partial\domain$, and $|x-z|\leq r_\domain$, where $r_\domain>0$ will be determined below and $C$ depends only on $\domain$. Without loss of generality, we may assume that $z=0$, and that
\begin{equation}\label{eq:bound_reg}
\domain\cap B_r = \{x\in B_r \, : \, x_d < \Phi(\overline{x})\},
\end{equation}
where $\overline{x} = (x_1,\dots,x_{d-1})$, $\Phi:\R^{d-1} \to \R$ is $C^{1,\alpha}$ with $\nabla \Phi(0)=0$, and $0 < r \leq r'_\domain$, where $r'_\domain$ depends only on $\domain$.
Since $\partial\domain$ is $C^{1,\alpha}$, there exists $C>0$ so that
\begin{equation}\label{eq:reg}
|\Phi(\overline{x})| \leq C|x|^{1+\alpha} \ \ \text{and}  \ \ |\nabla \Phi(\overline{x})|\leq C|x|^\alpha.
\end{equation}
Note we can write
\begin{equation}\label{eq:length}
d_\domain(x,0) = \inf_\gamma \int_0^1 |\gamma'(t)| \, dt,
\end{equation}
where the infimum is over $C^1$ curves $\gamma:[0,1]\to \overline{\domain}$ satisfying $\gamma(0)=0$ and $\gamma(1)=x$. Let us set
\[\gamma(t) =  tx + (\Phi(t\overline{x}) - t\Phi(\overline{x}))e_d.\]
Then $\gamma(0)=0$ and $\gamma(1)=x$.  Since $x_d \leq \Phi(\overline{x})$ we have
\[\gamma_d(t) = tx_d +\Phi(t\overline{x}) - t\Phi(\overline{x}) = t(x_d - \Phi(\overline{x_d})) + \Phi(t\overline{x}) \leq \Phi(t\overline{x}) =\Phi\left(\overline{\gamma(t)}\right).\]
By \labelcref{eq:reg} we have
\[|\gamma(t)| \leq t|x| + |\Phi(t\overline{x})| + t|\Phi(\overline{x})| \leq |x| + C|x|^{1+\alpha}.\]
Let us set $r=(1+C)|x|$ and restrict $|x|\leq r_\domain$, where $r_\domain \leq 1$ is sufficiently small so that $|x|\leq r_\domain$ implies $r \leq r'_\domain$ (i.e., $(1+C)r_\domain \leq r'_\domain$). Since $|x|\leq 1$ we have
\[|\gamma(t)| \leq (1+C)|x| = r,\]
and so $\gamma(t) \in B_r$ for all $0\leq t\leq 1$. It follows that $\gamma(t) \in \overline{\domain} \cap B_r$ for all $0\leq t\leq 1$.  We now compute
\[\gamma'(t) = x + \left(\langle \nabla \Phi(t\overline{x}),\overline{x}\rangle - \Phi(\overline{x})\right) e_d,\]
and so by \labelcref{eq:reg} we have
\[|\gamma'(t)| \leq |x| + |\nabla \Phi(\overline{x})| |x| + |\Phi(\overline{x})| \leq |x| + C|x|^{1+\alpha}.\]
Substituting this into \labelcref{eq:length} completes the proof of \labelcref{eq:boundary_est}.

\paragraph{Step 3} Combining steps 1 and 2 we find that
\[d_\domain(x,y) \leq |x-y| + C\sup_{z\in \partial \domain \cap B(x,|x-y|) }|x-z|^{1+\alpha} \leq |x-y| + C|x-y|^{1+\alpha},\] 
provided $|x-y|\leq r_\domain$, which completes the proof.}
\end{proof}
%
%

\begin{remark}
\cref{ass:geodesic_euclidean_2} is slightly stronger than the one made in \cite{roith2021continuum} which takes the form
\begin{align*}
    \limsup_{h\downarrow 0}\sup_{\abs{x-y}\leq h}\frac{d_\domain(x,y)}{\abs{x-y}}=1.
\end{align*}
This is an asymptotic version of \cref{ass:geodesic_euclidean_2} which was used in \cite{roith2021continuum} to prove a Gamma-convergence result.
However, for the quantitative analysis in this paper, this does not suffice.
\end{remark}

\revision{Note that also non-smooth domains can satisfy \cref{ass:geodesic_euclidean_2}.
A particular interesting example is the following star-shaped subset of $\R^2$:
\begin{align}\label{eq:neumann_star}
    \domain:=\left\{x\in[0,1]^2 \st \abs{x_1}^{2/3}+\abs{x_2}^{2/3}\leq 1\right\}.
\end{align}
See \cref{sec:numerics} for figures and numerical examples on this domain.
While the boundary $\partial \domain$ is only $C^{0,\frac{2}{3}}$ H\"older regular, we still have the following result.
\begin{proposition}\label{prop:neumann_star}
The domain \labelcref{eq:neumann_star} satisfies \cref{ass:geodesic_euclidean_2} with $\phi(h) = C h^{\frac{3}{2}}$ for some constant $C>0$ depending on $\domain$. 
It even holds that
\begin{equation}\label{eq:NS_reg}
d_\domain (x,y) \leq |x-y| + C|x-y|^{\frac{3}{2}},\quad\forall x,y\in\domain.
\end{equation}
\end{proposition}
\begin{proof}
We sketch a proof of this here. 
We first assume $x,y\in \domain$ are in the same quadrant, which we can assume, by symmetry, to be the first quadrant $\domain_1:=\domain \cap [0,\infty)^2$. In this case, the shortest path between $x$ and $y$ must also lie in the quadrant $\domain_1$, since if it were ever to leave and subsequently return to the quadrant, we could construct a path with strictly shorter length by projecting the portion of the path that left $\domain_1$ back to the quadrant. Given this observation, we have that 
\[d_\domain(x,y) = d_{\domain'}(x,y),\]
where the domain $\domain'$ given by
\[\domain' = \domain  \cup (\R^2 \setminus [0,\infty)^2).\]
It can be easily calculated that the domain $\domain'$ has a $C^{1,\frac{1}{2}}$ boundary, and so it follows from part 2 of \cref{prop:boundary_reg} that \labelcref{eq:NS_reg} holds whenever $x,y\in \domain$ are in the same quadrant.

The case where $x$ and $y$ lie in different quadrants is handled by symmetry. We first consider the case where $x$ and $y$ lie in two different quadrants that belong to a common halfspace. Without loss of generality, we take $x\in \domain_1$ and 
\[y\in \domain_2=\domain\cap [0,\infty)\times (-\infty,0].\]
These quadrants belong to the common halfspace $H:=[0,\infty)\times \R$, and a similar argument as above can be made to show that the shortest path between $x$ and $y$ must remain in $H\cap \domain$. Let $z\in H\cap \partial \domain_1\cap \partial \domain_2$ be the point on the line segment between $x$ and $y$. Then by the triangle inequality we have
\begin{equation}\label{eq:triangle_ineq}
d_\domain(x,y) \leq d_\domain(x,z) + d_\domain(z,y).
\end{equation}
Since $x,z\in \overline{\domain_1}$, and $z,y\in \overline{\domain_2}$, we can apply \labelcref{eq:NS_reg} to obtain
\[d_\domain(x,y) \leq |x-z| + |z-y| + C(|x-z|^{\frac{3}{2}} + |z-y|^{\frac{3}{2}}). \]
Since $z$ is on the line segment between $x$ and $y$ we have $|x-z| + |z-y| = |x-y|$ and so
\[d_\domain(x,y) \leq |x-y| + 2C|x-y|^{\frac{3}{2}}. \]

The last case is where $x$ and $y$ belong to diagonally opposing quadrants. Without loss we can consider $x\in \domain_1$ and 
\[y\in \domain_3=\domain\cap (-\infty,0]^2.\]
Here, we let $z\in \partial \domain_3$ be the point along the line segment between $x$ and $y$. Since the points $x$ and $z$ lie in the same halfspace, the previous step shows that
\[d_\domain(x,z) \leq |x-z| + 2C|x-z|^{\frac{3}{2}}. \]
Since $z$ and $y$ lie in the same quadrant we have 
\[d_\domain(z,y) \leq |z-y| + C|z-y|^{\frac{3}{2}}. \]
Inserting these into the triangle inequality \labelcref{eq:triangle_ineq} yields
\[d_\domain(x,y) \leq |x-y| + 3C|x-y|^{\frac{3}{2}}, \]
which completes the proof.
\end{proof}
}

\subsection{Geometric Graphs}

Remember that we consider a set of points $\domain_n\subset\closure\domain$ and a subset $\constr_n\subset\domain_n$ of labelled vertices, which acts as the discrete analog of $\constr$. 
On the point cloud $\domain_n$ we define the geometric graph $\vec G_{n,\gscale}=(\domain_n,\vec W_{n,\gscale})$ with vertices $\domain_n$ and a set of edge weights $\vec W_{n,\gscale}=(\vec w_{n,\gscale}(\vec x,\vec y))_{\vec x,\vec y\in \domain_n}$ given by
\begin{equation}\label{eq:weights}
\vec w_{n,\gscale}(\vec x,\vec y) =
\begin{cases}
0,\quad&\vec x = \vec y, \\
\sigma_\eta^{-1}\eta_\gscale\left(|\vec x-\vec y| \right),\qquad&\revision{\vec x\neq\vec y,}
\end{cases}
\end{equation}
where
\begin{equation}\label{eq:sigmaeta}
\sigma_\eta = \sup_{\revision{t>0}}t\eta(t).
\end{equation}
By \cref{ass:kernel} we can choose $t_0\in (0,1]$ as a maximum of $t\eta(t)$, so that $\sigma_\eta = t_0\eta(t_0)$. When the value of $t_0$ is not unique, the convergence rates below are slightly better by choosing the largest such $t_0$.

To simplify notation we will write $\vec G_n$ in place of $\vec G_{n,\gscale}$, and $\vec w_n$ in place of $\vec w_{n,\gscale}$, when the value of $\gscale$ is clear from the context.   We will denote the graph distance function $\vec d_{\vec G_n}$, defined in \labelcref{eq:graph_dist} by $\vec d_n:\domain_n\times \domain_n\to \R$ and the Lipschitz constant of a function $\vec u_n:\domain_n\to\R$ as $\Lip_n(\vec u_n):=\Lip_{\vec G_n}(\vec u_n)$.

\subsection{Convergence of Cones}
\label{sec:cones}

The proof of convergence of the graph distance $\vec d_n$ to the distance function $d_\domain$ is split into two parts. \cref{lem:cone_lower} gives the lower bound, while \cref{lem:cone_upper} gives the upper.
\begin{lemma}\label{lem:cone_lower}
Let \cref{ass:kernel,ass:geodesic_euclidean_2} hold.
Then for $\gscale \leq r_\domain$ and all $\vec x,\vec y\in \domain_n$ we have
\begin{equation}\label{eq:cone_lower}
\vec d_n(\vec x,\vec y) \geq  |\vec x - \vec y| \vee \left(1-\frac{ \phi(\gscale)}{\gscale}\right)d_\domain(\vec x,\vec y).
\end{equation}
\end{lemma}
\begin{proof}
Let $\vec x,\vec y\in \domain_n$. If there is no path in $\vec G_n$ from $\vec x$ to $\vec y$ then $\vec d_n(x,y) = \infty$ and \labelcref{eq:cone_lower} holds trivially. Thus, we may assume there is a path 
\[ \vec x = \vec x_1,\vec x_2,\dots,\vec x_m = \vec y\] 
such that $\vec w_n(\vec x_i,\vec x_{i+1}) > 0$ for all $i=1,\dots,m-1$. We can choose the path to be optimal for $\vec x,\vec y$, so that 
\[ \vec d_n(\vec x,\vec y) = \sum_{i=1}^{m-1} \vec w_n(\vec x_i,\vec x_{i+1})^{-1}.\]
It follows that
\[\vec d_n(\vec x,\vec y) = \sum_{i=1}^{m-1} \sigma_\eta \eta_\gscale(|\vec x_{i}-\vec x_{i+1}|)^{-1}  =\sum_{i=1}^{m-1} \sigma_\eta \gscale\eta\left(\frac{|\vec x_{i}-\vec x_{i+1}|}{\gscale}\right)^{-1}.\] 
By definition of $\sigma_\eta$, \labelcref{eq:sigmaeta}, we have $\sigma_\eta\eta(t)^{-1} \geq t$, and so
\[\vec d_n(\vec x,\vec y) \geq \sum_{i=1}^{m-1} |\vec x_{i}-\vec x_{i+1}|.\] 
It follows from the triangle inequality that $\vec d_n(\vec x,\vec y)\geq | \vec x-\vec y|$. 
Also, using \cref{ass:geodesic_euclidean_2} we have
\[d_\domain(\vec x_i,\vec x_{i+1}) \leq |\vec x_i - \vec x_{i+1}| + \phi(|\vec x_i - \vec x_{i+1}) \leq |\vec x_i - \vec x_{i+1}| +  d_\domain(\vec x_i,\vec x_{i+1})\frac{\phi(\gscale)}{\gscale}.\] 
Substituting this above yields
\[\vec d_n(\vec x,\vec y) \geq \left(1-\frac{\phi(\gscale)}{\gscale} \right)\sum_{i=1}^{m-1} d_\domain(\vec x_{i},\vec x_{i+1}) \geq \left(1-\frac{\phi(\gscale)}{\gscale} \right)d_\domain(\vec x,\vec y),\] 
where we again used the triangle inequality, this time for $d_\domain$.
\end{proof}

We now give the proof of the upper bound.
\begin{lemma}\label{lem:cone_upper}
Let \cref{ass:kernel,ass:geodesic_euclidean_2} hold.
Assume that $\gscale\leq r_\domain$ and that
\begin{equation}\label{eq:delta_upper}
\frac{\res_n}{\gscale} \leq \frac{t_0}{2(2+\frac{\phi(\res_n)}{\res_n})}.
\end{equation}
Then for any $\vec x,\vec y\in \domain_n$ we have
\begin{equation}\label{eq:cone_upper}
\vec d_n(\vec x,\vec y) \leq \left(1 + \frac{4\res_n}{t_0\gscale} + \frac{2\phi(\res_n)}{t_0\gscale}\right) d_\domain(\vec x,\vec y) + \tau_\eta h,
\end{equation}
where
\begin{equation}\label{eq:ry}
\tau_\eta := \revision{\sup_{0 < t \leq t_0}}\left\{\sigma_\eta \eta(t)^{-1} - t\right\}.
\end{equation}
\end{lemma}
\begin{remark}\label{rem:}
We note that since $\eta$ is nonincreasing, for $t\leq t_0$ we have 
\[\sigma_\eta \eta(t)^{-1}  \leq \sigma_\eta \eta(t_0)^{-1} = \sigma_\eta t_0 \sigma_\eta^{-1} = t_0.\]
Therefore, $\tau_\eta\leq t_0 \leq 1$. 
In the special case of the singular kernel $\eta(t) = t^{-1}$ we have $t\eta(t) = 1$, and so $\sigma_\eta = 1$. 
Hence, it holds $\sigma_\eta \eta(t)^{-1} = t$ and $\tau_\eta = 0$. 

In fact, we can show that if $\eta$ is any kernel satisfying $\tau_\eta=0$ then $\eta(t)=\sigma_\eta t^{-1}$ for $t\in (0,t_0]$. To see this, note that if $\tau_\eta=0$ then $\sigma_\eta \leq t\eta(t)$ for all $0 < t \leq t_0$. Since $\sigma_\eta = \revision{\sup_{0< t\leq 1}} t\eta (t) \geq t\eta(t)$, we have $\sigma_\eta = t\eta(t)$ for all $0 < t \leq t_0$, which establishes the claim.

\cref{fig:cone_sim} illustrates the improved approximation accuracy obtained by using the singular kernel $\eta(t)=t^{-1}$, compared to the uniform kernel $\eta=\one_{[0,1]}$. For the uniform kernel, the graph cone gives a roughly piecewise constant staircasing approximation of the true Euclidean cone, with steps of size $O(h)$, due to the additional $\tau_\eta h$ term in \cref{lem:cone_upper}.  These artifacts are removed by using the singular kernel $\eta(t)=t^{-1}$.
\end{remark}

\revision{
\begin{remark}
The condition in \labelcref{eq:delta_upper} is a consequence of \cref{ass:scaling}.
Since $\sigma_\phi(\gscale)\geq\frac{\phi(\gscale)}{\gscale}\geq 0$, the assumption implies
\begin{align*}
\frac{\res_n}{\gscale} + \frac{\phi(\gscale_n)}{\gscale} \leq \frac{t_0}{4}\left(1-2\frac{\gscale}{\nlscale}\right)
\end{align*}
and in particular $\res_n < \gscale$. Therefore, it holds $\phi(\res_n)\leq\sigma_\phi(\gscale)\res_n\leq\sigma_\phi(\gscale)\gscale$ and we can estimate, using \cref{ass:scaling} again,
\begin{align*}
\frac{\res_n}{\gscale}\left(2 + \frac{\phi(\res_n)}{\res_n}\right)
=
2\frac{\res_n}{\gscale} + \frac{\phi(\res_n)}{\gscale}\leq
2\left(\frac{\res_n}{\gscale} + \sigma_\phi(\gscale)\right)\leq
\frac{t_0}{2}\left(1-2\frac{\gscale}{\nlscale}\right)\leq 
\frac{t_0}{2}
\end{align*}
and with this
\begin{align*}
\frac{\res_n}{\gscale} \leq \frac{t_0}{2 (2+\frac{\phi(\res_n}{\res_n})}.
\end{align*}
\end{remark}
}

\begin{figure}[!t]
\centering
\subfigure[$\eta=1_{[0,1]}$]{\includegraphics[clip=true,trim=140 70 130 95,width=0.4\textwidth]{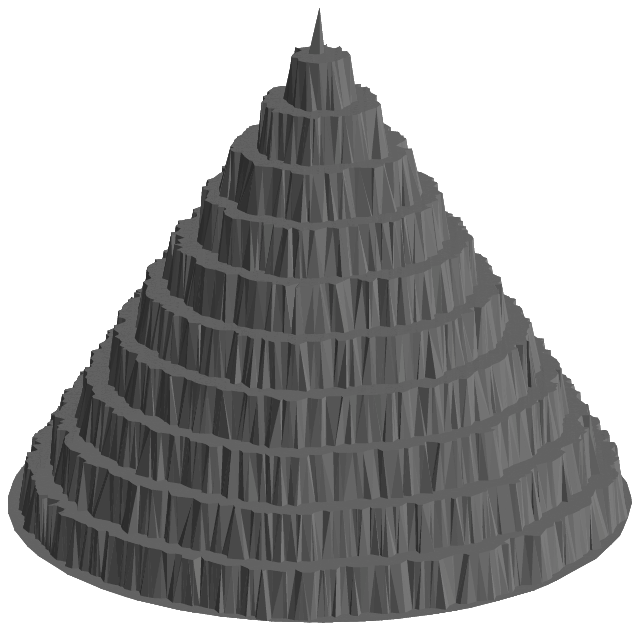}}\hspace{1cm}
\subfigure[$\eta(t)=t^{-1}$]{\includegraphics[clip=true,trim=140 70 130 90,width=0.4\textwidth]{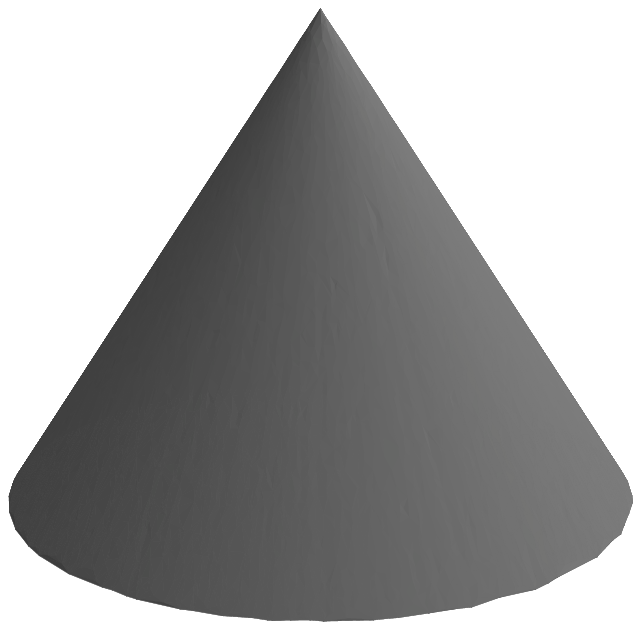}}
\caption{Examples of graph cones (graph distance functions to a point) computed with (a) the nonsingular kernel $\eta = \one_{[0,1]}$ and (b) the singular kernel $\eta(t)=t^{-1}$. The graph consists of an \emph{i.i.d.}~sample of size $n=10^4$ uniformly distributed on the unit ball $B(0,1)$, and we set $\gscale=0.1$.}
\label{fig:cone_sim}
\end{figure}

\begin{proof}
Let $\vec x,\vec y\in \domain_n$. We construct a sequence $\vec x_0,\vec x_1,\vec x_2,\dots$ as follows. Define $\vec x_0=\vec x$ and for $k\geq 1$ set 
\[ \vec x_k \in \argmin_{\domain_n \cap \closure B(x_{k-1},t_0\gscale)} d_\domain(\cdot,\vec y).\]
If $\vec y \in \closure B(\vec x_{k-1},t_0\gscale)$ then we set $\vec x_{k} = \vec y$ and the procedure terminates. We claim that whenever  $\vec y \not\in \closure B(\vec x_{k-1},t_0\gscale)$ we have
\begin{equation}\label{eq:dist_dec}
d_\domain(\vec x_k,\vec y) \leq d_\domain(\vec x_{k-1},\vec y) - (t_0h - 2\res_n - \phi(\res_n)).
\end{equation}
To see this, we start from the dynamic programming principle
\begin{align*}
    d_\domain(\vec x_{k-1},\vec y) = \min_{z\in \closure \domain \cap \partial B(\vec x_{k-1},t_0\gscale-\res_n)} \left\{ d_{\domain}(\vec x_{k-1},z) + d_\domain(z,\vec y)\right\},
\end{align*}
which is valid since \labelcref{eq:delta_upper} implies that $\res_n < t_0h$.  Choosing an optimal $z$ above yields
\begin{equation}\label{eq:dec1}
d_\domain(z,\vec y) \leq d_\domain(\vec x_{k-1},\vec y) -  t_0h + \res_n. 
\end{equation}
where we used that $d_{\domain}(\vec x_{k-1},z)\geq |\vec x_{k-1}-z| = t_0h - \res_n$.
We also have by \labelcref{eq:graph_res} that
\[|\pi_n(z)  - \vec x_{k-1}| \leq |\pi_n(z) - z| + |z - \vec x_{k-1}| \leq \res_n + t_0h - \res_n = t_0h.\]
Thus $\pi_n(z) \in \domain_n \cap \closure B(x_{k-1},t_0\gscale)$.  
Furthermore, by \cref{ass:geodesic_euclidean_2} and \labelcref{eq:dec1}, $\pi_n(z)$ satisfies
\[d_\domain(\pi_n(z),\vec y) \leq  d_\domain(z,\vec y) + d_\domain(\pi_n(z),z)\leq d_\domain(\vec x_{k-1},\vec y) -  t_0h +\res_n + \res_n + \phi(\res_n).\]
This establishes the claim \labelcref{eq:dist_dec}. 
Let us define
\[r_0 := t_0 - 2\frac{\res_n}{\gscale} - \frac{\phi(\res_n)}{\gscale}, \]
and note that \labelcref{eq:delta_upper} implies that $\res_n \leq \gscale$ and so 
\begin{equation}\label{eq:r0t0}
1 \geq  \frac{r_0}{t_0} \geq 1 -  2\frac{\res_n}{t_0h} -  \frac{\phi(\res_n) }{t_0h}= 1 - \frac{\res_n}{t_0h}\left( 2 + \frac{\phi(\res_n)}{\res_n}\right) \geq \frac{1}{2},
\end{equation}
where we again used \labelcref{eq:delta_upper} in the last inequality.
Therefore $r_0>0$ and so by  \labelcref{eq:dist_dec}  we have
\[d_\domain(\vec x_k,\vec y) \leq d_\domain(\vec x_{k-1},\vec y) - r_0\gscale\]
and hence
\[d_\domain(\vec x_k,\vec y) \leq d_\domain(\vec x,\vec y) - r_0\gscale k,\]
provided $\vec y \not\in \closure B(\vec x_{j},t_0\gscale)$ for all $j=0,\dots,k-1$. Hence, the condition $\vec y \not\in \closure B(\vec x_{j},t_0\gscale)$ for all $j=0,\dots,k-1$ implies that
\begin{equation}\label{eq:Tbound}
k \leq \frac{d_\domain(\vec x,\vec y) - d_\domain(\vec x_k,\vec y)}{r_0\gscale}.
\end{equation}
It follows that eventually $\vec y \in \closure B(\vec x_{k-1},t_0\gscale)$ and $\vec x_k = \vec y$ for some $k$. Let $T\geq 1$ denote the smallest integer such that $\vec x_T = \vec y$, and note that  \labelcref{eq:Tbound} implies 
\begin{equation}\label{eq:Tbound2}
T \leq \frac{d_\domain(\vec x,\vec y) - d_\domain(\vec x_{T-1},\vec y)}{r_0\gscale} + 1.
\end{equation}
We now compute
\begin{align*}
 \vec d_n(\vec x,\vec y) &\leq \sum_{k=1}^{T} \vec w_n(\vec x_{k-1},\vec x_{k})^{-1}
 = \sum_{k=1}^{T} \sigma_\eta \eta_\gscale(|\vec x_{k-1}-\vec x_{k}|)^{-1}
 = \sum_{k=1}^{T} \gscale\sigma_\eta \eta\left(\frac{|\vec x_{k-1}-\vec x_{k}|}{\gscale}\right)^{-1}.
\end{align*}
Since $|\vec x_{k-1}-\vec x_k| \leq t_0\gscale$ and $\eta$ is nonincreasing we have
\[\eta\left(\frac{|\vec x_{k-1}-\vec x_{k}|}{\gscale}\right)^{-1} \leq \eta\left(t_0\right)^{-1}.\]
It follows that
\begin{align}\label{eq:dn}
\vec d_n(\vec x,\vec y) &\leq \sum_{k=1}^{T-1} \gscale \sigma_\eta \eta(t_0)^{-1} +  \gscale\sigma_\eta \eta\left( \frac{|\vec x_{T-1} - \vec y|}{\gscale}\right)^{-1}\notag\\
&=\sum_{k=1}^{T-1} t_0\gscale  + \gscale\sigma_\eta \eta\left( \frac{|\vec x_{T-1} - \vec y|}{\gscale}\right)^{-1}\notag\\
&= t_0\gscale (T-1) + \gscale\sigma_\eta \eta\left( \frac{|\vec x_{T-1} - \vec y|}{\gscale}\right)^{-1}\notag\\
&\leq t_0\gscale \left( \frac{d_\domain(\vec x,\vec y) - d_\domain(\vec x_{T-1},\vec y)}{r_0\gscale}\right) + \gscale\sigma_\eta \eta\left( \frac{|\vec x_{T-1} - \vec y|}{\gscale}\right)^{-1}\notag\\
&\leq \frac{t_0}{r_0}d_\domain(\vec x,\vec y)  + \gscale\sigma_\eta \eta\left( \frac{|\vec x_{T-1} - \vec y|}{\gscale}\right)^{-1} - |\vec x_{T-1}-\vec y|\\
&\leq \frac{t_0}{r_0}d_\domain(\vec x,\vec y)  + \tau_\eta h,
\end{align}
where we used that $t_0\geq r_0$ and $d_\domain(\vec x_{T-1},\vec y) \geq |\vec x_{T-1}-\vec y|$ in the second to last inequality, and $\vec x_{T-1}\in \closure B(\vec y,t_0\gscale)$ in the last. By \labelcref{eq:r0t0} we have
\[\frac{t_0}{r_0} = \frac{1}{1 - \left( 1 - \frac{r_0}{t_0}\right)} \leq 1 + 2\left( 1 - \frac{r_0}{t_0}\right) = 1 + \frac{4\res_n}{t_0\gscale} + \frac{2\phi(\res_n)}{t_0\gscale},\]
which completes the proof.
\end{proof}

\subsection{Convergence of Infinity Harmonic Functions}
\label{sec:cvgc_inf_harm}

In the previous sections we have analyzed the continuum problem \labelcref{eq:continuum_problem}, introduced the corresponding discrete problem on a graph, and proved convergence of graph cone functions to continuum Euclidean cone functions. 
These are all the ingredients which we need in order to finally show that the solution of \labelcref{eq:graph_inflap} converges to the solution of \labelcref{eq:continuum_problem} and establish convergence rates.

For this we need to introduce discrete analogies of the operators $T^\nlscale$ and $T_\nlscale$ from \labelcref{eq:T-operators}.
For a graph function $\vec u_n:\domain_n\to\R$ we define $u_n^\nlscale,(u_n)_\nlscale:\closure\domain\to\R$ by
\begin{align}\label{eq:discr_extension}
    u_n^\nlscale(x):=\sup_{\closure B_\domain(x,\nlscale)\cap\domain_n}\vec u_n,\quad
    (u_n)_\nlscale(x):=\inf_{\closure B_\domain(x,\nlscale)\cap\domain_n}\vec u_n,\quad
    x\in\closure\domain.
\end{align}%
Before we proceed with a discrete-to-nonlocal consistency statement and prove \cref{thm:general_convergence_result}, we devote a section to collecting several technical lemmas, dealing with (approximate) Lipschitz continuity of the quantities involved.
For understanding the gist of our arguments the following section can be skipped, however. 
\subsubsection{Approximate Lipschitz Estimates}
\label{sec:lipschitz_lemmas}
\begin{lemma}\label{lem:lipschitz_NLInfL}
Under \cref{ass:kernel,ass:geodesic_euclidean_2,ass:scaling} there exists a constant $C>0$ such that for all $\vec u_n:\domain_n\to\R$ and all $x,y\in\closure\domain$ it holds
\begin{align}
    \abs{\Delta_\infty^\nlscale u_n^\nlscale(x) - \Delta_\infty^\nlscale u_n^\nlscale(y)} \leq \Lip_n(\vec u_n)C\frac{\abs{x-y}+\phi(\abs{x-y})+2\res_n+\phi(2\res_n)+\tau_\eta\gscale}{\nlscale^2}.
\end{align}
\end{lemma}
\begin{proof}
We first note that by definition \labelcref{eq:nonlocal_lapl}
\begin{align*}
    \nlscale^2\abs{\Delta_\infty^\nlscale u_n^\nlscale(x)-\Delta_\infty^\nlscale(y)} 
    &\leq 
    \abs{T^\nlscale u_n^\nlscale(x) - T^\nlscale u_n^\nlscale(y)}
    +
    \abs{T_\nlscale u_n^\nlscale(x) - T_\nlscale u_n^\nlscale(y)}
    +
    2\abs{u_n^\nlscale(x)-u_n^\nlscale(y)}
    \\
    &=
    \abs{u_n^{2\nlscale}(x) - u_n^{2\nlscale}(y)}
    +
    \abs{T_\nlscale u_n^\nlscale(x) - T_\nlscale u_n^\nlscale(y)}
    +
    2\abs{u_n^\nlscale(x)-u_n^\nlscale(y)}.
\end{align*}
We first estimate the third term and then argue that the same estimate also applies to the first and second.

For $x\in \revision{\closure\domain}$ we know that there exists a point $\vec{x}^\ast\in \closure B_\domain(x,\nlscale)\cap\domain_n$ such that
\begin{align*}
u_n^\nlscale(x) = \vec{u}_n(\vec{x}^\ast).
\end{align*}
For $y\in\revision{\closure\domain}$ we construct a point $\tilde{y}\in B_\domain(y,\nlscale)$ as follows. 
Since $(\domain,d_\domain)$ is a length space we can find a geodesic $\gamma:[0,1]\to\domain$ with $\gamma(0)=y$, $\gamma(1)=\vec x^\ast$, and length $d_\domain(y,\vec x^\ast)$.
We define $\tilde y:=\gamma(t_*)$ where $t_*:=\sup\{t>0\st\gamma(t)\in B_\domain(y,\nlscale)\}$ as the last point which still lies in $\closure B_\domain(y,\nlscale)$.
We distinguish two cases:
If $\vec x^\ast\in\closure B_\domain(y,\nlscale)$ then $\tilde y=\vec x^\ast$ and hence $d_\domain(\tilde y,\vec x^\ast)=0\leq d_\domain(x,y)$.
If $\vec x^\ast\notin\closure B_\domain(y,\nlscale)$ then it holds $d_\domain(y,\tilde y)=\nlscale$.
Furthermore, since $\gamma$ is a geodesic it holds
\begin{align*}
    d_\domain(y,\vec x^\ast) = d_\domain(y,\tilde y) + d_\domain(\tilde y,\vec x^\ast).
\end{align*}
Hence, we obtain that also in this case it holds
\begin{align*}
    d_\domain(\tilde y,\vec x^\ast) 
    =
    d_\domain(y,\vec x^\ast) -
    d_\domain(y,\tilde y)
    \leq 
    d_\domain(x,y) + d_\domain(x,\vec x^\ast) - d_\domain(y,\tilde y)
    \leq
    d_\domain(x,y).
\end{align*}
By definition of the graph resolution $\res_n$ there exists a point $\vec{y}^\ast\in \closure B_\domain(y,\nlscale)\cap\domain_n$, such that $\abs{\tilde{y} - \vec{y}^\ast} \leq 2\res_n$ and by definition of $u_n^\nlscale$ it holds $\vec{u}_n(\vec{y}^\ast) \leq u_n^\nlscale(y)$.
Furthermore, by \cref{ass:geodesic_euclidean_2}
\begin{align*}
d_\domain(\vec{x}^\ast, \vec{y}^\ast) &\leq
d_\domain(\tilde{y},\vec{x}^\ast) + d_\domain(\tilde{y}, \vec{y}^\ast)
\leq
d_\domain(x,y) + 2\res_n + \phi(2\res_n).
\end{align*}
Thus, using \cref{lem:cone_upper,ass:geodesic_euclidean_2,ass:scaling} it follows that
\begin{align*}
u_n^\nlscale(x) - u_n^\nlscale(y) &\leq 
\vec{u}_n(\vec{x}^\ast) - \vec{u}_n(\vec{y}^\ast)
\\&\leq
\Lip_n(\vec{u}_n)\,\vec{d}_n(\vec{x}^\ast,\vec{y}^\ast)
\\&\leq
\Lip_n(\vec{u}_n)\,
\big(
C\,d_{\domain}(\vec{x}^\ast,\vec{y}^\ast) + \tau_\eta\gscale
\big)
\\&\leq
\Lip_n(\vec{u}_n)\,
\big(C\,(
d_\domain(x,y) + 2\res_n + \phi(2\res_n)
+\tau_\eta\gscale
\big)
\\&\leq
\Lip_n(\vec{u}_n)\,C\,
\big(
\abs{x-y} + \phi(\abs{x-y}) + 2\res_n + \phi(2\res_n)
+
\tau_\eta\gscale
\big).
\end{align*}
Exchanging the role of $x$ and $y$ we can estimate the difference $u_n^\nlscale(y) - u_n^\nlscale(x)$.
The resulting estimate on $\abs{u_n^\nlscale(x)-u_n^\nlscale(y)}$ does not depend on the choice of $\nlscale$ and is therefore also valid for $\abs{u_n^{2\nlscale}(x)-u_n^{2\nlscale}(y)}$.
Finally, we can estimate
\begin{align*}
    T_\nlscale u_n^\nlscale(x)-T_\nlscale u_n^\nlscale(y)
    &=
    \inf_{B_\domain(x,\nlscale)} u_n^\nlscale - \inf_{B_\domain(y,\nlscale)} u_n^\nlscale
    \\
    &=
    \inf_{B_\domain(x,\nlscale)} u_n^\nlscale 
    -
    u_n^\nlscale(y^\ast),
\end{align*}
where $y^\ast\in \closure B_\domain(y,\nlscale)$ realizes the infimum.
With the same trick as above we can choose a geodesic $\gamma:[0,1]\to\domain$ from $y^\ast$ to $x$ and define the first point lying in $B_\domain(x,\nlscale)$ as $x^\ast:=\gamma(t^\ast)$ where
\begin{align*}
    t^\ast := \inf\{t>0\st \gamma(t)\in B_\domain(x,\nlscale)\}.
\end{align*}
Either $y^\ast\in B_\domain(x,\nlscale)$ and therefore $x^\ast=y^\ast$ or, similar as before, it holds $d_\domain(x,x^\ast)=\nlscale$ and
\begin{align*}         
    d_\domain(y^\ast,x^\ast)
    =
    d_\domain(y^\ast,x) 
    -
    d_\domain(x^\ast,x)
    \leq 
    d_\domain(x,y)
    +
    d_\domain(y^\ast,y)
    -
    d_\domain(x^\ast,x)
    \leq d_\domain(x,y).
\end{align*}
Hence, as before we can estimate
\begin{align*}
    T_\nlscale u_n^\nlscale(x) - T_\nlscale u_n^\nlscale(y) 
    &\leq
    u_n^\nlscale(x^\ast) - u_n^\nlscale(y^\ast)
    \\&\leq
    \Lip_n(\vec{u}_n)\,
    \big(C\,(
    \underbrace{d_\domain(x^\ast,y^\ast)}_{\leq d_\domain(x,y)} 
    + 2\res_n + \phi(2\res_n)
    +\tau_\eta\gscale
    \big)
    \\&\leq
    \Lip_n(\vec{u}_n)\,C\,
    \big(
    \abs{x-y} + \phi(\abs{x-y}) + 2\res_n + \phi(2\res_n)
    +
    \tau_\eta\gscale
    \big).
\end{align*}
\end{proof}
\begin{lemma}\label{lem:ball_to_func_cont}
For $u:\closure\domain\to\R$ it holds
\begin{align}
\sup_{\closure\domain}\abs{T^\nlscale u-u} 
\vee 
\sup_{\closure\domain}\abs{T_\nlscale u-u}
\leq \Lip_\domain(u)\,\nlscale.
\end{align}
\end{lemma}
\begin{proof}
We only give a proof of the first inequality since the one for $T_\nlscale u$ works analogously.
For $x\in\closure\domain$ one computes
\begin{align*}
    T^\nlscale u (x) - u(x)
    = 
    \max_{y\in B_\domain(x,\nlscale)} u(y) - u(x)
    \leq \Lip_\domain(u)\max_{y\in B_\domain(x,\nlscale)}d_\domain(x,y) \leq \Lip_\domain(u)\,\nlscale,\quad\forall x\in\domain^\nlscale_\constr,
\end{align*}
which implies $\sup_{\domain^\nlscale_\constr}\abs{T^\nlscale-u}\leq\Lip_\domain(u)\,\nlscale$.
\end{proof}
\begin{lemma}\label{lem:ball_to_func_disc}
Under \cref{ass:kernel,ass:geodesic_euclidean_2,ass:scaling} there \revision{exists} a constant $C>0$ such that for all $\vec{u}_n:\domain_n\to\R$ it holds
\begin{align}
\sup_{\closure\domain}\abs{u_n^\nlscale-u_n} \vee \sup_{\closure\domain}\abs{(u_n)_\nlscale-u_n}
\leq C\Lip_n(\vec u_n)\nlscale,
\end{align}
where $u_n:\closure\domain\to\R$ is the piecewise constant extension of $\vec u_n$, defined in \labelcref{eq:extension}.
\end{lemma}
\begin{proof}
Again we only prove the first estimate.
For every $x\in\domain$ it holds thanks to \cref{lem:cone_upper,ass:geodesic_euclidean_2,ass:scaling} that
\begin{align*}
u_n^\nlscale(x) - u_n(x)
&= 
\max_{\closure B_\domain(x,\nlscale)\cap\domain_n}\vec u_n - \vec u_n(\pi_n(x)) 
\leq 
\Lip_n(\vec u_n)
\max_{\vec y\in \closure B_\domain(x,\nlscale)\cap\domain_n} 
\vec d_n(\pi_n(x),\vec y)
\\
&\leq
\Lip_n(\vec u_n)
\max_{\vec y\in \closure B_\domain(x,\nlscale)\cap\domain_n} 
\left(
C d_\domain(\pi_n(x),\vec y) + \tau_\eta\gscale 
\right)
\\
&\leq
\Lip_n(\vec u_n)
\max_{\vec y\in \closure B_\domain(x,\nlscale)\cap\domain_n} 
\left(
C
(d_\domain(x,\vec y) + d_\domain(x,\pi_n(x))+ \tau_\eta \gscale 
\right)
\\
&\leq
\Lip_n(\vec u_n)
\left(
C
(\nlscale + \res_n + \phi(\res_n)) + \tau_\eta \gscale 
\right)
\\
&
\leq C\Lip_n(\vec u_n)\,\nlscale,
\end{align*}
where the constant $C>0$ changed its value.
\end{proof}
\begin{lemma}\label{lem:bdry_term}
Under \cref{ass:kernel,ass:geodesic_euclidean_2,ass:scaling,ass:data} there exists a constant $C>0$ such that for all $\vec{u}_n:\domain_n\to\R$ with $\vec u_n=\vec g_n$ on $\constr_n$ and all $u:\closure\domain\to\R$ with $u=g$ on $\constr$ it holds
\begin{align*}
\sup_{\domain^\nlscale_\constr\setminus\domain^{2\nlscale}_\constr} (u_n^\nlscale - T_\nlscale u) \leq
C(\Lip_n(\vec u_n)+\Lip_\domain(u))\nlscale.
\end{align*}
\end{lemma}
\begin{proof}
Using \cref{lem:ball_to_func_cont,lem:ball_to_func_disc} we have that
\begin{align*}
\sup_{\domain^\nlscale_\constr\setminus\domain^{2\nlscale}_\constr} \left(u_n^\nlscale - T_\nlscale u\right) &\leq
\sup_{\domain^\nlscale_\constr\setminus\domain^{2\nlscale}_\constr}
\big(\abs{u_n^\nlscale - u_n} + 
\abs{u_n - u} +
\abs{u - T_\nlscale u}
\big)\\
&\leq
(\Lip_n(\vec u_n) + \Lip_\domain(u)) C \nlscale + 
\sup_{\domain^\nlscale_\constr\setminus\domain^{2\nlscale}_\constr}
\abs{u_n - u}.
\end{align*}
Since \revision{$x\in \domain^\nlscale_\constr\setminus\domain^{2\nlscale}_\constr$} we know that there exists 
$z\in \constr$ such that $\abs{z  - x}\leq 2\nlscale$, for which we also find a vertex $\pi_{\constr_n}(z)\in\constr_n$ with $\abs{z - \pi_{\constr_n}(z)}\leq \res_n$ and therefore $\abs{x - \pi_{\constr_n}(z)}\leq 2\nlscale+\res_n.$
This also implies that $\abs{\pi_n(x) - \pi_{\constr_n}(z)}\leq 2\nlscale+2\res_n$.
With similar computations as before, using \cref{ass:scaling,lem:cone_upper,lem:ball_to_func_cont,lem:ball_to_func_disc,ass:data}, we have
\begin{align*}
\abs{u_n(x) - u(x)} 
&\leq 
\abs{\vec u_n(\pi_n(x)) - \vec g_n(\pi_{\constr_n}(z))} +
\abs{\vec g_n(\pi_{\constr_n}(z)) - g(z)} +
\abs{g(z) - u(x)}
\\
& = 
\abs{\vec u_n(\pi_n(x)) - \vec u_n(\pi_{\constr_n}(z))} +
\abs{\vec g_n(\pi_{\constr_n}(z)) - g(z)}
+
\abs{u(z) - u(x)}
\\
&\leq
\Lip_n(\vec u_n)\vec d_n(\pi_n(x),\pi_{\constr_n}(z)) 
+
\abs{\vec g_n(\pi_{\constr_n}(z)) - g(z)}
+
\Lip_\domain(u)d_\domain(z,x)
\\
&\leq
C(\Lip_n(\vec u_n)+\Lip_\domain(u))\nlscale.
\end{align*}
\end{proof}
\begin{lemma}\label{lem:lipschitz_un}
Under \cref{ass:kernel,ass:geodesic_euclidean_2,ass:scaling} there exists a constant $C>0$ such that for all $\vec u:\domain_n\to\R$ it holds
\begin{align*}
    \abs{u_n(x) - u_n(y)}\leq
    \Lip_n(\vec u_n)
    \left(
    C\big(d_\domain(x,y)+2\res_n+2\phi(\res_n)\big)+
    \tau_\eta \gscale\right),\quad\forall x,y\in\closure\domain,
\end{align*}
where $u_n:\closure\domain\to\R$ is the piecewise constant extension of $\vec u_n$, defined in \labelcref{eq:extension}.
\end{lemma}
\begin{proof}
Using \cref{lem:cone_upper,ass:geodesic_euclidean_2,ass:scaling} and the definition of $u_n$ in \labelcref{eq:extension} it holds
\begin{align*}
    \abs{u_n(x)-u_n(y)} 
    &= 
    \abs{\vec u_n(\pi_n(x)) - \vec u_n(\pi_n(y))} 
    \\
    &\leq 
    \Lip_n(\vec u_n) \vec d_n(\pi_n(x),\pi_n(y))
    \\
    &\leq
    \Lip_n(\vec u_n)
    \left(
    C\,d_\domain(\pi_n(x),\pi_n(y))+
    \tau_\eta \gscale\right)
    \\
    &\leq
    \Lip_n(\vec u_n)
    \left(
    C\,\big(d_\domain(\pi_n(x),x)+d_\domain(x,y)+d_\domain(y,\pi_n(y)\big)+
    \tau_\eta \gscale\right)
    \\
    &\leq
    \Lip_n(\vec u_n)
    \left(
    C\,\big(d_\domain(x,y)+2\res_n+2\phi(\res_n)\big)+
    \tau_\eta \gscale\right).
\end{align*}
\end{proof}
%
%
\subsubsection{Discrete-to-Nonlocal Consistency}

In this section we will prove that if $\vec u_n$ is a solution of the graph infinity Laplace equation \labelcref{eq:discrete_problem} then a discrete-to-continuum version of the max-ball statement \cref{lem:nonloc2loc} holds true.
Using only the comparison with graph distance functions, established in \cref{thm:comp_graph_cones}, we shall prove
\begin{align*}
    -\Delta_\infty^\nlscale 
    u_n^\nlscale
    \leq 
    C(\res_n,\gscale,\nlscale)
    \quad\text{and}\quad
    -\Delta_\infty^\nlscale (u_n)_\nlscale
    \geq -C(\res_n,\gscale,\nlscale)
    ,
\end{align*}
where $C(\res_n,\gscale,\nlscale)>0$ is small if $\res_n$ and $\gscale$ are small.
This means that the extension operators \labelcref{eq:discr_extension} turn discrete solutions into approximate sub- and supersolutions of the nonlocal equation $-\Delta_\infty^\nlscale u=0$.
Compared to the continuum setting from \cref{lem:nonloc2loc}, the inhomogeneity $C(\res_n,\gscale,\nlscale)$ originates from the discretization error.
Using the perturbation results from \cref{sec:max_ball_perturb} we can then derive uniform convergence rates of $\vec u_n$ to an AMLE \labelcref{eq:AMLE}.

In the proof we will first show the desired statement for all graph vertices and then use an approximate Lipschitz property to extend it to the whole continuum domain.
\begin{theorem}\label{thm:discrete_NL}
Assume that \cref{ass:kernel,ass:geodesic_euclidean_2,ass:scaling} hold.
Let $\vec u_n:\domain_n\to\R$ solve the graph infinity Laplacian equation \labelcref{eq:discrete_problem}.
Then there exists a constant $C>0$ such that for all $x_0\in\domain_\constr^{2\nlscale+3\res_n}$ it holds
\begin{subequations}\label{ineq:consist_bounds}
\begin{align}
    \label{ineq:bound_subharmonic}
    -\Delta_\infty^\nlscale u_n^\nlscale(x_0) 
    &\leq 
    \Lip_n(\vec g_n)C\left(\frac{\res_n}{\gscale\nlscale} + \frac{\gscale}{\nlscale^2}
    +
    \frac{\phi(\gscale)}{\gscale\nlscale}
    \right),
    \\
    \label{ineq:bound_superharmonic}
    -\Delta_\infty^\nlscale (u_n)_\nlscale(x_0) 
    &\geq
    \revision{-}
    \Lip_n(\vec g_n)C\left(\frac{\res_n}{\gscale\nlscale} + \frac{\gscale}{\nlscale^2}
    +
    \frac{\phi(\gscale)}{\gscale\nlscale}
    \right).
\end{align}
\end{subequations}
\end{theorem}
\begin{proof}
We only proof the first inequality \labelcref{ineq:bound_subharmonic}.
The second one can be established by applying the first one to $-\vec u_n$ and $-\vec g_n$.
Our proof strategy consists in proving the desired inequality for points $\vec x_0\in\domain_\constr^{2\nlscale+2\res_n}\cap\domain_n$ and then extending this to the whole domain by the help of \cref{lem:lipschitz_NLInfL} which gives another $\res_n$ contribution and leads to $\domain_\constr^{2\nlscale+3\res_n}$.

By definition of the nonlocal infinity Laplacian \labelcref{eq:nonlocal_lapl} it holds
\begin{align*}
    -\nlscale^2\Delta^\nlscale_\infty u_n^\nlscale(\vec x_0) = 
    2u_n^\nlscale(\vec x_0) - \sup_{y\in\closure B_\domain(\vec x_0,\nlscale)} u_n^\nlscale(y)
    - \inf_{y\in\closure B_\domain(\vec x_0,\nlscale)} u_n^\nlscale(y).
\end{align*}
For the $\sup$ term we have
\begin{align*}
\sup_{y\in\closure B_\domain(\vec x_0,\nlscale)} u_n^\nlscale(y) = 
\sup_{y\in\closure B_\domain(\vec x_0,\nlscale)}
\;\,
\sup_{\closure B_\domain(y,\nlscale)\cap\domain_n} \vec u_n
=
\sup_{\closure B_\domain(\vec x_0,2\nlscale)\cap\domain_n} \vec u_n
= u_n^{2\nlscale}(\vec x_0).
\end{align*}
For the inf term, using the fact that $u_n^\nlscale$ is a simple function, we find $y^\ast\in \closure B_\domain(\vec x_0,\nlscale)$ such that
\begin{align*}
\inf_{y\in\closure B_\domain(\vec x_0,\nlscale)} u_n^\nlscale(y) 
&= 
u_n^\nlscale(y^\ast)
= 
\sup_{\closure B_\domain(y^\ast,\nlscale)\cap\domain_n} \vec u_n 
\geq \vec u_n(x_0)
\end{align*}
and therefore
\begin{align*}
-\nlscale^2\Delta^\nlscale_\infty u_n^\nlscale(\vec x_0)
\leq 2 u_n^\nlscale (\vec x_0) - u_n^{2\nlscale}(\vec x_0) - 
\vec{u}_n (\vec x_0).
\end{align*}
Denoting by $\vec{p}_n^\nlscale(\vec x_0)\in\closure B_\domain(\vec x_0,\nlscale) \cap\domain_n$ and $\vec{p}_n^{2\nlscale}(\vec x_0)\in \closure B_\domain(\vec x_0,2\nlscale) \cap\domain_n$ two vertices such that
\begin{align*}
u_n^\nlscale(\vec x_0)
&=
\sup_{\closure B_\domain(\vec x_0,\nlscale) \cap \domain_n} \vec{u}_n 
= 
\vec u_n (\vec{p}_n^\nlscale(\vec x_0)),
\\
u_n^{2\nlscale}(\vec x_0)
&=
\sup_{\closure B_\domain(\vec x_0,2\nlscale) \cap \domain_n} \vec{u}_n 
= 
\vec u_n (\vec{p}_n^{2\nlscale}(\vec x_0)),
\end{align*}
we rewrite this inequality to 
\begin{align}\label{ineq:estimate_nonl_inf_lapl}
-\nlscale^2\Delta^\nlscale_\infty u_n^\nlscale(\vec x_0)
\leq 2\vec{u}_n(\vec{p}_n^\nlscale(\vec x_0)) - 
\vec{u}_n(\vec{p}_n^{2\nlscale}(\vec x_0)) - 
\vec{u}_n (\vec x_0).
\end{align}
In order to estimate this term we have to use the comparison with cones property, which $\vec u_n$ satisfies thanks to \cref{thm:comp_graph_cones}.
To this end we define the following subset of vertices
\begin{align*}
    B:=\left\{\vec w\in\domain_n\setminus\{\vec x_0\}\st \vec d_n(\vec x_0,\vec w)\leq 2\nlscale\left(1-\frac{\phi(\gscale)}{\gscale}\right) - \gscale\right\}.
\end{align*}
First note that $B$ is non-empty.
To see this, we know that by definition of the graph resolution $\res_n$ in \labelcref{eq:graph_res} there exists $\vec y\in\domain_n\setminus\{\vec x_0\}$ such that $\abs{\vec x_0-\vec y}\leq 2\res_n$.
\revision{Thanks to \cref{ass:scaling}} it trivially holds $2\res_n/\gscale\leq t_0$ and also $\phi(\gscale)/\gscale\leq 1/2$ and $(1+t_0)\frac{\gscale}{\nlscale}<1$, which is why we obtain
\begin{align*}
    \vec d_n(\vec x_0,\vec y) 
    &= 
    \frac{\sigma_\eta}{\eta_h(\abs{\vec x_0-\vec y})} 
    \leq \frac{\sigma_\eta\gscale}{\eta(2\res_n/\gscale)}
    \leq t_0\gscale
    \\
    &=
    2\nlscale\left(1-\frac{\phi(\gscale)}{\gscale}\right)-\gscale + t_0\gscale - 2\nlscale\left(1-\frac{\phi(\gscale)}{\gscale}\right) + \gscale
    \\
    &=
    2\nlscale\left(1-\frac{\phi(\gscale)}{\gscale}\right)-\gscale +
    \nlscale
    \left(
    (1+t_0)\frac{\gscale}{\nlscale} - 2 + 2\frac{\phi(\gscale)}{\gscale}
    \right)
    \\
    &\leq
    2\nlscale\left(1-\frac{\phi(\gscale)}{\gscale}\right)-\gscale,
\end{align*}
meaning that $\vec y\in B$.
We claim that the graph boundary of $B$ satisfies
\begin{align*}
    \partial B \subset \left\{\vec w\in\domain_n\st
    2\nlscale\left(1-\frac{\phi(\gscale)}{\gscale}\right) - \gscale
    < 
    \vec d_n(\vec x_0,\vec w)
    \;\;\text{and}
    \;\;
    d_\domain(\vec x_0,\vec w)
    \leq 
    2\nlscale
    \right\}
    \cup
    \{\vec x_0\}
    =:B'.
\end{align*}
By definition of the graph boundary \labelcref{eq:graph_boundary} it holds $\vec w\in\partial B$ if and only if $\vec w\notin B$ and there exists $\vec y\in B$ such that $\eta_h(\abs{\vec w-\vec y})>0$.
Hence, the only non-trivial property to check is the inequality $d_\domain(\vec x_0,\vec w)\leq 2\nlscale$. 
For this, we observe that $\vec y\in B$ as above satisfies $\abs{\vec w-\vec y}\leq \gscale$ and hence, using \cref{lem:cone_lower,ass:geodesic_euclidean_2}, we get
\begin{align*}
    d_\domain(\vec x_0,\vec w) 
    &\leq 
    d_\domain(\vec x_0,\vec y)
    +
    d_\domain(\vec y,\vec w)
    \leq
    \frac{\vec d_n(\vec x_0,\vec y)}{1-\frac{\phi(\gscale)}{\gscale}}
    +
    \abs{\vec y-\vec w}+\phi(\abs{\vec y-\vec w})
    \\
    &\leq
    2\nlscale - \frac{\gscale}{1-\frac{\phi(\gscale)}{\gscale}} 
    +
    \gscale + \phi(\gscale)
    =
    2\nlscale
    +
    \gscale\left(1+\frac{\phi(\gscale)}{\gscale}-\frac{1}{1-\frac{\phi(\gscale)}{\gscale}}\right)
    \leq 
    2\nlscale
\end{align*}
using that $1+x-\frac{1}{1-x}\leq 0$ for all $0\leq x<1$ and $\frac{\phi(\gscale)}{\gscale}\leq\frac{1}{2}<1$ by \cref{ass:scaling}.
With this we have established that $\partial B\subset B'$.
Next we claim that any vertex $\vec w\in B'$ satisfies the inequality
\begin{align}\label{ineq:comparison_with_cone}
\vec{u}_n(\vec w) \leq 
\vec{u}_n(\vec{x}_0) + 
\frac{\vec{u}_n(\vec p_n^{2\nlscale}(\vec x_0)) - \vec{u}_n(\vec{x}_0)}{2\nlscale\left(1-\frac{\phi(\gscale)}{\gscale}\right) - \gscale} 
\vec d_n(\vec x_0, \vec w).
\end{align}
On one hand, if $\vec w=\vec x_0$ then the inequality is trivially fulfilled.
On the other hand, for all other $\vec w\in B'$ by definition of $\vec p_n^{2\nlscale}(\vec x_0)$ it holds
\begin{align*}
    \vec u_n(\vec w) \leq \sup_{\closure B_\domain(\vec x_0,2\nlscale)\cap\domain_n}\vec u_n = \vec u_n(\vec p_n^{2\nlscale}(\vec x_0)),
    \quad\text{and}\quad
    \vec u_n (\vec x_0) \leq \vec u_n(\vec p_n^{2\nlscale}(\vec x_0)).
\end{align*}
Combining these two inequalities with the definition of $B'$ shows \labelcref{ineq:comparison_with_cone} for all $\vec w\in B'$ and in particular for all $\vec w\in\partial B$.
Before we are able to use comparison with cones we have to argue that $B'\subset\domain_n\setminus\constr_n$, i.e., $B'$ does not contain any labelled points.
Let therefore $\vec w\in B'$ be arbitrary.
Using the fact that $\dist_\domain(\vec x_0,\constr)\geq 2\nlscale+2\res_n$ it holds
\begin{align*}
    \dist_\domain(\vec w,\constr) \geq \dist_\domain(\vec x_0,\constr) - d_\domain(\vec x_0,\vec w) \geq 2\nlscale +2\res_n - 2\nlscale = 2\res_n.
\end{align*}
Since the graph resolution $\res_n$ is defined with respect to the Euclidean distance, we have to transfer this statement to the Euclidean distance.

Aiming for a contradiction we assume there exists $y\in\constr$ with $\abs{\vec w-y}\leq\res_n$.
By \cref{ass:geodesic_euclidean_2} it holds
\begin{align*}
    \abs{\vec w-y} \geq d_\domain(\vec w,y) - \phi(\abs{\vec w-y}) 
    \geq 
    2\res_n - \abs{\vec w-y}\frac{\phi(\abs{\vec w-y})}{\abs{\vec w-y}}.
\end{align*}
We observe that thanks to \cref{ass:scaling} it holds $\res_n\leq\gscale$
and also
\begin{align*}
    \frac{\phi(\abs{\vec w-y})}{\abs{\vec w-y}} \leq \frac{\phi(\res_n)}{\res_n} \leq
    \revision{\sigma_\phi(\gscale)}<1.
\end{align*}
Thus, we obtain
\begin{align*}
    \abs{\vec w-y} \geq 2\res_n-\abs{\vec w-y}\frac{\phi(\abs{\vec w-y})}{\abs{\vec w-y}}>2\res_n-\res_n=\res_n
\end{align*}
\revision{which is a contradiction.}
Therefore, we get that $\abs{\vec w-y}>\res_n$ for all $y\in\constr$ and, since $\constr$ is a compact set, it holds $\dist(\vec w,\constr)>\res_n$.
Next, we observe that it holds
\begin{align*}
    \dist(\vec w,\constr) 
    &= 
    \inf_{y\in\constr}\abs{\vec w-y} 
    \leq
    \inf_{y\in\constr}\left(\abs{\vec w-\vec y} + \abs{y-\vec y}\right)
    =
    \abs{\vec w-\vec y} + \inf_{y\in\constr}\abs{y-\vec y},\quad\forall\vec y\in\constr_n.
\end{align*}
Taking the infimum over $\vec y$ we get
\begin{align*}
    \dist(\vec w,\constr) 
    &\leq
    \inf_{\vec y\in\constr_n}
    \left(
    \abs{\vec w-\vec y} + \inf_{y\in\constr}\abs{y-\vec y}
    \right)
    \leq
    \inf_{\vec y\in\constr_n}
    \left(
    \abs{\vec w-\vec y} + \sup_{\vec y\in\constr_n}\inf_{y\in\constr}\abs{y-\vec y}
    \right)
    \\
    &\leq
    \dist(\vec w,\constr_n) + d_H(\constr,\constr_n).
\end{align*}
Hence, can we conclude
\begin{align*}
    \dist(\vec w,\constr_n) 
    &\geq 
    \dist(\vec w,\constr) - d_H(\constr,\constr_n) 
    > \res_n - \res_n
    = 0,
\end{align*}
which means $B'\cap\constr_n=\emptyset$ and therefore $B'\subset\domain_n\setminus\constr_n$.

Therefore, we can utilize that $\vec u_n$ satisfies comparison with cones from above (see \cref{thm:comp_graph_cones}) and infer that \labelcref{ineq:comparison_with_cone} holds for all vertices $\vec w \in\domain_n$ with $\vec d_n(\vec x_0,\vec w)\leq 2\nlscale(1-\phi(\gscale)/\gscale)-\gscale$.
We would like to choose $\vec w = \vec p_n^\nlscale(\vec x_0)$ and this choice is possible because of the following estimate.
Using \cref{lem:cone_upper} \revision{and abbreviating $\tilde C(\gscale):=\frac{4+2\sigma_\phi(\gscale)}{t_0}$} it holds
\begin{align}
    \vec d_n(\vec x_0,\vec p_n^\nlscale(\vec x_0))
    &\leq 
    \left(1 + \tilde C(\gscale) \frac{\res_n}{\gscale}\right)d_\domain(\vec x_0, \vec p^{\nlscale}(\vec x_0)) + 
    \tau_\eta\gscale 
    \notag\\
    &\leq
    \left(1 + \tilde C(\gscale) \frac{\res_n}{\gscale}\right)\nlscale + \tau_\eta\gscale 
    \notag\\
    &=
    \nlscale + \tilde C(\gscale) \frac{\res_n}{\gscale}\nlscale + \tau_\eta\gscale
    \notag
    \\
    &=
    \underbrace{2\nlscale\left(1-\frac{\phi(\gscale)}{\gscale}\right) - \gscale
    -\left(2\nlscale-2\nlscale\frac{\phi(\gscale)}{\gscale}-\gscale\right)}_{=0}
    \notag\\
    &\qquad +
    \nlscale + \tilde C(\gscale) \frac{\res_n}{\gscale}\nlscale + \tau_\eta\gscale
    \notag \\
    &= 
    2\nlscale\left(1-\frac{\phi(\gscale)}{\gscale}\right)-\gscale
    +2\nlscale\frac{\phi(\gscale)}{\gscale}
    -\nlscale 
    + (1+\tau_\eta)\gscale + \tilde C(\gscale)\frac{\res_n}{\gscale}\nlscale
    \notag\\
    &= 
    2\nlscale\left(1-\frac{\phi(\gscale)}{\gscale}\right)-\gscale
    +
    \nlscale
    \left(
    2\frac{\phi(\gscale)}{\gscale}
    -1 
    + (1+\tau_\eta)\frac{\gscale}{\nlscale} + \tilde C(\gscale)\frac{\res_n}{\gscale}\right)
    \notag\\
    \label{ineq:gdist_x0_peps}
    &\leq 2\nlscale\left(1-\frac{\phi(\gscale)}{\gscale}\right)-\gscale,
\end{align}
\revision{where we used $t_0\leq 1$, $2\leq\frac{4+2\sigma_\phi(\gscale)}{t_0}=\tilde C(\gscale)$, $\tau_\eta\leq 1$, and $\frac{\phi(\gscale)}{\gscale}\leq\sigma_\phi(\gscale)$ to obtain from \cref{ass:scaling}
\begin{align*}
    2\frac{\phi(\gscale)}{\gscale}
    -1
    +
    (1+\tau_\eta)\frac{\gscale}{\nlscale} + \tilde C(\gscale)\frac{\res_n}{\gscale}
    \leq 
    \tilde C(\gscale)\left(\sigma_\phi(\gscale) + \frac{\res_n}{\gscale}\right)-1+2\frac{\gscale}{\nlscale}
    \leq 0.
\end{align*}}
Before starting with the central estimate we use \cref{ineq:gdist_x0_peps,prop:lipschitz_discr,lem:cone_upper,ass:scaling} to compute
\begin{align}
    \abs{\vec u_n(\vec p_n^{2\nlscale}(\vec x_0)) - \vec u_n(\vec x_0)} 
    &\leq 
    \Lip_n(\vec g_n)\vec d_n(\vec p_n^{2\nlscale}(\vec x_0),\vec x_0)) 
    \notag
    \\
    &\leq 
    \Lip_n(\vec g_n)\left(\tilde C(\gscale)\,d_\domain(\vec p_n^{2\nlscale}(\vec x_0),\vec x_0) + \tau_\eta\gscale\right)
    \notag
    \\
    &\leq
    \Lip_n(\vec g_n)\left(2\tilde C(\gscale)\nlscale + \tau_\eta\gscale\right)
    \notag
    \\
    &\leq
    \Lip_n(\vec g_n)\overline C\nlscale
    \label{ineq:lipschitz_bound}
\end{align}
\revision{for a constant $\overline C>0$.}
Using \labelcref{ineq:comparison_with_cone} with $\vec w=\vec p^\nlscale(\vec x_0)$, \cref{ineq:lipschitz_bound,lem:cone_upper,ass:scaling} we can compute
\begin{align*}
2\vec{u}_n(\vec{p}_n^\nlscale(\vec x_0)) 
&\leq 
2\vec{u}_n(\vec{x}_0) + 
\frac{\vec{u}_n(\vec p_n^{2\nlscale}(\vec x_0)) - \vec{u}_n(\vec{x}_0)}{\nlscale\left(1-\frac{\phi(\gscale)}{\gscale}\right) - \gscale/2} 
\vec d_n(\vec x_0, \vec{p}_n^\nlscale(\vec x_0))\\
&\leq 
2\vec{u}_n(\vec{x}_0)+
\frac{\vec{u}_n(\vec p_n^{2\nlscale}(\vec x_0)) - \vec{u}_n(\vec{x}_0)}{\nlscale\left(1-\frac{\phi(\gscale)}{\gscale}\right) - \gscale/2} 
\left(\left(1+\tilde C(\gscale)\frac{\res_n}{\gscale}\right)\nlscale+\tau_\eta\gscale\right)\\
&= 
2\vec{u}_n(\vec{x}_0)+
\left({\vec{u}_n(\vec p_n^{2\nlscale}(\vec x_0)) - \vec{u}_n(\vec{x}_0)}\right)
\left(\left(1+\tilde C(\gscale)\frac{\res_n}{\gscale}\right)\frac{\nlscale}{\nlscale\left(1-\frac{\phi(\gscale)}{\gscale}\right) - \gscale/2}+\frac{\tau_\eta\gscale}{\nlscale\left(1-\frac{\phi(\gscale)}{\gscale}\right) - \gscale/2}\right)\\
&= 
\vec{u}_n(\vec{x}_0)+\vec{u}_n(\vec p_n^{2\nlscale}(\vec x_0))+
\\
&\qquad
\left({\vec{u}_n(\vec p_n^{2\nlscale}(\vec x_0)) - \vec{u}_n(\vec{x}_0)}\right)
\left(\tilde C(\gscale)\frac{\res_n}{\gscale}+\frac{\nlscale\frac{\phi(\gscale)}{\gscale}+\gscale/2}{\nlscale(1-\frac{\phi(\gscale)}{\gscale})-\gscale/2}\left(1+\tilde C(\gscale)\frac{\res_n}{\gscale}\right) + \frac{\tau_\eta\gscale}{\nlscale(1-\frac{\phi(\gscale)}{\gscale})-\gscale/2}\right)
\\
&\leq
\vec{u}_n(\vec{x}_0)+\vec{u}_n(\vec p_n^{2\nlscale}(\vec x_0))+ 
\Lip_n(\vec g_n)\overline C\nlscale\Bigg(
\tilde C(\gscale)\frac{\res_n}{\gscale} 
+ 
\left(\frac{\phi(\gscale)}{\gscale}
+
\frac{\gscale}{2\nlscale}
\right)
\left(
1+2\frac{\phi(\gscale)}{\gscale}
+\frac{\gscale}{\nlscale}
\right)
\left(1+\tilde C(\gscale)\frac{\res_n}{\gscale}\right)
\\
&\qquad
+ 
\tau_\eta
\frac{\gscale}{\nlscale}
\left(
1+2\frac{\phi(\gscale)}{\gscale}
+\frac{\gscale}{\nlscale}
\right)
\Bigg)
\\
&\leq
\vec{u}_n(\vec{x}_0)+\vec{u}_n(\vec p_n^{2\nlscale}(\vec x_0))+ 
\Lip_n(\vec g_n)C\nlscale\left(
\frac{\res_n}{\gscale} + 
\frac{\gscale}{\nlscale}
+
\frac{\phi(\gscale)}{\gscale}
\right)
\\
&=
\vec{u}_n(\vec{x}_0)+\vec{u}_n(\vec p_n^{2\nlscale}(\vec x_0))+ 
\Lip_n(\vec g_n)C\left(\frac{\res_n\nlscale}{\gscale}+\gscale
+\nlscale\frac{\phi(\gscale)}{\gscale}
\right),
\end{align*}
\revision{for a constant $C>0$.}
Plugging the estimate above into \labelcref{ineq:estimate_nonl_inf_lapl} and dividing by $\nlscale^2$ we get
\begin{align*}
    -\Delta_\infty^\nlscale u_n^\nlscale(\vec x_0) \leq \Lip_n(\vec g_n)C\left(\frac{\res_n}{\gscale\nlscale} + \frac{\gscale}{\nlscale^2}
    +\frac{\phi(\gscale)}{\nlscale\gscale}
    \right).
\end{align*}
\end{proof}
\begin{remark}
Almost the whole proof of \cref{thm:discrete_NL} works under the assumption that $\vec u_n$ merely satisfies comparison with cones from above / below in which case one could only prove \labelcref{ineq:bound_subharmonic} / \labelcref{ineq:bound_superharmonic}.
However, in order to have the discrete Lipschitz estimate from \cref{prop:lipschitz_discr}, we need $\vec u_n$ to satisfy both comparison properties.
\end{remark}
\begin{remark}\label{rem:singular_kernel}
We were hoping that in the case of a singular kernel, where $\tau_\eta=0$ and hence the convergence rate of graph distance functions is better according to \cref{lem:cone_upper}, the consistency error in \labelcref{ineq:consist_bounds} might be better.
Closely following the proof, one observes that almost all order $\gscale$ contributions are multiplied with $\tau_\eta$.
Only in the definition of the graph set $B$, we have to use the lower bound which features $-\gscale$.
This then leads to the term $\gscale\nlscale^{-2}$ which, however, does not contribute for small graph length scales $\gscale$.
\end{remark}

\subsubsection{Convergence Rates}

Now we are in the position to prove an important result on our way to convergence rates which can be interpreted as an approximate maximum principle.
\begin{proposition}\label{prop:approx_max_princ}
Let $u\in C(\closure\domain)$ satisfy CDF (cf. \labelcref{eq:continuum_problem}) and define $\tilde\nlscale_n:=\nlscale+\frac{3}{2}\res_n$.
Under the conditions of \cref{thm:discrete_NL} there exists a constant $C>0$ such that
\begin{subequations}
\begin{align}
    \label{ineq:cvgc_subharmonic}
    \sup_{\domain^{\tilde\nlscale_n}_\constr}(u_n^\nlscale-T_\nlscale u) 
    &\leq
    \sup_{\domain_\constr^{\tilde\nlscale_n}\setminus\domain_\constr^{2\tilde\nlscale_n}}(u_n^\nlscale-T_\nlscale u) +
    \sqrt[3]{\Lip_n(\vec g_n)C
    \left(\frac{\res_n}{\gscale\nlscale} +
    \frac{\gscale}{\nlscale^2}+
    \frac{\phi(\gscale)}{\gscale\nlscale}
    \right)},
    \\
    \label{ineq:cvgc_superharmonic}
    \sup_{\domain^{\tilde\nlscale_n}_\constr}(T^\nlscale u-(u_n)_\nlscale) 
    &\leq
    \sup_{\domain_\constr^{\tilde\nlscale_n}\setminus\domain_\constr^{2\tilde\nlscale_n}}(T^\nlscale u-(u_n)_\nlscale) + \sqrt[3]{\Lip_n(\vec g_n)C
    \left(\frac{\res_n}{\gscale\nlscale} + 
    \frac{\gscale}{\nlscale^2} +
    \frac{\phi(\gscale)}{\gscale\nlscale}
    \right)}.
\end{align}
\end{subequations}
If additionally it holds that $\alpha:=\inf_{\closure\domain}\abs{\grad u}>0$, then we even have
\begin{subequations}\label{ineq:cvgc_pos_grad}
\begin{align}
    \sup_{\domain^{\tilde\nlscale_n}_\constr}(u_n^\nlscale-T_\nlscale u) 
    &\leq
    \sup_{\domain_\constr^{\tilde\nlscale_n}\setminus\domain_\constr^{2\tilde\nlscale_n}}(u_n^\nlscale-T_\nlscale u) + 
    {\Lip_n(\vec g_n)C
    \left(
    \frac{\res_n}{\gscale\nlscale} + 
    \frac{\gscale}{\nlscale^2}+
    \frac{\phi(\gscale)}{\gscale\nlscale}
    \right)},
    \\
    \sup_{\domain^{\tilde\nlscale_n}_\constr}(T^\nlscale u-(u_n)_\nlscale) 
    &\leq
    \sup_{\domain_\constr^{\tilde\nlscale_n}\setminus\domain_\constr^{2\tilde\nlscale_n}}(T^\nlscale u-(u_n)_\nlscale) + {\Lip_n(\vec g_n)C
    \left(
    \frac{\res_n}{\gscale\nlscale} + 
    \frac{\gscale}{\nlscale^2} +
    \frac{\phi(\gscale)}{\gscale\nlscale}
    \right)}.
\end{align}
\end{subequations}
\end{proposition}
\begin{proof}
We only prove \labelcref{ineq:cvgc_subharmonic}, the second inequality is obtained analogously.
By \cref{thm:discrete_NL} we know that there is a constant $C>0$ such that
\begin{align*}
    -\Delta_\infty^\nlscale u_n^\nlscale \leq \Lip_n(\vec g_n) C 
    \left|
    \frac{\res_n}{\gscale\nlscale} + 
    \frac{\gscale}{\nlscale^2} +
    \frac{\phi(\gscale)}{\gscale\nlscale}
    \right| =: C_{n,\nlscale} \quad\text{in }\domain_\constr^{2\tilde\nlscale_n}.
\end{align*}
By \cref{lem:nonloc2loc} we know that $-\Delta_\infty^\nlscale T_\nlscale u \geq 0$ in $\domain_\constr^{2\nlscale}$ and hence also in $\domain_\constr^{2\tilde\nlscale_n}$.
Furthermore, since $u$ is infinity harmonic, it is bounded and hence also $T_\nlscale u$ is bounded.
Applying \cref{lem:perturb_pos_grad} with $\delta:=(C_{n,\nlscale})^{\frac13}$ we can hence choose $v:\domain_\constr^\nlscale\to\R$ such that
\begin{align*}
    -\Delta_\infty^\nlscale v \geq 0,\quad S_\nlscale^- v \geq (C_{n,\nlscale})^{\frac13},\quad T_\nlscale u \leq v \leq T_\nlscale u + 2 (C_{n,\nlscale})^{\frac13} \dist(\cdot, \domain_\constr^\nlscale\setminus\domain_\constr^{2\nlscale})\quad\text{in }\domain_\constr^{2\nlscale}.
\end{align*}
Since $u\in C(\closure\domain)$
is a bounded function, the same holds for $v$.
Define $w:=v-(C_{n,\nlscale})^{\frac13} v^2$ which according to \cref{lem:perturb_strict} and its definition satisfies
\begin{align*}
    -\Delta_\infty^\nlscale w \geq C_{n,\nlscale} \quad\text{in }\domain_\constr^{2\nlscale},\quad \norm{w-T_\nlscale u}_{L^\infty(\domain_\constr^\nlscale)} \leq c(C_{n,\nlscale})^{\frac13}.
\end{align*}
Then it holds $-\Delta_\infty^\nlscale u_n^\nlscale\leq C_{n,\nlscale}\leq -\Delta_\infty^\nlscale w$ in $\domain_\constr^{2\tilde\nlscale_n}$ and applying \cref{lem:nonlocal_max_princ} with $\nlscale$ replaced by $\tilde\nlscale_n$ we can compute 
\begin{align*}
    \sup_{\domain_\constr^{\tilde\nlscale_n}}(u_n^\nlscale-T_\nlscale u) 
    &\leq
    \sup_{\domain_\constr^{\tilde\nlscale_n}}(u_n^\nlscale-w) + c(C_{n,\nlscale})^{\frac13}\\
    &\leq
    \sup_{\domain_\constr^{\tilde\nlscale_n}\setminus\domain_\constr^{2\tilde\nlscale_n}}(u_n^\nlscale-w) + c(C_{n,\nlscale})^{\frac13}\\
    &\leq
    \sup_{\domain_\constr^{\tilde\nlscale_n}\setminus\domain_\constr^{2\tilde\nlscale_n}}(u_n^\nlscale-T_\nlscale u) + 2c(C_{n,\nlscale})^{\frac13}.
\end{align*}
In the case that $\alpha=\inf_{\closure\domain}\abs{\grad u}>0$ we \revision{do not have to apply \cref{lem:perturb_pos_grad} to construct a perturbation $v$ with positive slope. Instead, we} define $w:=u-\frac{C_{n,\nlscale}}{\alpha^2}u^2$ which according to \cref{lem:perturb_strict} satisfies
\begin{align*}
    -\Delta_\infty^\nlscale w \geq C_{n,\nlscale} \quad\text{in }\domain_\constr^{2\nlscale},\quad \norm{w-T_\nlscale u}_{L^\infty(\domain_\constr^\nlscale)} \leq \frac{c}{\alpha^2} C_{n,\nlscale}.
\end{align*}
From here the proof continues as before.
\end{proof}
For getting rid of the $\eps$-extension operators and boundary terms in \cref{prop:approx_max_princ}, we have to use the lemmas from \cref{sec:lipschitz_lemmas}.
Basically, one uses Lipschitzness of $u$ (and $u_n^\nlscale$ in an appropriate sense) in order to transfer the statement of \cref{prop:approx_max_princ} to $u$ and $\vec u_n$, estimating the boundary term, and extending the uniform estimate to the whole of $\closure\domain$.
All this comes with an additional additive error term of order $\nlscale>0$. 
We are now ready to prove the main theorem of this article, \cref{thm:general_convergence_result}.
\begin{proof}[Proof of \cref{thm:general_convergence_result}]
First we notice that thanks to \cite[Lemma 2.3]{roith2021continuum} we have
\begin{align*}
    \sup_{\domain_n}\abs{u-\vec u_n} = \sup_{\closure\domain}\abs{u-u_n}.
\end{align*}
Again we abbreviate $\tilde\nlscale_n := \nlscale + \frac{3}{2}\res_n$.
Using \cref{lem:ball_to_func_cont,lem:ball_to_func_disc,lem:bdry_term,prop:approx_max_princ} we get
\begin{align*}
    \sup_{\domain_\constr^{\tilde\nlscale_n}}(u_n - u)
    &\leq 
    \sup_{\domain_\constr^{\tilde\nlscale_n}}(u_n - u_n^\nlscale) 
    + 
    \sup_{\domain_\constr^{\tilde\nlscale_n}}(u_n^\nlscale - T_\nlscale u)
    +
    \sup_{\domain_\constr^{\tilde\nlscale_n}}(T_\nlscale u - u)
    \\
    &\leq 
    C\Lip_n(\vec u_n)\nlscale + C\Lip_\domain(u)\nlscale +
    \sup_{\domain_\constr^{\tilde\nlscale_n}\setminus\domain_\constr^{2\tilde\nlscale_n}}(u_n^\nlscale-T_\nlscale u) +
    \sqrt[3]{\Lip_n(\vec g_n)C
    \left(\frac{\res_n}{\gscale\nlscale} +
    \frac{\gscale}{\nlscale^2}+
    \frac{\phi(\gscale)}{\gscale\nlscale}
    \right)}
    \\
    &\leq 
    C\Lip_n(\vec u_n)\nlscale + C\Lip_\domain(u)\nlscale +
    C\nlscale +
    \sqrt[3]{\Lip_n(\vec g_n)C
    \left(\frac{\res_n}{\gscale\nlscale} +
    \frac{\gscale}{\nlscale^2}+
    \frac{\phi(\gscale)}{\gscale\nlscale}
    \right)}.
\end{align*} 
By \cref{prop:lipschitz_discr} we know that $\Lip_n(\vec u_n)=\Lip_n(\vec g_n)$ and by \cref{ass:data} it holds $\sup_{n\in\N}\Lip_n(\vec g_n)\leq C$.
The term $\sup_{\domain_\constr^{\tilde\nlscale_n}}(u-u_n)$ can be estimated analogously.
Hence, we get the inequality
\begin{align*}
    \sup_{\domain_\constr^{\tilde\nlscale_n}}\abs{u-u_n} 
    \leq 
    C\left(\nlscale + 
    \sqrt[3]{\frac{\res_n}{\gscale\nlscale} +
    \frac{\gscale}{\nlscale^2}+
    \frac{\phi(\gscale)}{\gscale\nlscale}
    }
    \right).
\end{align*}
Note that it holds $\abs{u(x)-u(y)}\leq\Lip_\domain(u)d_\domain(x,y)$ for all $x,y\in\closure\domain$.
Using this together with \cref{lem:lipschitz_un,ass:scaling} we get for all $x\in\closure\domain\setminus\domain_\constr^{\tilde\nlscale_n}$ that
\begin{align*}
    \abs{u(x)-u_n(x)} 
    &\leq
    \abs{u(x)-u(\tilde x)}
    +
    \abs{u(\tilde x)-u_n(\tilde x)}
    +
    \abs{u_n(\tilde x)-u_n(x)}
    \\
    &\leq 
    \Lip_\domain(u)\tilde\nlscale_n
    +
    C\left(\nlscale + 
    \sqrt[3]{\frac{\res_n}{\gscale\nlscale} +
    \frac{\gscale}{\nlscale^2}+
    \frac{\phi(\gscale)}{\gscale\nlscale}
    }
    \right)
    +
    C\Lip_n(\vec u_n) \tilde\nlscale_n,
\end{align*}
where $\tilde x\in\domain_\constr^{\tilde\nlscale}$ is such that $d_\domain(x,\tilde x)\leq \tilde\nlscale_n$.
Using also that thanks to \cref{ass:scaling} it holds $\tilde\nlscale_n=\nlscale(1+\frac{3}{2}\res_n)\leq C\nlscale$, we can pass from $\sup_{\domain_\constr^{\tilde\nlscale}}\abs{u-u_n}$ to $\sup_{\closure\domain}\abs{u-u_n}$ at the cost of another term of order $\nlscale$ which proves the first assertion of the theorem.
In the case that $\inf_{\closure\domain}\abs{\nabla u}>0$ one simply uses the better estimate in \cref{prop:approx_max_princ} without $\sqrt[3]{\cdot}$.
\end{proof}

\section{Numerical Results}
\label{sec:numerics}

\begin{figure}[t!]
    \def\width{0.3\textwidth}
    \centering
    \includegraphics[width=\textwidth,trim=1.8cm 1.1cm 1.8cm 0.8cm,clip]{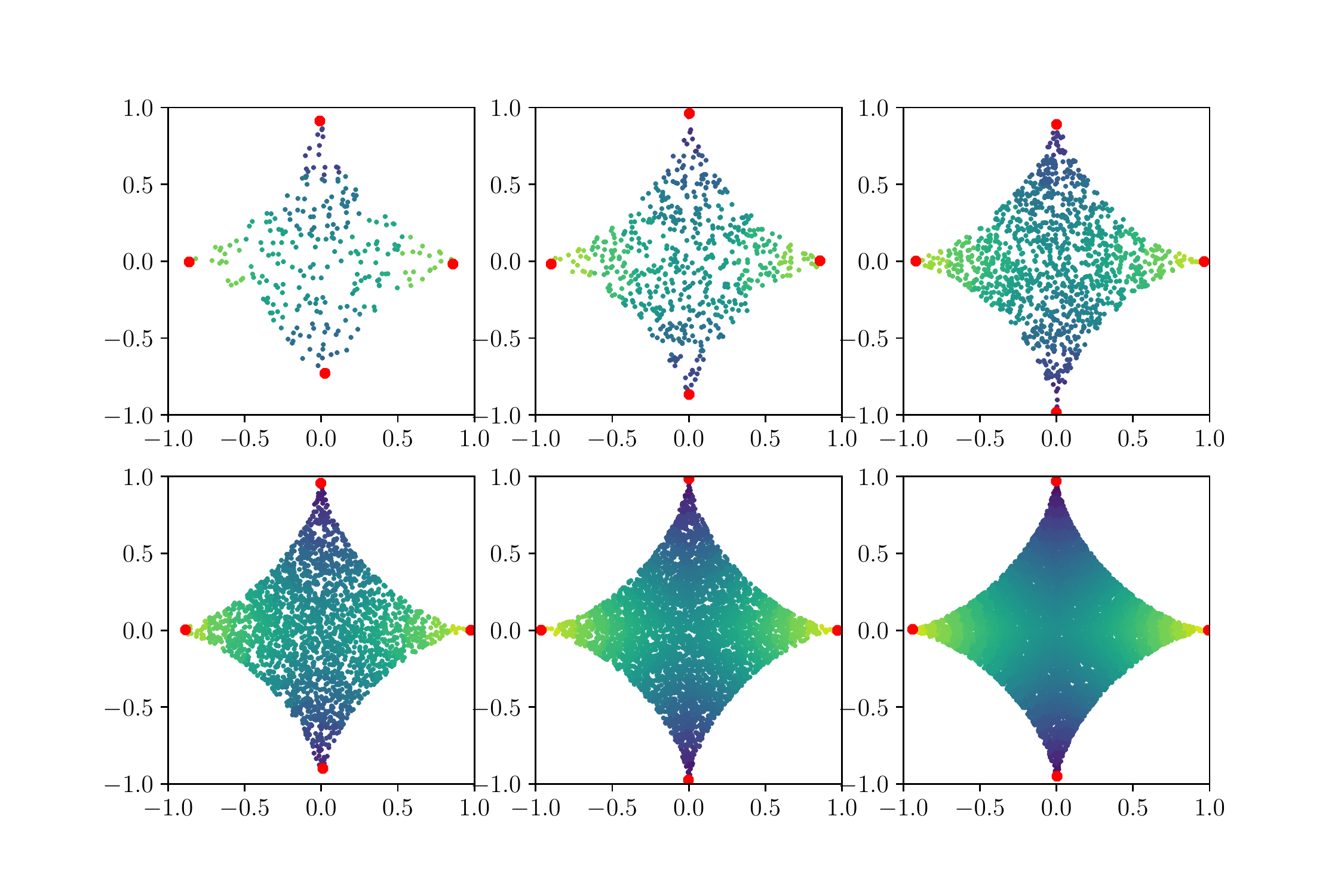}
    \caption{Solutions of the Lipschitz learning problem for different graph resolutions on a non-convex domain with mixed Neumann and Dirichlet conditions.
    The four labelled points in $\constr_n^\mathrm{cp}$ are marked in red.}
    \label{fig:nonconvex}
    \vspace{1cm}
    \centering
    \includegraphics[width=0.49\textwidth,trim=.4cm .4cm .4cm .4cm,clip]{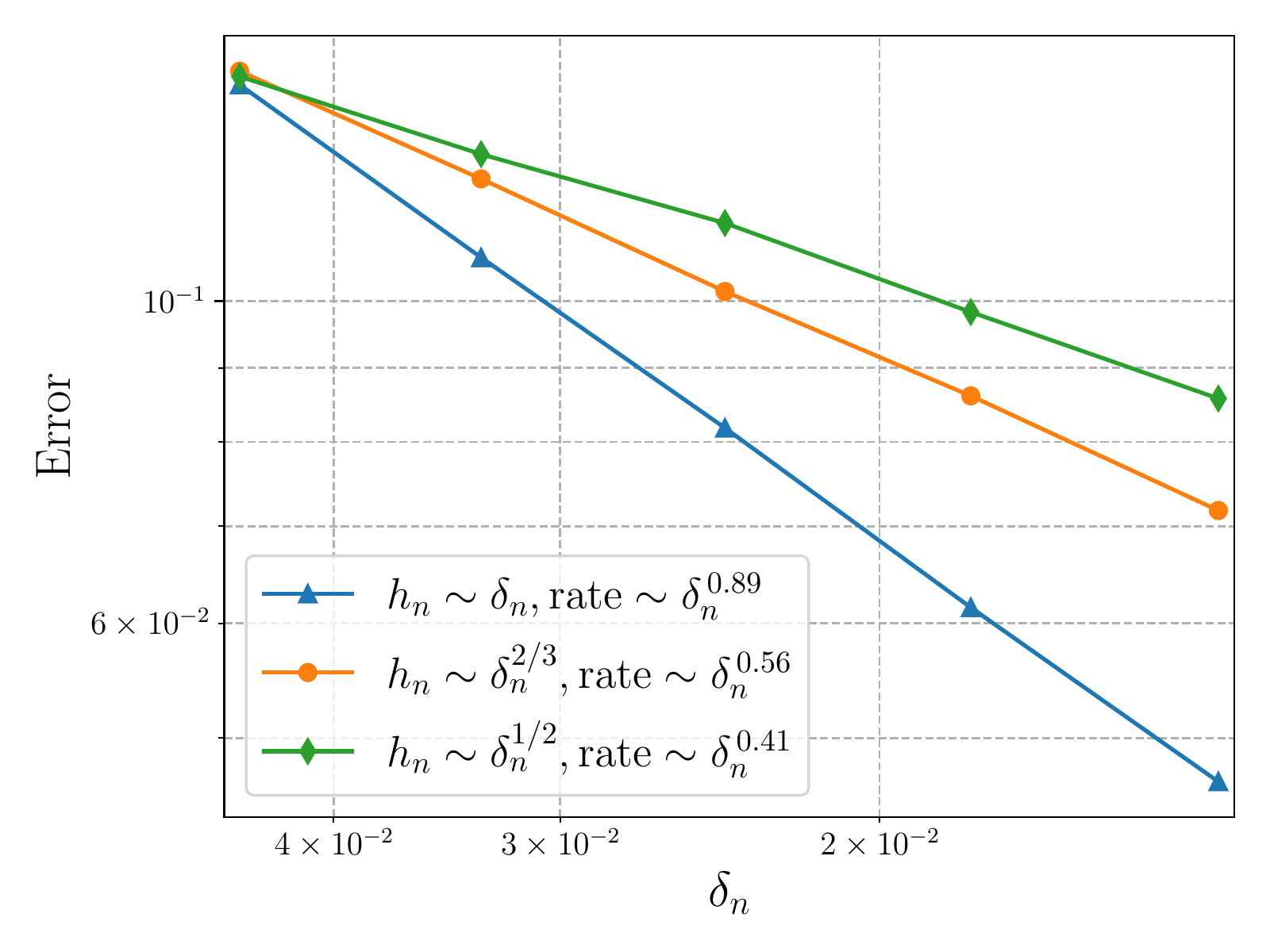}%
    \hfill%
    \includegraphics[width=0.49\textwidth,trim=.4cm .4cm .4cm .4cm,clip]{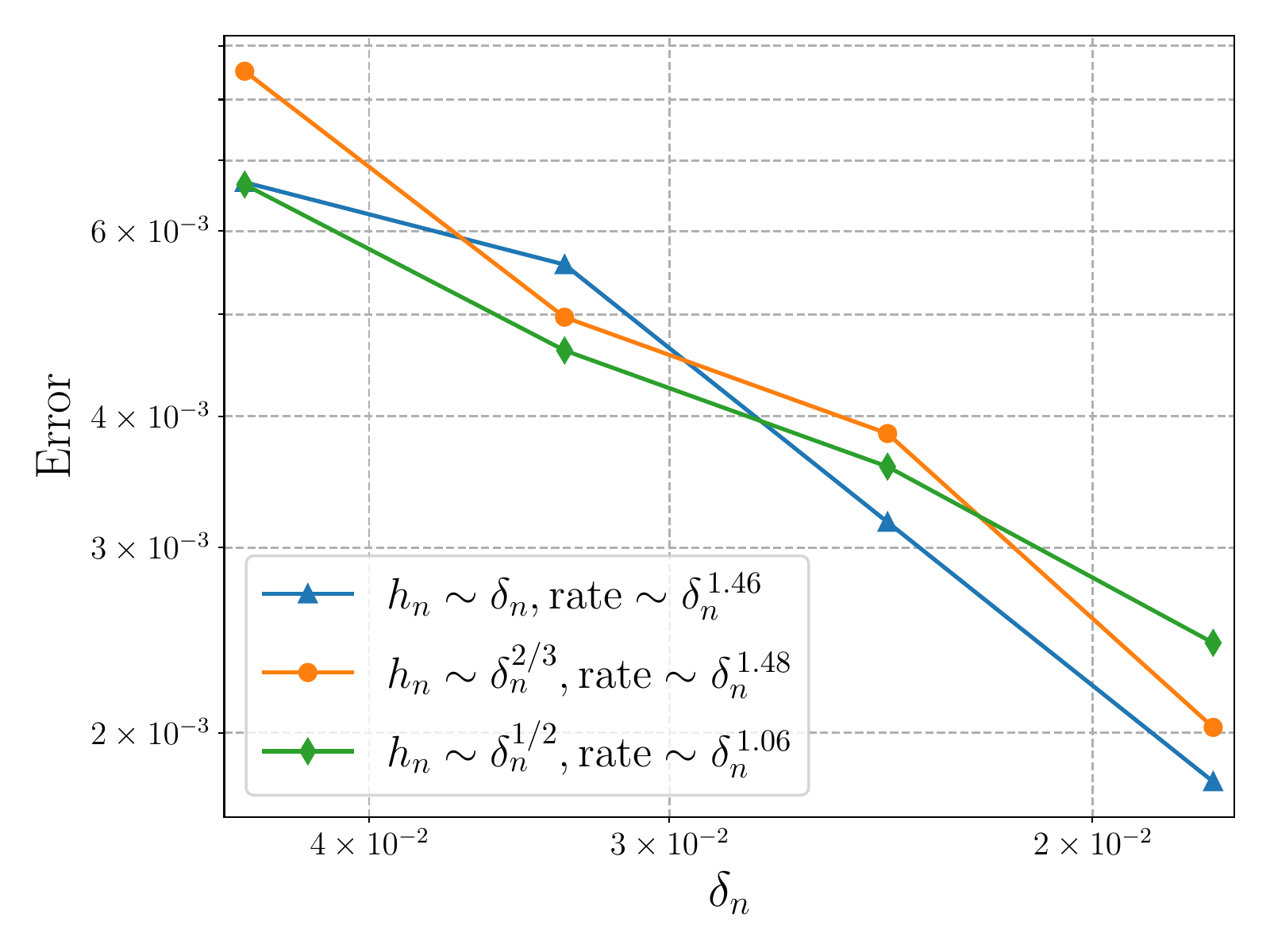}
    \caption{Empirical convergence rates of the solutions in \cref{fig:nonconvex} for unit weights (\textbf{left}) and singular weights (\textbf{right}) and different scalings with the constraint set $\constr_n^\mathrm{cp}$. The rates are specified in the legend.
    Note that our theory does not prove convergence for the blue scaling.}
    \label{fig:rates_nonconvex}
\end{figure}

\begin{figure}[t!]
    \def\width{0.3\textwidth}
    \centering
    \includegraphics[width=\textwidth,trim=1.8cm 1.1cm 1.8cm 0.8cm,clip]{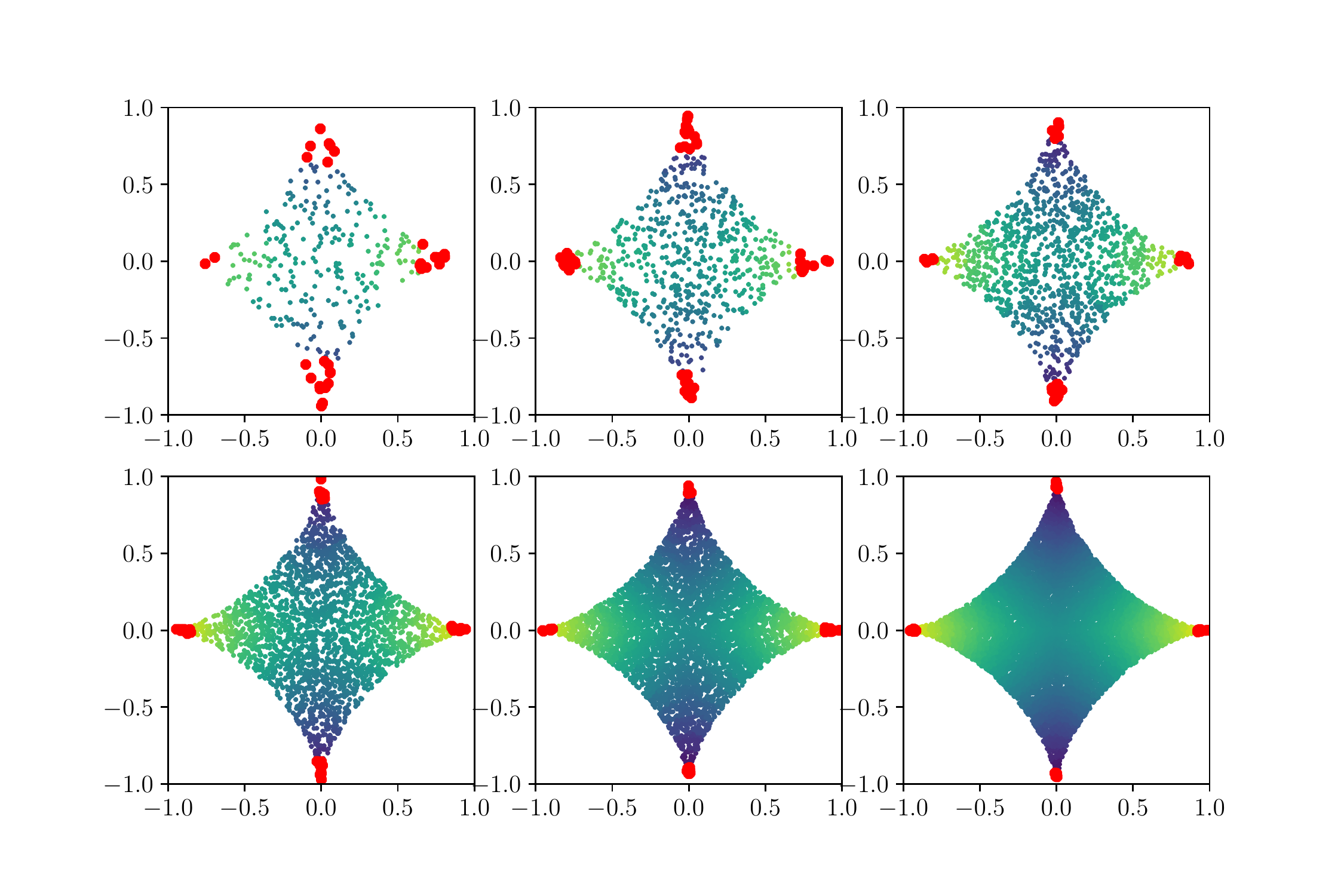}
    \caption{Solutions of the Lipschitz learning problem for different graph resolutions on a non-convex domain with mixed Neumann and Dirichlet conditions.
    The labelled points in $\constr_n^\mathrm{dil}$ are marked in red.}
    \label{fig:nonconvex_dilate}
    \vspace{1cm}
    \centering
    \includegraphics[width=0.49\textwidth,trim=.4cm .4cm .4cm .4cm,clip]{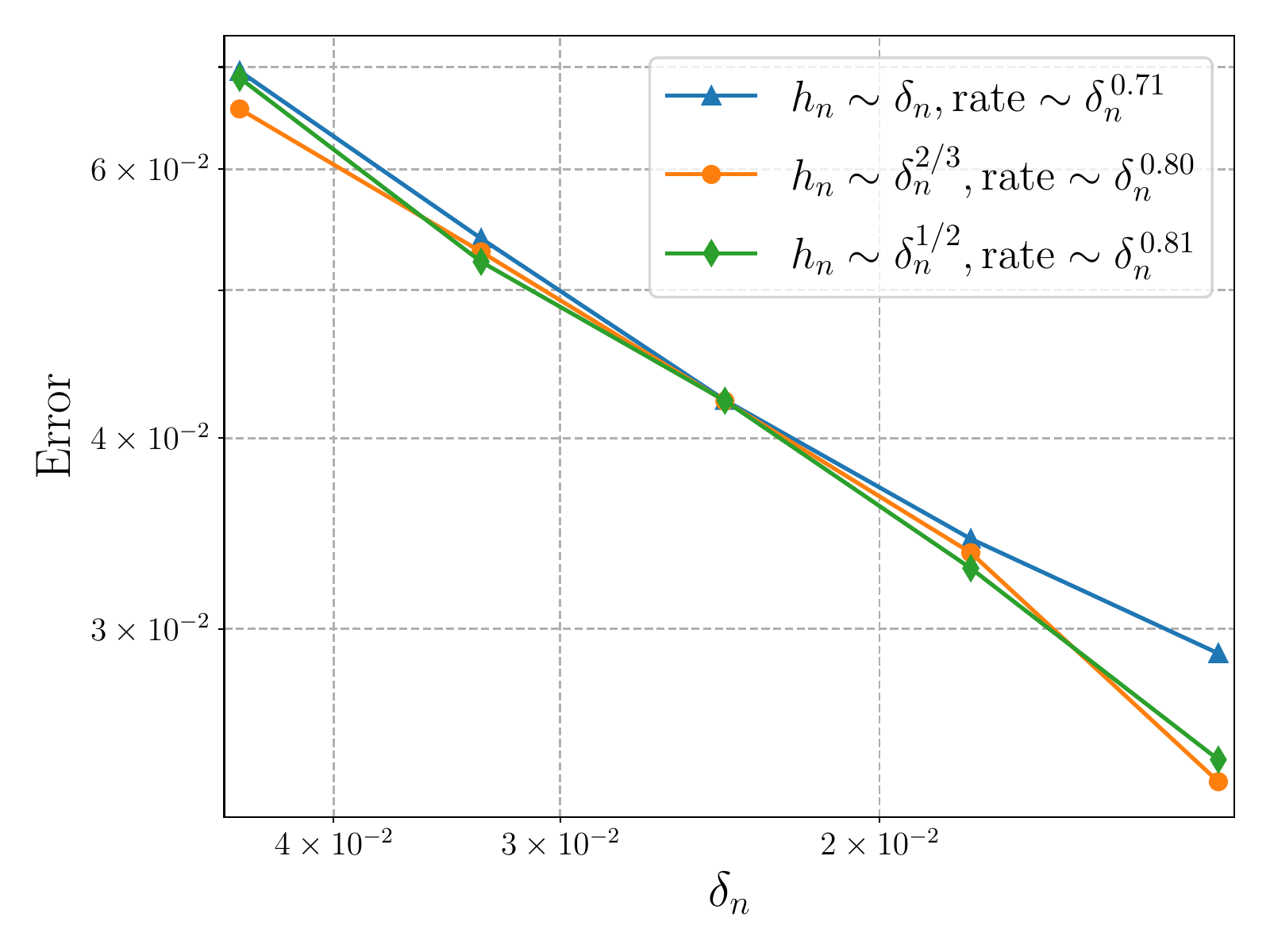}%
    \hfill%
    \includegraphics[width=0.49\textwidth,trim=.4cm .4cm .4cm .4cm,clip]{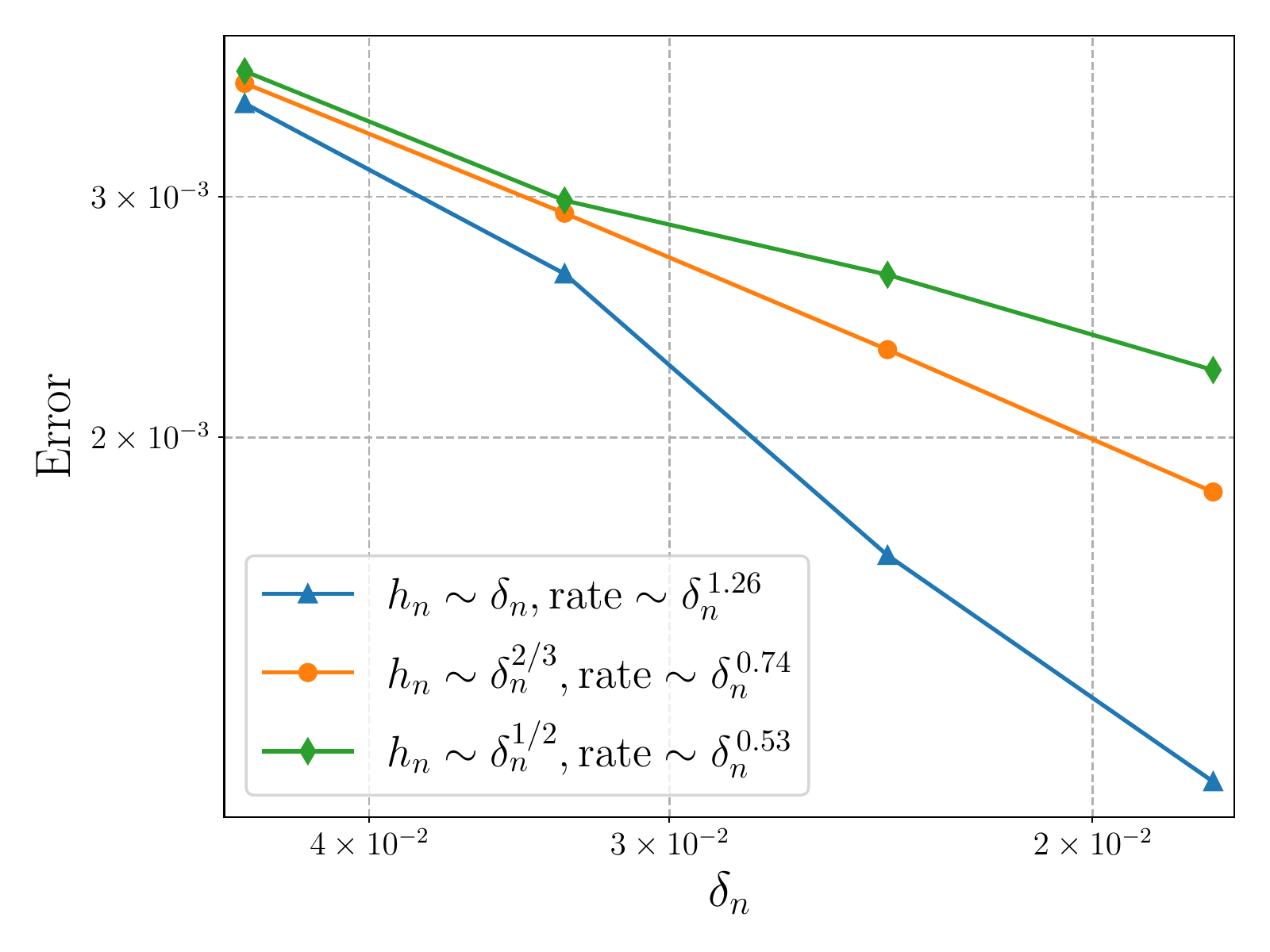}
    \caption{Empirical convergence rates of the solutions in \cref{fig:nonconvex_dilate} for unit weights (\textbf{left}) and singular weights (\textbf{right}) and different scalings with the constraint set $\constr_n^\mathrm{dil}$. The rates are specified in the legend.
    Note that our theory does not prove convergence for the blue scaling.}
    \label{fig:rates_nonconvex_dilate}
\end{figure}

In this section we present a brief numerical study of convergence rates for the continuum limit of Lipschitz learning.\footnote{The code for our experiments is available online: \url{https://github.com/jwcalder/LipschitzLearningRates}.} In order to test convergence rates, we work with the Aronsson function
\[u(x_1,x_2)=\abs{x_1}^{4/3}-\abs{x_2}^{4/3},\]
which is infinity harmonic on the plane $\R^2$.  
We \revision{work with the} non-convex and non-smooth domain \labelcref{eq:neumann_star} which satisfies \cref{ass:geodesic_euclidean_2} \revision{according to \cref{prop:neumann_star}} and on which the Aronsson function is an AMLE \labelcref{eq:AMLE}. 
\revision{For convenience, we remind the reader that this domain was defined as follows}
\begin{align*}
    \domain:=\left\{x\in[0,1]^2 \st \abs{x_1}^{2/3}+\abs{x_2}^{2/3}\leq 1\right\},
\end{align*}
and coincides with the non-convex unit ball around zero of the $\ell_{2/3}$ metric space.

As constraint set we choose the four points $\constr=\{(\pm 1,0)\}\cup\{(0,\pm 1)\}$. This domain was chosen so that the Aronsson function $u$ satisfies the homogeneous Neumann condition $\partial u/\partial \nu = 0$ on $\partial \domain \setminus \constr$. The point cloud $\domain_n$ is generated by drawing \emph{i.i.d.} samples from $\domain$.
For the discrete constraint set we consider two different choices.
The first one consists of the four closest graph points to $\constr$,
\begin{align*}
    \constr_n^\mathrm{cp} = \bigcup_{x\in\constr}\argmin_{\vec x\in\domain_n}\abs{\vec x-x},
\end{align*}
and the second one is given using a dilated boundary:
\begin{align*}
    \constr_n^\mathrm{dil} = \{\vec x\in\domain_n \st \dist(\vec x,\constr)\leq \gscale_n\}.
\end{align*}
\cref{fig:nonconvex} shows the solutions of the Lipschitz learning problem for different numbers $n\in\N$ of vertices, using the constraint set $\constr^\mathrm{cp}_n$.
Furthermore, in \cref{fig:rates_nonconvex} we show a convergence study where we compute empirical rates\footnote{We computed the rates by fitting a linear function to the log-log plots in \cref{fig:rates_nonconvex}.} of convergence for the different bandwidths and kernels.
We consider two different kernels: the constant one with $\eta(t)=1$ and the singular one with $\eta(t)=t^{-1}$.
Furthermore, we consider three different scalings of the graph bandwidth $\gscale_n$ which we choose as $\gscale_n\sim\{\delta_n,\delta_n^{2/3},\delta_n^{1/2}\}$. For the first scaling our theory does not even prove convergence.
This is due to the fact that the angular error in the point cloud should go to zero which manifests in $\gscale_n\gg\res_n$. 
Still, for a randomly sampled point cloud one can hope for a homogenization effect and still obtain convergence.
The other two scalings are covered by our theory, whereby $\gscale_n\sim\res_n^{2/3}$ falls into the small length scale regime (cf. \cref{cor:small_scale}) and $\gscale_n\sim\res_n^{1/2}$ falls into the large one with less smooth boundary (cf. \cref{cor:less_smooth}).

For the constant kernel $\eta(t)=1$ we ran experiments with $n=2^{12}$ up to $n=2^{16}$ points and averaged the errors over 100 different random samples of $\domain_n$. For the singular kernel $\eta(t)=t^{-1}$, the numerical solver is much slower and we ran experiments only up to $n=2^{15}$ and averaged errors over 20 trials. We observe in \cref{fig:rates_nonconvex} that all scalings appear to be convergent. Interestingly, the smallest bandwidth $\gscale_n\sim\res_n$, where convergence cannot be expected in general, exhibits the best rates and smallest errors.
Notably all empirical rates are better than the ones from \cref{cor:small_scale,cor:less_smooth}. 
This does not necessarily mean that our results are not sharp since for the Aronsson solution one can typically prove better rates, see the discussion in \cite{smart2010infinity}.
At the same time, it constitutes the only non-trivial (i.e., not a cone) explicit solution of the AMLE problem which is why it is a popular test case.
We also observe in \cref{fig:rates_nonconvex} that the rates for the singular kernel $\eta(t)=t^{-1}$ are better than the ones for unit weights which suggests that there might be a way to lift the better approximation of the distance functions (cf. \cref{lem:cone_upper}) to a better approximation of the (non-local) infinity Laplacian (cf. \cref{rem:singular_kernel}).
Also the magnitudes of the errors are much smaller for the singular kernel than for the constant one.

In \cref{fig:nonconvex_dilate,fig:rates_nonconvex_dilate} we show the results using the dilated constraint set $\constr_n^\mathrm{dil}$. 
For the singular kernel we see that the smallest graph length scales again yield the best rates, while for the constant kernel $\eta(t)=1$, all length scales exhibit similar convergence rates.
We can partially attribute this to the fact that the constraint set $\constr_n^\mathrm{cp}$ leads to sharp cusps which is resolved better using the smallest bandwidth whereas the dilated boundary in $\constr_n^\mathrm{dil}$ generates smoother domains where the larger length scales are well-suited for approximating the infinity Laplacian.

\section*{Acknowledgments}
The authors thank the Institute for Mathematics and its Applications (IMA) where this collaboration started at a workshop on ``Theory and Applications in Graph-Based Learning'' in Fall 2020. 
Part of this work was also done while the authors were visiting the Simons Institute for the Theory of Computing to participate in the program ``Geometric Methods in Optimization and Sampling'' during the Fall of 2021. 
LB acknowledges funding by the Deutsche Forschungsgemeinschaft (DFG, German Research Foundation) under Germany's Excellence Strategy - GZ 2047/1, Projekt-ID 390685813. 
JC acknowledges funding from NSF grant DMS:1944925, the Alfred P.~Sloan foundation, and a McKnight Presidential Fellowship. 
TR acknowledges support by the German Ministry of Science and Technology (BMBF) under grant agreement No. 05M2020 (DELETO).

\printbibliography

\end{document}